\documentclass[12pt]{amsart}
\usepackage{amsfonts}
\usepackage{bbm}
\usepackage{mathrsfs}
\usepackage{amssymb}
\usepackage{color}

\numberwithin{equation}{section}

\DeclareMathOperator{\supp}{supp}

\begin{document}

\baselineskip 16.1pt \hfuzz=6pt

\theoremstyle{plain}
\newtheorem{theorem}{Theorem}[section]
\newtheorem*{main}{Main Theorem}
\newtheorem{prop}[theorem]{Proposition}
\newtheorem{lemma}[theorem]{Lemma}
\newtheorem{corollary}[theorem]{Corollary}
\newtheorem*{corollaryA}{Corollary A}
\newtheorem*{corollaryB}{Corollary B}
\newtheorem*{corollaryC}{Corollary C}
\newtheorem*{coro}{Corollary}
\newtheorem{example}[theorem]{Example}
\renewcommand{\theequation}
{\thesection.\arabic{equation}}

\theoremstyle{definition}
\newtheorem{definition}[theorem]{Definition}
\newtheorem{remark}[theorem]{Remark}

\allowdisplaybreaks

\newcommand{\XX}{X}
\newcommand{\X}{\widetilde{X}}
\newcommand{\GG}{\mathop G \limits^{    \circ}}
\newcommand{\GGs}{{\mathop G\limits^{\circ}}_{\eta}}
\newcommand{\xoneandxtwo}{X_1\times\mathcal X_2}

\newcommand{\GGp}{{\mathop G\limits^{\circ}}_{\eta_1,\eta_2}}
\newcommand{\GGpp}{{\mathop G\limits^{\circ}}_{\theta_1,\theta_2}}

\newcommand{\e}{\varepsilon}
\newcommand{\bmo}{{\rm BMO}}
\newcommand{\vmo}{{\rm VMO}}
\newcommand{\cmo}{{\rm CMO}}

\newcommand{\Hatomic}{H_{\rm at}}


\newcommand{\tpoissonx}{\frac{\displaystyle r_{1}^{\gamma_{1}}}{\displaystyle (r_{1}+\rho(x,x_{0}))^{1+\gamma_{1}} } }
\newcommand{\tpoissony}{\frac{\displaystyle r_{2}^{\gamma_{2}}}{\displaystyle (r_{2}+\rho(y,y_{0}))^{1+\gamma_{2}} } }
\newcommand{\tsmoothx}{ \big( \frac{\displaystyle \rho(x,x^{'})}{\displaystyle r_{1}+\rho(x,x_{0})} \big)^{\beta_{1}} }
\newcommand{\tsmoothy}{ \big( \frac{\displaystyle \rho(y,y^{'})}{\displaystyle r_{2}+\rho(y,y_{0})} \big)^{\beta_{2}} }
\newcommand{\tconditionx}{\rho(x,x^{'})\leq \frac{1}{\displaystyle
2A}[r_{1}+\rho(x,x_{0})]}
\newcommand{\tconditiony}{\rho(y,y^{'})\leq \frac{1}{\displaystyle
2A}[r_{2}+\rho(y,y_{0})]}



\newcommand{\sumkonektwo}{\sum_{k_{1}=-\infty}^{\infty}\sum_{k_{2}=-\infty}^{\infty}}
\newcommand{\sumtauone}{\sum_{\tau_1\in I_{k_1}}}
\newcommand{\sumtautwo}{\sum_{\tau_2\in I_{k_2}}}
\newcommand{\sumvone}{\sum_{v_1=1}^{N(k_1,\tau_1)}}
\newcommand{\sumvtwo}{\sum_{v_2=1}^{N(k_2,\tau_2)}}

\newcommand{\sumkonektwop}{\sum_{k_{1}^{'}=-\infty}^{\infty}\sum_{k_{2}^{'}=-\infty}^{\infty}}
\newcommand{\sumtauonep}{\sum_{\tau_1^{'}\in I_{k_1}^{'}}}
\newcommand{\sumtautwop}{\sum_{\tau_2^{'}\in I_{k_2}^{'}}}
\newcommand{\sumvonep}{\sum_{v_1^{'}=1}^{N(k_1^{'},\tau_1^{'})}}
\newcommand{\sumvtwop}{\sum_{v_2^{'}=1}^{N(k_2^{'},\tau_2^{'})}}



\newcommand{\DD}{D_{k_1}D_{k_2}}
\newcommand{\DDqta}{\tilde{D}_{k_1}\tilde{D}_{k_2}}
\newcommand{\DDbar}{\bar{D}_{k_1}\bar{D}_{k_2}}
\newcommand{\EE}{E_{k_1}E_{k_2}}
\newcommand{\EEqta}{\tilde{E}_{k_1}\tilde{E}_{k_2}}
\newcommand{\EEbar}{\bar{E}_{k_1}\bar{E}_{k_2}}

\newcommand{\DDp}{D_{k_1}'D_{k_2}'}
\newcommand{\DDqtap}{\tilde{D}_{k_1}'\tilde{D}_{k_2}'}
\newcommand{\DDbarp}{\bar{D}_{k_1}'\bar{D}_{k_2}'}
\newcommand{\EEp}{E_{k_1}'E_{k_2}'}
\newcommand{\EEqtap}{\tilde{E}_{k_1}'\tilde{E}_{k_2}'}
\newcommand{\EEbarp}{\bar{E}_{k_1}'\bar{E}_{k_2}'}



\newcommand{\Qone}{Q_{\tau_1}^{k_1,v_1}}
\newcommand{\Qtwo}{Q_{\tau_2}^{k_2,v_2}}
\newcommand{\Qonep}{Q_{\tau_1^{'}}^{k_1^{'},v_1^{'}}}
\newcommand{\Qtwop}{Q_{\tau_2^{'}}^{k_2^{'},v_2^{'}}}

\newcommand{\sumktauv}{\sum_{k_1,k_2}\sum_{\tau_1,\tau_2}\sum_{v_1,v_2}}
\newcommand{\sumktauvp}{\sum_{k_1^{'},k_2^{'}}\sum_{\tau_1^{'},\tau_2^{'}}\sum_{v_1^{'},v_2^{'}}}

\newcommand{\SR}{\sup_{x_1\in Q_{\tau_1}^{k_1,v_1},x_2\in Q_{\tau_2}^{k_2,v_2} }\big|D_{k_1}D_{k_2}(f)(x_1,x_2) \big|^2}
\newcommand{\TR}{\inf_{x_1\in Q_{\tau_1}^{k_1,v_1},x_2\in Q_{\tau_2}^{k_2,v_2} }\big|E_{k_1}E_{k_2}(f)(x_1,x_2) \big|^2}

\newcommand{\kaiR}{\displaystyle\chi_{ \{Q_{\tau_1}^{k_1,v_1}\times Q_{\tau_2}^{k_2,v_2}\subset
\Omega\} }(k_1,k_2,\tau_1,\tau_2,v_1,v_2)}
\newcommand{\kaiRp}{\displaystyle\chi_{ \{Q_{\tau_1^{'}}^{k_1^{'},v_1^{'}}\times Q_{\tau_2^{'}}^{k_2^{'},v_2^{'}}\subset
\Omega\} }(k_1^{'},k_2^{'},\tau_1^{'},\tau_2^{'},v_1^{'},v_2^{'})}

\newcommand{\muR}{\mu(Q_{\tau_1}^{k_1,v_1})\mu(Q_{\tau_2}^{k_2,v_2})}
\newcommand{\muRp}{\mu(Q_{\tau_1^{'}}^{k_1^{'},v_1^{'}})\mu(Q_{\tau_2^{'}}^{k_2^{'},v_2^{'}})}
\newcommand{\muQone}{\mu(Q_{\tau_1}^{k_1,v_1})}
\newcommand{\muQtwo}{\mu(Q_{\tau_2}^{k_2,v_2})}
\newcommand{\muQonep}{\mu(Q_{\tau_1^{'}}^{k_1^{'},v_1^{'}})}
\newcommand{\muQtwop}{\mu(Q_{\tau_2^{'}}^{k_2^{'},v_2^{'}})}
\newcommand{\omegareverse}{\frac{\displaystyle 1}{\displaystyle \mu(\Omega)}}

\newcommand{\vone}{\mu(Q_1)\vee\mu(Q_1^{'})}
\newcommand{\vtwo}{\mu(Q_2)\vee\mu(Q_2^{'})}



\newcommand{\dxone}{d\mu(x_{1})}
\newcommand{\dxtwo}{d\mu(x_{2})}
\newcommand{\dyone}{d\mu(y_{1})}
\newcommand{\dytwo}{d\mu(y_{2})}
\newcommand{\dzone}{d\mu(z_{1})}
\newcommand{\dztwo}{d\mu(z_{2})}
\newcommand{\dyonep}{d\mu(y_{1}^{'})}
\newcommand{\dytwop}{d\mu(y_{2}^{'})}



\newcommand{\hardy}{H^{1}(\XX\times\XX)}

\newcommand{\poissonone}{\frac{\displaystyle 2^{-k_1\epsilon}}{\displaystyle (2^{-k_1}+\rho(x_1,y_1))^{1+\epsilon}}}
\newcommand{\poissontwo}{\frac{\displaystyle 2^{-k_2\epsilon}}{\displaystyle (2^{-k_2}+\rho(x_2,y_2))^{1+\epsilon}}}

\newcommand{\smoothxone}{\bigg(\frac{\displaystyle \rho(x_1,x_1^{'})}{\displaystyle (2^{-k_1}+\rho(x_1,y_1))}\bigg)^{\epsilon}}
\newcommand{\smoothxtwo}{\bigg(\frac{\displaystyle \rho(x_2,x_2^{'})}{\displaystyle (2^{-k_2}+\rho(x_2,y_2))}\bigg)^{\epsilon}}
\newcommand{\smoothyone}{\bigg(\frac{\displaystyle \rho(y_1,y_1^{'})}{\displaystyle (2^{-k_1}+\rho(x_1,y_1))}\bigg)^{\epsilon}}
\newcommand{\smoothytwo}{\bigg(\frac{\displaystyle \rho(y_2,y_2^{'})}{\displaystyle (2^{-k_2}+\rho(x_2,y_2))}\bigg)^{\epsilon}}

\newcommand{\conditionxone}{\rho(x_1,x_1^{'})\leq \frac{1}{\displaystyle 2A}(2^{-k_1}+\rho(x_1,y_1))}
\newcommand{\conditionxtwo}{\rho(x_2,x_2^{'})\leq \frac{1}{\displaystyle 2A}(2^{-k_1}+\rho(x_2,y_2))}
\newcommand{\conditionyone}{\rho(y_1,y_1^{'})\leq \frac{1}{\displaystyle 2A}(2^{-k_2}+\rho(x_1,y_1))}
\newcommand{\conditionytwo}{\rho(y_2,y_2^{'})\leq \frac{1}{\displaystyle 2A}(2^{-k_2}+\rho(x_2,y_2))}

\newcommand{\xy}{(x_1,x_2,y_1,y_2)}
\newcommand{\xyp}{(x_1,x_2,y_{\tau_1^{'}}^{k_1^{'},v_1^{'}},y_{\tau_2^{'}}^{k_2^{'},v_2^{'}})}



\newcommand{\orth}{2^{-|k_1-k_1^{'}|\epsilon^{'}}2^{-|k_2-k_2^{'}|\epsilon^{'}}}
\newcommand{\orthpone}{\frac{\displaystyle 2^{-(k_1\wedge k_1^{'})\epsilon}}{\displaystyle (2^{-(k_1\wedge k_1^{'})}+\rho(x_1,y_1))^{1+\epsilon}}}
\newcommand{\orthptwo}{\frac{\displaystyle 2^{-(k_2\wedge k_2^{'})\epsilon}}{\displaystyle (2^{-(k_2\wedge k_2^{'})}+\rho(x_2,y_2))^{1+\epsilon}}}
\newcommand{\orthR}{
\bigg( \frac{\displaystyle \mu(\Qone)}{\displaystyle\mu(\Qonep)}
\wedge \frac{\displaystyle \mu(\Qonep)}{\displaystyle\mu(\Qone)}
\bigg)^{\epsilon^{'}} \bigg( \frac{\displaystyle
\mu(\Qtwo)}{\displaystyle\mu(\Qtwop)} \wedge \frac{\displaystyle
\mu(\Qtwop)}{\displaystyle\mu(\Qtwo)} \bigg)^{\epsilon^{'}} }
\newcommand{\orthQone}{\bigg( \frac{\displaystyle \mu(\Qone)}{\displaystyle\mu(\Qonep)} \wedge \frac{\displaystyle \mu(\Qonep)}{\displaystyle\mu(\Qone)} \bigg)^{\epsilon^{'}}}
\newcommand{\orthQtwo}{\bigg( \frac{\displaystyle \mu(\Qtwo)}{\displaystyle\mu(\Qtwop)} \wedge \frac{\displaystyle \mu(\Qtwop)}{\displaystyle\mu(\Qtwo)} \bigg)^{\epsilon^{'}}}
\newcommand{\orthponeR}{\frac{\displaystyle (\muQone\vee\muQonep)^{\epsilon}}{\displaystyle (\muQone\vee\muQonep+dist(\Qone,\Qonep))^{1+\epsilon}}}
\newcommand{\orthptwoR}{\frac{\displaystyle (\muQtwo\vee\muQtwop)^{\epsilon}}{\displaystyle (\muQtwo\vee\muQtwop+dist(\Qtwo,\Qtwop))^{1+\epsilon}}}
\newcommand{\orthRfinal}{
\bigg(\frac{\displaystyle\mu(\Qone)}{\displaystyle\mu(\Qonep)}\wedge\frac{\displaystyle\mu(\Qonep)}{\displaystyle\mu(\Qone)}\bigg)^{1+\epsilon^{'}}
\bigg(\frac{\displaystyle\mu(\Qtwo)}{\displaystyle\mu(\Qtwop)}\wedge\frac{\displaystyle\mu(\Qtwop)}{\displaystyle\mu(\Qtwo)}\bigg)^{1+\epsilon^{'}}
}
\newcommand{\orthQonefinal}{\bigg( \frac{\displaystyle \mu(\Qone)}{\displaystyle\mu(\Qonep)} \wedge \frac{\displaystyle \mu(\Qonep)}{\displaystyle\mu(\Qone)} \bigg)^{1+\epsilon^{'}}}
\newcommand{\orthQtwofinal}{\bigg( \frac{\displaystyle \mu(\Qtwo)}{\displaystyle\mu(\Qtwop)} \wedge \frac{\displaystyle \mu(\Qtwop)}{\displaystyle\mu(\Qtwo)} \bigg)^{1+\epsilon^{'}}}
\newcommand{\orthponeRfinal}{\frac{\displaystyle (\muQone\vee\muQonep)^{1+\epsilon}}{\displaystyle (\muQone\vee\muQonep+dist(\Qone,\Qonep))^{1+\epsilon}}}
\newcommand{\orthptwoRfinal}{\frac{\displaystyle (\muQtwo\vee\muQtwop)^{1+\epsilon}}{\displaystyle (\muQtwo\vee\muQtwop+dist(\Qtwo,\Qtwop))^{1+\epsilon}}}



\newcommand{\yone}{y_{\tau_1}^{k_1}}
\newcommand{\ytwo}{y_{\tau_2}^{k_2}}

\newcommand{\qone}{Q_{\tau_1}^{k_1}}
\newcommand{\qtwo}{Q_{\tau_2}^{k_2}}

\newcommand{\mR}{m_{\qone\times\qtwo}}
\newcommand{\muqone}{\mu(\qone)}
\newcommand{\muqtwo}{\mu(\qtwo)}

\newcommand{\pone}{\frac{\displaystyle 2^{-k_1\epsilon}}{\displaystyle (2^{-k_1}+\rho(x_1,y_{\tau_1}^{k_1}))^{1+\epsilon}}}
\newcommand{\ptwo}{\frac{\displaystyle 2^{-k_2\epsilon}}{\displaystyle (2^{-k_2}+\rho(x_2,y_{\tau_2}^{k_2}))^{1+\epsilon}}}

\newcommand{\ponep}{\frac{\displaystyle 2^{-k_1\epsilon}}{\displaystyle (2^{-k_1}+\rho(x_1^{'},y_{\tau_1}^{k_1}))^{1+\epsilon}}}
\newcommand{\ptwop}{\frac{\displaystyle 2^{-k_2\epsilon}}{\displaystyle (2^{-k_2}+\rho(x_2^{'},y_{\tau_2}^{k_2}))^{1+\epsilon}}}

\newcommand{\poney}{\frac{\displaystyle 2^{-k_1\epsilon}}{\displaystyle (2^{-k_1}+\rho(y_1,y_{\tau_1}^{k_1}))^{1+\epsilon}}}
\newcommand{\ptwoy}{\frac{\displaystyle 2^{-k_2\epsilon}}{\displaystyle (2^{-k_2}+\rho(y_2,y_{\tau_2}^{k_2}))^{1+\epsilon}}}

\newcommand{\sone}{ \bigg(\frac{\displaystyle \rho(x_1,y_1)}{\displaystyle 2^{-k_1}+\rho(x_1,y_{\tau_1}^{k_1})} \bigg)^{\epsilon^{'}} }
\newcommand{\stwo}{ \bigg(\frac{\displaystyle \rho(x_2,y_2)}{\displaystyle 2^{-k_2}+\rho(x_2,y_{\tau_2}^{k_2})} \bigg)^{\epsilon^{'}} }

\newcommand{\pdone}{\frac{\displaystyle 2^{-j_1\epsilon}}{\displaystyle (2^{-j_1}+\rho(x_1,y_1))^{1+\epsilon}}}
\newcommand{\pdtwo}{\frac{\displaystyle 2^{-j_2\epsilon}}{\displaystyle (2^{-j_2}+\rho(x_2,y_2))^{1+\epsilon}}}

\newcommand{\pdonep}{\frac{\displaystyle 2^{-j_1\epsilon}}{\displaystyle (2^{-j_1}+\rho(x_1,\yone))^{1+\epsilon}}}
\newcommand{\pdtwop}{\frac{\displaystyle 2^{-j_2\epsilon}}{\displaystyle (2^{-j_2}+\rho(x_2,\ytwo))^{1+\epsilon}}}


\pagestyle{myheadings}\markboth{\rm\small Yongsheng Han, Ji Li,
M.~Cristina Pereyra and Lesley A.~Ward}{\rm\small Equivalent
definitions of product $\bmo$}

\title[Atomic decomposition of product Hardy spaces $H^p(\widetilde{X})$]{Atomic decomposition of product Hardy spaces
 via wavelet bases on spaces of homogeneous type}

\author{Yongsheng Han, Ji Li, M.~Cristina Pereyra and Lesley A.~Ward}

\thanks{J.L.\ and L.A.W.\ are supported by the
Australian Research Council under Grant No.~ARC-DP160100153
which also supported M.C.P.'s travel. M.C.P.\ is supported by
the National Science Foundation under Grant
No.~NSF-DMS1800587.}

\subjclass[2010]{Primary 42B35; Secondary 43A85, 30L99, 42B30,
42C40}

\keywords{Product Hardy spaces, spaces of homogeneous type,
orthonormal wavelet basis, test functions, distributions,
$(p,q)$-atoms, product atomic Hardy spaces.}

\begin{abstract}
    We provide an atomic decomposition of the product Hardy
    spaces~$H^p(\widetilde{X})$ which were recently developed
    by Han, Li, and Ward in the setting of product spaces of
    homogeneous type $\widetilde{X} = X_1 \times  X_2$. Here
    each factor $(X_i,d_i,\mu_i)$, for $i = 1$, $2$, is a space
    of homogeneous type in the sense of Coifman and Weiss.
    These Hardy spaces make use of the orthogonal wavelet bases
    of Auscher and Hyt\"onen and their underlying reference dyadic grids.
     However, no additional
    assumptions on the quasi-metric or on the doubling measure
    for each factor space are made. To carry out this program,
    we introduce product $(p,q)$-atoms on~$\X$ and product
    atomic Hardy spaces $H^{p,q}_{{\rm at}}(\widetilde{X})$. As
    consequences of the atomic decomposition of~$H^p(\X)$, we
    show that for all $q > 1$ the product atomic Hardy spaces
    coincide with the product Hardy spaces,
    and we show that the product Hardy spaces are independent of the
    particular choices of both the wavelet bases and the
    reference dyadic grids. Likewise, the product Carleson
    measure spaces~${\rm CMO}^p(\widetilde{X})$, the bounded
    mean oscillation space~${\rm BMO}(\widetilde{X})$, and the
    vanishing mean oscillation space~${\rm
    VMO}(\widetilde{X})$, as defined by Han, Li, and Ward, are
    also independent of the particular choices of both wavelets
    and reference dyadic grids.
\end{abstract}

\maketitle

\vspace{-1cm}

\tableofcontents

\section{Introduction}\label{sec:introduction}
\setcounter{equation}{0}

The product Hardy spaces $H^p(\widetilde{X})$ were recently
developed in~\cite{HLW} in the setting of product spaces of
homogeneous type $\widetilde{X} = X_1 \times X_2$, where each
factor $(X_i,d_i,\mu_i)$, $i = 1$, 2, is a space of homogeneous
type in the sense of Coifman and Weiss. In this paper we
provide an atomic decomposition of these product Hardy spaces
$H^p(\widetilde{X})$.

Spaces of homogeneous type were introduced by Coifman and Weiss
in the early 1970s~\cite{CW}. We say that $(\XX,d,\mu)$ is a
\emph{space of homogeneous type in the sense of Coifman and
Weiss} if $X$ is a set, $d$ is a quasi-metric on~$\XX$, and
$\mu$ is a nonzero {Borel}-regular measure on~$X$ satisfying
the doubling condition.  A \emph{quasi-metric}~$d$ on a
set~$\XX$ is a function
$d:\XX\times\XX\longrightarrow[0,\infty)$ satisfying~(i)
$d(x,y) = d(y,x) \geq 0$ for all $x$, $y\in\XX$; (ii) $d(x,y) =
0$ if and only if $x = y$; and (iii) the \emph{quasi-triangle
inequality}: there is a constant $A_0\in
[1,\infty)$ such that for all $x$, $y$, $z\in\XX$, 
\begin{equation}\label{eqn:quasitriangleineq}
    d(x,y)
    \leq A_0 \big[d(x,z) + d(z,y)\big].
\end{equation}
The quasi-metric ball is defined by $B(x,r) := \{y\in X: d(x,y)
< r\}$ for $x\in X$ and $r > 0$. Note that the quasi-metric, in
contrast to a metric, may not be H\"older regular and
quasi-metric balls may not be
open\footnote{\textcolor{black}{Any quasi-metric defines a
topology, for which the balls $B(x, r)$  form a neighbourhood base of each $x\in X$, in particular, a set $A$ is open
if for every $x\in A$ there is $r>0$ such that $B(x,r)\subset A$, see \cite[Section~4,  Theorem~4.5]{Wi}. However
when $A_0 > 1$ the balls need not be open.  The measure $\mu$
is assumed to be defined on a $\sigma$-algebra that contains
all balls $B(x,r)$ and all Borel sets induced by this
topology.}}. We say that a nonzero  measure $\mu$ satisfies the
\emph{doubling condition} if there is a constant~$C_\mu \geq 1$
such that for all $x\in\XX$ and $r > 0$,
\begin{equation}\label{eqn:doubling condition}
   0
   < \mu \big (B(x,2r) \big )
   \leq C_\mu \,  \mu \big (B(x,r) \big )
   < \infty.
\end{equation}
We say a measure $\mu$ is \emph{Borel regular} if for each
measurable set $A$ there is a Borel set $B$ such that $B\subset
A$ and $\mu(B) = \mu(A)$. This Borel regularity ensures that
the Lebesgue Differentiation Theorem holds on~$(X,d,\mu)$ and
that step functions are dense in $L^2(X,\mu)$~\cite{AM,AH2}.

We point out that the doubling condition~\eqref{eqn:doubling
condition} implies that there exist positive constants~$C$
and~$\omega$ (known as an \emph{upper dimension} of~$X$) such
that for all $x\in X$, $\lambda\geq 1$ and $r > 0$,
\begin{equation}\label{eqn:upper dimension}
    \mu \big (B(x, \lambda r) \big )
    \leq C\lambda^{\omega} \mu \big (B(x,r) \big ).
\end{equation}
We can express $C$ and $\omega$ in condition~\eqref{eqn:upper
dimension} in terms of the doubling constant~$C_{\mu}$ of the
measure. In fact we can and will choose $C = C_{\mu}\geq 1$ and
$\omega = \log_2{C_{\mu}}$.

Throughout this paper we assume that $\mu(X) = \infty$. Given a
space of homogeneous type $(X,d,\mu)$, the quasi-triangle
constant $A_0$, the doubling constant $C_{\mu}$, and an upper
dimension~$\omega$ are referred to as the \emph{geometric
constants} of the space $X$.

In the classical theory, the Hardy spaces $H^p$ can be defined
via maximal functions, via approximations to the identity and
Littlewood-Paley theory, via square functions, or via atomic
decompositions, and all these definitions coincide. When moving
to more exotic settings one can start with any of the
equivalent definitions and then hope to show that they all
define the same space.
In the one-parameter setting of spaces of homogeneous type this
program was carried out, but additional conditions were
required on the quasi-metric or on the measure. The first
author was involved in many of these developments. For more
details see Section~\ref{sec:history}.

A natural question arises: can one develop the theory of the
spaces $H^p$ and $\bmo$ on spaces of homogeneous type in the
sense of Coifman and Weiss, with only the original
quasi-metric~$d$ and a Borel-regular doubling measure~$\mu$?

This question was posed, and answered in the affirmative, in
\cite{HLW}, in both the one-parameter and product settings. The
key idea used in \cite{HLW} was to employ the remarkable
orthonormal wavelet basis constructed by Auscher and Hyt\"onen
for spaces of homogeneous type~\cite{AH} to define appropriate
product square functions and Hardy spaces. Note that it is in
the construction of the wavelets that the Borel regularity of
the measure is required~\cite{AH2}. In the current paper we
provide an atomic decomposition in the product setting and, as
a consequence of our main result, we show that the
$H^p(\widetilde{X})$ spaces defined via a wavelet basis in
\cite{HLW} are independent {not only of the chosen wavelet
basis, but also of the choice of underlying reference dyadic
grids.}

In the one-parameter setting the Hardy space $H^p(X)$ was
built in~\cite{HLW} using the Hyt\"onen-Auscher wavelets
(themselves built upon a fixed reference dyadic grid). Using
the Plancherel-P\'olya inequalities proved in~\cite{HLW} (see
also~\cite{H2}), {one can observe}
that the spaces $H^p(X)$ are well defined, meaning they are
independent of the choice of wavelet basis (built upon the same
reference dyadic grid).
Later, in \cite{HHL}, the atomic and molecular
characterizations of the one-parameter Hardy space were
studied; it was shown that $H^p(X)$ is equivalent to $H^p_{{\rm
at}}(X)$, the Coifman-Weiss atomic Hardy space, and therefore
the definition of~$H^p(X)$ is independent of the choice of the
wavelets and of the underlying reference dyadic grid.   See also the work in \cite{FY}
characterizing the atomic Hardy space via wavelet basis.
More recently, in \cite{HHLLYY}, a complete real-variable theory 
of one-parameter Hardy spaces on spaces of homogeneous type was provided, 
in particular proving the  radial maximal characterization of $H^p_{{\rm
at}}(X)$ and   completely answering a question  by Coifman and Weiss \cite[p.642]{CW2}.

We now turn to the product case. As in the one-parameter
case, the product Plancherel-P\'olya inequalities proved
in~\cite{HLW} would imply  that $H^p(\widetilde{X})$ is
independent of the choice of wavelet basis (built upon fixed
reference dyadic grids on each component of
the product~$\widetilde{X}$ of spaces of homogenenous type). 
In this paper, instead we introduce the product $(p,q)$-atoms
for $0<p\leq 1<q$ and corresponding atomic product Hardy spaces
$H^{p,q}_{{\rm at}}(\widetilde{X})$, whose definition is
independent of any wavelet bases and also of the reference
dyadic grids.
As a direct application, we deduce that the product Hardy
spaces $H^p(\widetilde{X})$ are independent of the choices of
wavelets and of underlying reference dyadic grids. This result
is consistent with the product theory on the Euclidean setting
$\mathbb{R}^n\times \mathbb{R}^m$, and parallel to the
one-parameter theory on spaces of homogenenous type $(X,d,\mu)$
as presented in \cite{HHL}.

Important features in the one-parameter case, treated
in~\cite{HHL}, are that $H^p(X)\cap L^2(X)$ is dense in
$H^p(X)$ and functions in  $H^p(X)\cap L^2(X)$ have a nice
atomic decomposition  which converges both in $L^2(X)$ and
$H^p(X)$. These features allow a linear operator bounded on
$L^2(X)$ to pass through the sum in an atomic decomposition,
hence reducing the proof of the boundedness of the operator to
verifying uniform boundedness on atoms. See the discussion
in~\cite[p.3431--3432]{HHL} regarding  applications of these
features to prove  $T(1)$ theorems.   Similar features hold in
the product case, as shown in~\cite{HLLin}. In this paper, we
will show that for all $q > 1$ and all $p$ with $0<p\leq 1$,
not only is $H^p(\X)\cap L^q(\X)$ dense in $H^p(\X)$, but also
$H^p(\X)\cap L^q(\X)$ is a subset of $L^p(\widetilde{\XX})$,
with the $L^p$-(semi)norm controlled by the $H^p$-(semi)norm.
These facts will be an important cornerstone in proving the
atomic decomposition for~$H^p(\widetilde{X})$.

The product Carleson measure space $\cmo^p (\widetilde{\XX})$
was introduced in \cite{HLW}. It was shown in the same paper
that $\cmo^p (\widetilde{\XX})$ is the dual of
$H^p(\widetilde{\XX})$, that the space of bounded mean
oscillation  $\bmo (\widetilde{\XX})$ coincides with
$\cmo^1(\widetilde{\XX})$ and hence is the dual of
$H^1(\widetilde{X})$, and that the vanishing mean oscillation
space~$\vmo (\widetilde{\XX})$ is the predual of
$H^1(\widetilde{\XX})$. As a consequence of our result for the
product Hardy spaces, we see that the spaces $\cmo^p
(\widetilde{\XX})$, $\bmo(\widetilde{X})$, and $\vmo
(\widetilde{\XX})$ are also independent not only of the chosen
wavelet basis, but also of the  chosen  reference dyadic grids.
{Note that in the one-parameter case it was shown in~
\cite[Proposition 4.3]{HHL} that $\cmo^p(X)$ coincides with the
Campanato space $\mathcal{C}_{\frac1p-1}(X)$, which is the dual
of Coifman-Weiss atomic Hardy space $H^p_{{\rm at}}(X)$, and is
a space defined independently of any wavelets and their
reference dyadic grids.

When $\widetilde{\XX} = X_1\times \dots \times X_n$, the spaces
$H^p(\widetilde{X})$ constructed in~\cite{HLW} are defined for
all $p > \max\big\{\omega_i/(\omega_i + \eta_i) : i =
1, 2, \dots, n\big\}$.
Here $\eta_i\in (0,1)$ is the H\"older regularity exponent
of the Auscher-Hyt\"onen wavelets, defined on the spaces of
homogeneous type $(X_i,d_i,\mu_i)$, that are used in the
construction of $H^p(\widetilde{\XX})$, and $\omega_i > 0$ is
an upper dimension of $X_i$, for~$i = 1, 2, \dots , n$.

Our main result is the following.

\begin{main}
    Let $\widetilde{X} = X_1\times X_2$, where for $i = 1$,
    $2$, $(X_i,d_i,\mu_i)$ are spaces of homogeneous type in
    the sense of Coifman and Weiss as described above, with
    quasi-metrics~$d_i$ and Borel-regular doubling
    measures~$\mu_i$. Let ${\omega}_i$ be an upper dimension
    for~$X_i$, and let $\eta_i$ be the exponent of regularity
    of the Auscher-Hyt\"onen wavelets used in the construction
    of the Hardy space~$H^p(\widetilde{\XX})$. Suppose that
    $\max\big\{\omega_i/(\omega_i + \eta_i) : i = 1, 2\big\}
    < p \leq 1 < q < \infty$, and {$f\in L^q(\widetilde{\XX})$}.
    Then $f\in H^p( \widetilde{\XX} )$ if and only if $f$ has
    an atomic decomposition, that is,
    \begin{eqnarray}\label{atom decom_copy}
        f=\sum_{j=-\infty}^\infty\lambda_ja_j,
    \end{eqnarray}
    where the $a_j$ are $(p, q)$-atoms, $\sum_{j =
    -\infty}^{\infty}|\lambda_j|^p<\infty,$
    and {the series converges in 
    $L^q(\widetilde{\XX})$}.  Moreover, {the series also
    converges in $H^p(\widetilde{\XX})$ and}
    \begin{eqnarray*}
        \|f\|_{H^p( \widetilde{\XX})}
        \sim \inf \Big\{ \Big (\sum_{j = -\infty}^{\infty}|\lambda_j |^p \Big )^{\frac{1}{p}}:
            \,f = \sum_{j = -\infty}^{\infty}\lambda_ja_j\Big\},
    \end{eqnarray*}
    where the infimum is taken over all decompositions as in
    \eqref{atom decom_copy}. The implicit constants are
    independent of the $L^q( \widetilde{\XX})$-norm and the
    $H^p( \widetilde{\XX} )$-{\rm (}semi{\rm )}norm  of $f$,
    they are only dependent on the geometric constants of the
    spaces $X_i$ for $i = 1$, 2.
\end{main}

For simplicity we work in the case of two factors:
$\widetilde{X} = X_1 \times X_2$. However, we expect our
results and proofs to go through for arbitrarily many factors;
in particular one would need a $n$-parameter version of
Journ\'e's Lemma on spaces of homogeneous type, which would
generalise both Pipher's $n$-parameter Euclidean
version~\cite{P} and Han, Li and Lin's two-parameter version on
spaces of homogeneous type~\cite{HLLin}.

We deduce three corollaries from the Main Theorem. First, the
atomic product spaces~$H^{p,q}_{{\rm at}}$ we define coincide,
for all $q > 1$, with the product Hardy spaces $H^p$ defined
in~\cite{HLW}.

\begin{corollaryA}
    \label{cor:A_Hp_atomic_is_Hp}
    For all $q$ with $1<q<\infty$ and $p$ with
    $\max\big\{\omega_i/(\omega_i + \eta_i): i = 1, 2\big\} < p\leq 1$, we have
    $$H^{p,q}_{{\rm at}}(X_1\times X_2) = H^p(X_1\times X_2).$$
\end{corollaryA}
\noindent Thus, we can define $H^p_{{\rm at}}(X_1\times X_2)$ to be $H^{p,q}_{{\rm at}}(X_1\times X_2)$
for any $q>1$. 

Second, as a consequence, we deduce that the product Hardy
spaces are independent both of wavelets and of reference dyadic
grids.

\begin{corollaryB}\label{cor:B_Hp_well_defined}
   For all $p$ with $p >
    \max\big\{\omega_i/(\omega_i + \eta_i): i = 1,
    2\big\}$, the product Hardy spaces
    $H^p(X_1\times X_2)$ as defined in~\cite{HLW} are
    independent of the particular choices of the
    Auscher-Hyt\"onen wavelets and of the reference dyadic
    grids used in their construction.
\end{corollaryB}

Third, the Carleson measure spaces and the spaces of bounded
mean oscillation and of vanishing mean oscillation are also
independent of both wavelets and reference dyadic grids.

\begin{corollaryC}
    \label{cor:C_CMOp_BMO_VMO_well_defined} 
   For all $p$ with  $ \max\big\{\omega_i/(\omega_i + \eta_i)
    : i = 1, 2\big\}<p\leq 1$, 
    the Carleson measure spaces ${\rm CMO}^p(X_1\times X_2)$,
    the space of bounded mean oscillation ${\rm BMO}(X_1\times
    X_2)$, and the space of vanishing mean oscillation ${\rm
    VMO}(X_1\times X_2)$, as defined in~\cite{HLW}, are
    independent of the particular choices of the
    Auscher-Hyt\"onen wavelets and of the reference dyadic
    grids used in their construction.
\end{corollaryC}

In the special case when $p = 1$ and $q = 2$, the $(p,q)$-atoms
defined in this paper, and the corresponding atomic
decomposition found for $H^p(\widetilde{X})\cap
L^q(\widetilde{X})$, were used in establishing dyadic structure
theorems for $H^1(\widetilde{\XX})$ and $\bmo(\widetilde{\XX})$
\cite[Definition~5.3 and Theorem~5.4]{KLPW}. To achieve this
goal, corresponding dyadic atomic Hardy spaces were introduced
in~\cite[Definition~6.3 and Theorem~6.5]{KLPW}.

We would like to mention that Fu and Yang~\cite{FY} present a
characterization of the Coifman and Weiss atomic Hardy space
$\Hatomic^1(X)$ in the one-parameter case, using the
Auscher-Hyt\"onen wavelets, under the assumptions that $(X, d,
\mu)$ is a metric measure space of homogeneous type, diam$(X) =
\infty$, and $X$ is a non-atomic space, meaning that
$\mu(\{x\})=0$ for all $x\in X$.  They prove that the
Auscher-Hyt\"onen wavelets form an unconditional basis in
$H^1(X)$ and from there they deduce that a function being in
$\Hatomic^1(X)$ is equivalent to the unconditional convergence
in $L^1(X)$ of the function's wavelet expansion, and equivalent
to the boundedness on $L^1(X)$ of  each of  three different discrete
square functions, one of them coinciding with that in the
definition of $H^1(X)$ presented in \cite{HLW}. All these
one-parameter Hardy spaces $H^1(X)$ coincide when the
conditions assumed in~\cite{FY} are met. Fu and Yang did not
address the case $p<1$, nor the product 
case, which are the focus of this article.

The paper is organized as follows. In Section~\ref{sec:history}
we place our work in historical context, describing some of
the progress made to date, from the original work of Coifman
and Weiss until the present setting, mostly in the one-parameter case.

In Section~\ref{sec:preliminaries} we recall the basic
ingredients involved in the definition of product Hardy and
$\bmo$ spaces, on spaces of homogeneous type in the sense of
Coifman and Weiss with only the original quasi-metric and a
Borel-regular doubling measure $\mu$, as introduced
in~\cite{HLW}. These preliminaries include  the
Hyt\"onen-Kairema systems of dyadic cubes \cite{HK}, the
Auscher-Hyt\"onen orthonormal basis and reference dyadic grids
\cite{AH,AH2}, and the test functions and distributions in both
the one-parameter and product settings~\cite{HLW}.

In Section~\ref{sec:productHp} we recall the definitions
in~\cite{HLW} of product Hardy spaces $H^p(X_1\times X_2)$;
{their duals, the Carleson measure spaces ${\rm
CMO}^p(X_1\times X_2)$; the space of bounded mean oscillation
${\rm BMO}(X_1\times X_2)$; and the space of vanishing mean
oscillation ${\rm VMO}(X_1\times X_2)$, which turns out to be
the predual of $H^1(X_1\times X_2)$}. These definitions are
based on product square functions, themselves defined using the
Auscher-Hyt\"onen wavelets and the reference dyadic grids used
in their construction~\cite{HLW}. We prove a key new lemma in
Section~\ref{sec:productHp} that allows us to decompose the
Auscher-Hyt\"onen wavelets into compactly supported building
blocks rescaled as needed and with appropriate size,
smoothness, and cancellation properties, following the approach
in Nagel and Stein~\cite{NS}. In turn this lemma allows us to
show that, within the allowed range of $p$ dictated by the
geometric constants and the H\"older-continuity parameters of
the wavelets, functions in $H^p(X_1\times X_2)\cap
L^q(X_1\times X_2)$  for $1<q<\infty$
are $L^p$-integrable, with $L^p$-(semi)norm controlled by their $H^p$-(semi)norm.

In Section 5 we introduce the product $(p,q)$-atoms and atomic
product Hardy spaces $H^{p,q}_{{\rm at}}(X_1\times X_2)$ for
$1<q <\infty$ and for $p$ in the same range for which the
product Hardy spaces $H^p(X_1\times X_2)$ are defined.
We restate the Main Theorem, and use it to prove
Corollaries~A, B, and~C, thus establishing that the atomic
product Hardy spaces $H^{p,q}_{{\rm at}}(X_1\times X_2)$
coincide with the product Hardy spaces $H^p(X_1\times X_2)$ for
all $q > 1$, and that the spaces ${\rm CMO}^p(X_1\times X_2)$,
${\rm BMO}(X_1\times X_2)$, and ${\rm VMO}(X_1\times X_2)$ are
independent of the choices of wavelet bases and of reference
dyadic grids on $X_1$ and~$X_2$ used in their construction.
Finally we prove the Main Theorem, yielding an atomic
decomposition for $H^p(X_1\times X_2)\cap L^q(X_1\times X_2)$
in terms of $(p,q)$-atoms for each $q$ with $1< q<\infty$, with
convergence in both $H^p$ and~$L^q$ and showing that $(p,q)$-atoms
 are uniformly in $H^p(X_1\times X_2)$. Key in this decomposition
is the use of a Journ\'e-type covering lemma in the product
setting, which was proved in~\cite{HLLin}.

Throughout the paper the following notation is used. First,
$A\lesssim B$ means there is a constant $C>0$ depending only on
the geometric constants (quasi-triangle constants of the
quasi-metrics, doubling constants of the measures, and upper
dimensions of $X_i$ for $i = 1$, 2) such that $A\leq CB$.
Second, $A\sim B$ means that $A\lesssim B$ and $B\lesssim A$.
Third, the value of a constant $C > 0$ may change from line to
line within a string of inequalities. If the constant~$C$
depends on some other parameter(s), for example on $q > 1$ and
$\delta > 0$, then we may denote it by $C_{q,\delta}$.
Likewise, the notation $\lesssim_{q,\delta}$ indicates that the implied
constant in the inequality depends also on the parameters $q$ and
$\delta$. We denote by $\chi_A$ the characteristic function of
a set $A\subset X$, that is, $\chi_A(x)=1$ if $x\in A$ and
$\chi_A(x)=0$ otherwise.


\section{Context and significance}\label{sec:history}

In this section we discuss the developments in the theory of
one-parameter Hardy spaces that led to the results presented in
this paper. This is by no means a comprehensive historical
survey, rather a series of snapshots that will give some
perspective to our work. For a  more complete survey
see~\cite{HHL2} and also \cite{HHLLYY}.

We  recall the atomic Hardy space $\Hatomic^p(X)$ on a space of
homogeneous type, following~\cite{CW2}.
Given $(X, d, \mu)$, a space of homogeneous type in the sense
of Coifman and Weiss, as presented in the Introduction, the
atomic Hardy space $\Hatomic^p(X)$ is defined to be a certain
subcollection of the bounded linear functionals on the
Campanato space $\mathcal{C}_\alpha(X)$ with $\alpha =
\frac{1}{p} - 1$, $0 < p \leq 1$.
Namely, $\Hatomic^p(X)$ is defined to be those bounded linear
functionals on $\mathcal{C}_\alpha(X)$ that admit an atomic
decomposition
\begin{equation}\label{atomic hs}
    f = \sum_{j=1}^\infty \lambda_j a_j,
\end{equation}
where the functions $a_j$ are $(p, 2)$-atoms,
$\sum_{j=1}^\infty |\lambda_j|^p<\infty$, and the series in
(\ref{atomic hs}) converges in the dual space of
$\mathcal{C}_\alpha(X)$.  The quasi-norm of $f$ in $\Hatomic^p(X)$ is defined by
\[
    \|f\|_{H^p_{{\rm at}}(X)}
    := \inf \Big \lbrace \Big (\sum_{j=1}^\infty |\lambda_j|^p\Big)^{\frac{1}{p}}
        \Big \rbrace,
\]
where the infimum is taken over all such atomic representations
of $f$.

Here a function $a(x)$ is said to be a \emph{$(p, 2)$-atom} if
the following conditions hold:
\begin{itemize}
    \item[(i)] (Support condition) the support of $a(x)$ is
        contained in a ball $B(x_0,r)$ for some $x_0\in X$
        and $r>0$;
    \item[(ii)] (Size condition)  $\|a\|_{L^2(X)} \leq
        \mu \big (B(x_0,r)\big )^{\frac{1}{2}-\frac{1}{p}}$; and
    \item[(iii)] (Cancellation condition) $\int_X a(x)\,
        d\mu(x) = 0$.
\end{itemize}

Recall that the Campanato spaces $\mathcal{C}_\alpha(X)$,
$\alpha \geq 0$, consists of those functions $f$ for which
\begin{eqnarray}\label{Lip space}
    \left\lbrace \frac{1}{\mu(B)} \int_B|f(x)-f_B|^2 d\mu(x)
        \right\rbrace^{\frac{1}{2}}
    \leq C[\mu(B)]^\alpha,
\end{eqnarray}
where $B$ is any quasi-metric ball, $f_B :=
\frac{1}{\mu(B)}\int_B f(x)\,d\mu(x)$, and the constant $C>0$
is independent of the ball $B$.
Let $||f||_{{\mathcal C}_{\alpha}(X)}$ be the infimum of all
$C$ for which (\ref{Lip space}) holds. {On~$\mathbb{R}^n$ the
Campanato spaces $\mathcal{C}_{\alpha}(\mathbb{R}^n)$ coincide
with the $\alpha$-Lipschitz class when $0<\alpha\leq 1$ and
with $\bmo$ when $\alpha =0$, thanks to the John-Nirenberg
Lemma.}

The Coifman-Weiss definition of the atomic Hardy space $\Hatomic^p(X)$ does
not require any regularity on the quasi-metric $d$, and
requires only the doubling property on the Borel-regular
measure~$\mu$. Moreover, for each atomic decomposition
$\sum_{j=1}^\infty \lambda_j a_j$ where the functions $a_j$ are
$(p, 2)$-atoms with $\sum_{j=1}^\infty |\lambda_j|^p<\infty,$
the series automatically converges in the dual space of
$\mathcal{C}_\alpha(X)$ with $\alpha =\frac{1}{p} - 1.$ Indeed,
if $a$ is a $(p,2)$-atom and $g\in \mathcal{C}_\alpha(X)$ with
$\alpha =\frac{1}{p} - 1,$ then, applying first the support and
cancellation conditions on the atom  $a$ and second  H\"older's
inequality together with the size condition on the atom $a$, we
obtain
$$
    \Big |\int_B a(x)g(x)\,d\mu(x) \Big |
    =\Big |\int_B a(x)[g(x)-g_B]\,d\mu (x) \Big|
    \leq \|a\|_2 \Big (\int_B [g(x)-g_B]^2\,d\mu(x )\Big )^{\frac{1}{2}}
    \leq \|g\|_{\mathcal{C}_\alpha(X)},
$$
where $B = B(x_0,r)$. 

Therefore, if $\sum_{j=1}^\infty \lambda_j a_j$ is an atomic
decomposition, $g\in \mathcal{C}_\alpha(X)$, and  $\alpha
=\frac{1}{p} - 1,$ then
$$\Big | \big\langle\sum_{j=1}^\infty \lambda_j a_j, g\big\rangle
\Big |\leq \sum_{j=1}^{\infty}|\lambda_j| \, \|g\|_{\mathcal{C}_\alpha(X)}
\leq \Big \{ \sum_{j=1}^{\infty}|\lambda_j|^p \Big \}^{\frac{1}{p}}
\|g\|_{\mathcal{C}_\alpha(X)},$$
which implies that the atomic decomposition $\sum_{j=1}^\infty
\lambda_j a_j$ converges in the dual space
of~$\mathcal{C}_\alpha(X)$.

In fact, in \cite[Theorem A, p.592]{CW2}, Coifman and Weiss
define $(p,q)$-atoms, replacing 2 by $q > 1$ in the definition
above, and define corresponding atomic Hardy spaces $\Hatomic^{p,q}(X)$.
They show that for each fixed~$p\leq 1$, the spaces
$\Hatomic^{p,q}(X)$ for $q > 1$ all coincide. We will show in Section~\ref{sec:atomicHp} that the analogue of this result
holds for appropriately defined product $(p,q)$-atoms and product atomic spaces $H^{p,q}_{{\rm at}}(X_1\times X_2)$.

The atomic Hardy spaces have many applications. For example, if
an operator $T$ is bounded on $L^2(X)$ and from $\Hatomic^p(X)$
to $L^p(X)$ for some $p\leq 1,$ then $T$ is bounded on $L^q(X)$
for $1<q\leq 2.$ See~\cite{CW2} for this and for more
applications.

We would like to point out that Coifman and Weiss introduced
the atomic Hardy spaces $\Hatomic^p(X)$ on spaces of
homogeneous type $(X, d, \mu )$ where the quasi-metric balls
were required to be open; see \cite{CW2} for more details.
To establish the maximal function characterization of
the atomic Hardy space of Coifman and Weiss, some additional
geometrical considerations on the quasi-metric $d$ and the
measure $\mu$ were imposed. For this purpose, Mac\'ias and
Segovia~\cite{MS1} proved the following fundamental results.
The first pertains to quasi-metric spaces; the second to spaces
of homogeneous type.

{First, suppose that $(X,d)$ is a space endowed with a
quasi-metric $d$
that may have no regularity. Then there exists a quasi-metric
$d'$  {that is topologically equivalent to $d$} such that
$d(x,y) \sim d'(x,y)$ for all $x$ ,$y\in X$ and there exist
constants $\theta\in(0,1)$ and $C > 0$ so that $d'$ has the
following regularity:
\begin{eqnarray}\label{smetric}
    |d'(x,y) - d'(x',y)|
    \le C \, d'(x,x')^\theta \,
        [d'(x,y) + d'(x',y)]^{1 - \theta}
\end{eqnarray}
for all $x$, $x'$, $y\in X$. Moreover, if the quasi-metric
balls are defined by this new quasi-metric~$d'$, that is,
$B'(x,r) := \{y\in X: d'(x,y) < r\}$ for $r > 0$, then these
balls are open in the topology induced by~$d'$. See
\cite[Theorem 2, p.259]{MS1}.}

Second, suppose that $(X,d,\mu)$ is a space of homogeneous
type in the sense of Coifman and Weiss, with the property that
the balls are open subsets. Then the function $d'':X\times X\to
\mathbb{R}$ defined by $d''(x,y) := \inf{\{\mu{(B)} : x,y\in B,
B \;\mbox{is a $d$-ball}\} }$ if $x\neq y$,
 and $d''(x,y) = 0$ if $x = y$, is a quasi-metric
topologically equivalent to $d$. Furthermore, the measure $\mu$
satisfies the following property for all $d''$-balls~
$B''(x,r)$, where $x\in X$ and $r > 0$:
\begin{eqnarray}\label{regular}
    \mu \big (B''(x,r) \big )\sim r.
\end{eqnarray}
See \cite[Theorem 3, p.259]{MS1}.
Spaces satisfying property \eqref{regular} are called
\emph{$1$-Ahlfors regular quasi-metric spaces}\footnote{A
quasi-metric Borel measure space $(X,d,\mu)$ is
$n$-\emph{Ahlfors regular} if $ \mu \big (B(x,r) \big )\sim
r^n$.}. Note that property \eqref{regular} is much stronger
than the doubling condition.

Starting with a quasi-metric~$d$ for which the balls are not
necessarily open, we can obtain $d'$, 
and we can then pass to its topologically equivalent
quasi-metric $d''(x,y) := \inf\{\mu{(B')}: x,y\in B', B'
\;\mbox{is a $d'$-ball}\}$ to obtain a quasi-metric
satisfying~\eqref{smetric} 
and with the measure $\mu$ satisfying~\eqref{regular}.

Mac\'{i}as and Segovia obtained a grand maximal function
characterization for the atomic Hardy spaces $H^p (X)$
on spaces of homogeneous type $(X,d,\mu)$ 
that satisfy the regularity condition \eqref{smetric} on the
quasi-metric~$d$,
and property~\eqref{regular} on the measure $\mu$, with
${1}/{(1 + \theta)} < p \leq 1$, where $\theta$ is the
regularity exponent of the quasi-metric~\cite[Theorem (5.9),
p.306]{MS2}.

For an authoritative modern account of Hardy spaces on
$n$-Ahlfors regular quasi-metric spaces, see the book by
Alvarado and Mitrea~\cite{AM}. Given a quasi-metric~$d$, the
authors work with an equivalence class of quasi-metrics that
includes $d$ and the Mac\'ias-Segovia quasi-metric. In
contrast, the approach in the present paper is to keep the
original quasi-metric~$d$ untouched but to allow for a certain
randomness in the cubes that enter into the construction of the
wavelets.

To develop the Littlewood-Paley characterization of Hardy
spaces on \emph{normal spaces of homogeneous type $(X,d,\mu)$
of order~$\theta$}, in other words, spaces satisfying the
regularity condition~\eqref{smetric} on the quasi-metric $d$
and property \eqref{regular} on the measure $\mu$, a suitable
approximation to the identity was required. The construction of
such an approximation to the identity is due to Coifman
\cite{DJS}, and this construction leads to a corresponding
Calder\'on-type reproducing formula and Littlewood-Paley theory
\cite[p.3--4]{DH}.
A further discretization of this  Calder\'on reproducing
formula is needed, and it was achieved, using the dyadic cubes
of  Christ \cite{Chr},  by the first author and  Sawyer. See
\cite{H1, H2, HS} for more details. In the present paper, a
further discretization will also be needed; we will instead use
the dyadic cubes of  Hyt\"onen and Kairema~\cite{HK} on which
the wavelets of Auscher and Hyt\"onen \cite{AH, AH2} are based.

To carry out the Littlewood-Paley characterization of the
atomic Hardy space on a normal space $(X, d,\mu)$ of
order~$\theta$,
the following test function spaces were introduced
in~\cite{HS}.

\begin{definition}[Test functions \cite{HS}]\label{def-of-test-func-space1}
    Let $(X,d,\mu)$ be a normal space of homogeneous type of
    order $\theta$. Fix $x_0\in X$, $r > 0$,
    $\beta\in(0,\theta]$ where $\theta$ is the regularity
    exponent of~$d$, and $\gamma > 0$. A function $f$ defined
    on $X$ is said to be a {\it test function of type
    $(x_0,r,\beta,\gamma)$ centered at $x_0\in X$} if $f$
    satisfies the following three conditions:
    \begin{enumerate}
        \item[(i)] \textup{(Size condition)} For all $x\in
            X$,
            \[
                |f(x)|
                \leq C \,\frac{r^\gamma }{\big (r + d(x,x_0)\big )^{1+\gamma}}.
            \]

        \item[(ii)] \textup{(H\"older regularity
            condition)} For all $x$, $y\in X$ with $d(x,y)
            < (2A_0)^{-1}\big (r + d(x,x_0)\big )$,
            \[
                |f(x) - f(y)|
                \leq C \Big(\frac{d(x,y)}{r + d(x,x_0)}\Big)^{\beta}
                \frac{r^\gamma}{\big (r + d(x,x_0)\big )^{1+\gamma}}.
            \]

        \item[(iii)] \textup{(Cancellation condition)}
            \[
                \int_X f(x) \, d\mu(x)
                = 0.
            \]
    \end{enumerate}
\end{definition}

Denote by $\mathcal M(x_0,r,\beta,\gamma)$ the set of all test
functions of type $(x_0,r,\beta,\gamma)$. The norm of $f$ in $
\mathcal M(x_0,r,\beta,\gamma)$ is defined by
\[
    \|f\|_{\mathcal M(x_0,r,\beta,\gamma)}
    := \inf\{C>0:\ {\rm(i)\  and \ (ii)}\ {\rm hold} \}.
\]

For each fixed $x_0$, let $\mathcal M(\beta,\gamma) :=
\mathcal M(x_0,1,\beta,\gamma)$. It is easy to check that for each fixed
$x_0'\in X$ and $r > 0$, we have $\mathcal M(x_0',r,\beta,\gamma) =
\mathcal M(\beta,\gamma)$ with equivalent norms. Furthermore, it is also
easy to see that $\mathcal M(\beta,\gamma)$ is a Banach space with
respect to the norm on $\mathcal M(\beta,\gamma)$.

We remark that the above test function space $\mathcal
M(\beta,\gamma)$ on $(X,d,\mu)$ offers the same service as the
Schwartz test function space $\mathcal S_\infty=\{f\in \mathcal
S: \int f(x)x^\alpha \,dx=0, |\alpha|\geq 0\}$ does on $\mathbb{R}^n$, {and as the Campanato space $\mathcal{C}_{\alpha}(X)$
does on a space $X$ of homogenenous type in the sense of
Coifman and Weiss.}

In \cite{NS}, Nagel and Stein developed the product $L^p$-theory
$(1 < p < \infty)$ in the setting of
Carnot-Carath\'eodory spaces formed by vector fields satisfying
H\"{o}rmander's $m$-finite rank condition, {where $m\geq 2$ is a positive integer}. The
Carnot-Carath\'eodory spaces studied in~\cite{NS} are spaces of
homogeneous type with a regular quasi-metric $d$ and a
measure~$\mu$ satisfying the conditions $\mu \big (B(x, sr)
\big ) \sim s^{m+2} \mu \big (B(x,r) \big )$ for $s\geq 1$
and $\mu \big (B(x, sr) \big ) \sim s^4\mu \big (B(x,r)\big )$ for
$s\leq 1.$ These conditions on the measure are weaker than
property~\eqref{regular} but are still stronger than the
original doubling condition~\eqref{eqn:doubling condition}.

Motivated by the work of Nagel and Stein, Hardy spaces via
Littlewood-Paley theory were developed by the first author,
M\"uller and Yang~\cite{HMY1, HMY2} on spaces of homogeneous
type with a regular quasi-metric and a measure satisfying some
additional conditions. To be precise, let $(X,d,\mu)$ be a space of homogeneous type
where the quasi-metric $d$ satisfies the H\"older regularity
property~\eqref{smetric}, and the measure $\mu$ satisfies the
doubling condition~\eqref{eqn:doubling condition} and the
\emph{reverse doubling condition}; that is, there are constants
$\kappa \in (0, \omega ]$ and $c \in (0, 1]$ such that
\begin{equation}\label{eqn:reverse doubling}
    c \lambda^\kappa \mu \big ( B (x, r) \big )
    \leq \mu \big ( B(x, \lambda r) \big )
\end{equation}
for all $x \in X$, $r$ with $0 < r < \displaystyle\sup_{x, y \in X} d (x, y) / 2$
and $\lambda$ with $1\leq \lambda < \displaystyle\sup_{x, y \in X} d (x, y) / 2r.$
The first author, M\"uller, and Yang observed in~\cite{HMY1,HMY2} 
that Coifman's construction of an approximation to the identity
still works on spaces of homogeneous type $(X, d, \mu)$ with
these properties.
They also showed how to define the corresponding test functions
of type $(x_0,r,\beta, \gamma )$. {Their definition is very
similar to Definition~\ref{def-of-test-func-space1} above,
except that one power of $\big (r+d(x,x_0)\big )$ in the
denominator is replaced by $\big ( \mu \big ( B(x,r)\big ) +
\mu  \big (B(x,d(x,x_0) ) \big ) \big )$. Also, their
definition is identical to the definition of test functions
needed in our setting, Definition~\ref{def-of-test-func-space},
except that in their case $\beta\in [0,\theta]$ where $\theta$
is the regularity exponent of the metric, while in our case
$\beta\in[0,\eta]$ where $\eta$ is the H\"older exponent of the
wavelets.}

Applying Coifman's approximation to the identity and a proof
similar to the one in~\cite{H1, H2, HS},
the first author, M\"uller, and Yang proved that a discrete
Calder\'on reproducing formula still holds on $(X,d,\mu)$ when
the quasi-metric $d$ satisfies the regularity
condition~\eqref{smetric} and the measure $\mu$ satisfies the
doubling condition~\eqref{eqn:doubling condition} and the
reverse doubling condition~\eqref{eqn:reverse doubling}. As a
consequence, the Hardy spaces defined via the Littlewood-Paley
theory were established for such spaces of homogeneous type
and, moreover, these Hardy spaces have atomic decompositions.
See~\cite{HMY1} for more details.

However, there are settings for which the reverse doubling
condition is not available. One specific example of such a
space of homogeneous type appears in the Bessel setting treated
by Muckenhoupt and Stein~\cite{MuS}. They studied the Bessel
operator
\[
    \Delta_{\lambda}
    = -\frac{d}{dx^2} - \frac{2\lambda}{x}\frac{d}{dx},
    \quad\quad
    \lambda\in\big(-{1}/{2},\infty\big), \; x\in (0,\infty),
\]
with the underlying space $(X,d,\mu) = \big ((0,\infty),
|\cdot|, x^{2\lambda}\,dx\big )$. The corresponding Hardy space
was studied in \cite{BDT} and the weak factorization was
obtained in \cite{DLWY}. We note that the measure
$x^{2\lambda}\,dx$ is doubling when
$\lambda\in(-1/2,\infty)$, however when
$\lambda\in(-1/2,0)$ the measure does not satisfy a reverse
doubling condition. We also note that we cannot change the
metric twice as in~\cite{MS1}, for if we did we would be
changing the whole setting, including the Bessel operator in
question.

In \cite{HLW}, the first, second and fourth authors developed a
theory of Hardy spaces $H^p$ and $\bmo$ on spaces of
homogeneous type in the sense of Coifman and Weiss, with only
the original quasi-metric~$d$ and a (Borel-regular) doubling
measure~$\mu$, in both the one-parameter and product settings.
A crucial idea in~\cite{HLW} was to use a square-function
characterization where the square function was built using the
Auscher-Hyt\"onen orthonormal wavelet basis on spaces of
homogeneous type~\cite{AH,AH2}. 
In the current paper we provide an atomic decomposition for
$H^p(\widetilde{X})\cap L^q(\widetilde{X})$ for each $q$ with
$1<q<\infty$, for $\widetilde{X}=X_1\times X_2$ with $X_i$ a
space of homogenenous type in the sense of Coifman and Weiss
for $i = 1$, 2. This atomic decomposition is completely
independent of any wavelet bases and reference dyadic grids on
$X_i$ for $i=1,2$ used to define $H^p(\widetilde{X})$. As a
consequence of the main result of this paper, the
$H^p(\widetilde{X})$ spaces defined in~\cite{HLW} via a
particular Auscher-Hyt\"onen wavelet basis are independent {not
only of the chosen wavelet bases, but also of the choice of
reference dyadic grids.}

\section{Preliminaries}\label{sec:preliminaries}
\setcounter{equation}{0}

In this section, we recall first Hyt\"onen and Kairema's systems of
dyadic cubes~\cite{HK}, second Auscher and Hyt\"onen's orthonormal
basis~\cite{AH} paying close attention to their underlying reference dyadic grids, and third  the sets of test functions and distributions developed in \cite{HLW} in both the one-parametr and the product settings. We recall that the Auscher and Hyt\"onen wavelets in both one-parameter and product setting are suitable test functions. These are all necessary ingredients in the definition of product Hardy spaces introduced in \cite{HLW}   that we present in Section~\ref{sec:productHp}.

\subsection{Systems of dyadic cubes}
\label{sec:dyadiccubes}

We now describe the Hyt\"onen and Kairema \cite{HK} families of dyadic ``cubes" built on  geometrically doubling quasi-metric spaces. A quasi-metric space $(X,d)$ is \emph{geometrically doubling}
if  there exists a natural number $N$ such that any quasi-metric ball  $B(x,r)$ can be covered with no more than $N$ balls of half the radius. Coifman and Weiss \cite{CW} showed that  spaces of homogeneous type $(X,d,\mu )$ are geometrically doubling quasi-metric spaces. The Hyt\"onen-Kairema construction builds on seminal work of Guy David
\cite{Da},  Christ \cite{Chr}, and Sawyer and Wheeden \cite{SW}.

\begin{theorem}[\cite{HK}, Theorem 2.2] 
\label{thm:dyadiccubes}
  Given a geometrically doubling quasi-metric space $(X,d)$, let $A_0>0$ denote the quasi-triangle constant for the metric $d$.
    Given constants $c_0$ and $C_0$ with $0 < c_0 \leq C_0 < \infty$, and
    constant $\delta\in(0,1)$ satisfying
    \begin{equation}\label{eqn:testconditionforcubes}
        12 A_0^3 C_0\delta
        \leq c_0.
    \end{equation}
    Given a set of points $\{z_\alpha^k\}_{\alpha \in \mathscr{A}_k}$,  where $\mathscr{A}_k$ is a countable set of indices for each $k\in\mathbb{Z}$, with the properties that
    \begin{equation}\label{eqn:sparseproperty}
        d(z_\alpha^k,z_\beta^k)
        \geq c_0\delta^k\
            (\alpha\not=\beta),\hskip1cm
        \min_{\alpha
    \in \mathscr{A}_k} d(x,z_\alpha^k)
        < C_0\delta^k,
            \qquad \text{for all $x\in X$},
    \end{equation}
    $($called a $(c_0, C_0)$-\emph{maximal set of $\delta^k$-separated points}$)$,
    we can construct families of sets $\widetilde{Q}_\alpha^k
    \subseteq Q_\alpha^k \subseteq
    \overline{Q}_\alpha^k$  $($called open, half-open and closed
    \emph{dyadic cubes}$)$, such that:
    \begin{align}
        & \widetilde{Q}_\alpha^k \mbox{ and } \overline{Q}_\alpha^k
            \mbox{ are the interior and closure of } Q_\alpha^k, \mbox{ respectively};\\
        &  (\mbox{\emph{Nested family}}) \ \ \mbox{if } \ell\geq k, \mbox{ then either } Q_\beta^\ell\subseteq
            Q_\alpha^k \mbox{ or } Q_\alpha^k
            \cap Q_\beta^\ell=\emptyset  ;\label {DyadicP1}\\
        &  (\mbox{\emph{Disjoint union}})  \ \ X = \bigcup_{\alpha
    \in \mathscr{A}_k} Q_\alpha^k  \qquad
            \text{for all $k\in\mathbb{Z}$};\label {DyadicP2} \\
        & (\mbox{{\emph{Inner and outer balls}}}) \ \ B(z_\alpha^k,c_1\delta^k)\subseteq Q_\alpha^k\subseteq
            B(z_\alpha^k,C_1\delta^k),\ \  \mbox{where } c_1 := (3 A_0^2)^{-1}c_0 \label{prop_cube3}\\
         &   \mbox{and}\  C_1 := 2A_0C_0; \nonumber  \\
        &\mbox{if } \ell \geq k \mbox{ and } Q_\beta^\ell\subseteq Q_\alpha^k,
            \mbox{ then } B(z_\beta^\ell,C_1\delta^\ell)\subseteq
            B(z_\alpha^k,C_1\delta^k).\label {DyadicP4}
    \end{align}
    The open and closed cubes $\widetilde{Q}_\alpha^k$ and
    $\overline{Q}_\alpha^k$ depend only on the points
    $z_\beta^\ell$ for $\ell\geq k$. The half-open cubes
    $Q_\alpha^k$ depend on $z_\beta^\ell$ for $\ell\geq
    \min(k,k_0)$, where $k_0\in\mathbb{Z}$ is a preassigned
    number entering the construction.
\end{theorem}
We denote by $\mathscr{D}$ the family of dyadic cubes $\{Q^k_{\alpha}\}_{k\in\mathbb{Z}, \mathscr{A}_k}$ as in Theorem~\ref{thm:dyadiccubes}. We will refer to $\mathscr{D}$ as a \emph{Hyt\"onen-Kairema dyadic system} or \emph{grid on $X$}.  We will refer to any cube $Q^{k+1}_{\beta}\in\mathscr{D}$ that is contained in  $Q^k_{\alpha}\in\mathscr{D}$ as a \emph{child of $Q^k_{\alpha}$}. Note that every cube has at least one child and no more than $M$ children where $M$ is a uniform bound determined by the geometric doubling condition.

The existence of countable sets of  separated points as in~\eqref{eqn:sparseproperty}
 is ensured by the geometric doubling property of the quasi-metric space $(X,d)$.
For a given Hyt\"onen-Kairema dyadic system of cubes, we will call $c_0$ and $C_0$ the \emph{separation constants} of the system, $c_1$ and  $C_1$ the \emph{dilation constants} of the system, and $\delta$ the \emph{base side length} of the cubes,  collectively these will be called  \emph{structural constants} of the dyadic system or of the dyadic grid.
Note that in \eqref{prop_cube3},  as it should be,  the
dilation constants $c_1$ and $C_1$, determining the radii of the inner
and outer balls for each cube, satisfy  $0<c_1<C_1$,  since by
hypothesis the separation constants $0< c_0\leq C_0$, but \emph{a priori}  $C_1$ is not
necessarily less than one. We will sometimes denote by $B'_Q$ and $B''_Q$ the inner and outer balls of a dyadic cube $Q$.

Given a cube $Q^k_\alpha$, we denote the quantity $\delta^k$ by
$\ell(Q^k_\alpha)$, by analogy with the sidelength of a
Euclidean cube. We  define the  dilate $\lambda Q_{\alpha}^k$ of a dyadic cube
to be the  $\lambda$-dilate of its outer ball. That is,  for~$\lambda>0$,
\[
 \lambda Q_{\alpha}^k := B(z_\alpha^k,\lambda C_1\delta^k).
 \]

 By construction, the cubes are unions of quasi-metric balls,
hence in the setting of a space of homogeneous type, the cubes
are measurable.
 In the presence of a doubling measure $\mu$ (doubling with respect to balls)  the measure $\mu$  is ``doubling" with respect to Hyt\"onen-Kairema cubes. More precisely, 
 \begin{equation}\label{doubling-dilate-cubes1}
 \mu (\lambda Q^k_{\alpha})\leq \Big (\lambda {C_1}/{c_1}\Big  )^{\omega} \mu \big (B(z^k_1, c_1\delta^k) \big ) \leq \lambda^{\omega} \Big ({C_1}/{c_1}\Big )^{\omega} \mu (Q^k_{\alpha}).
 \end{equation}

Where the first inequality is a consequence of the doubling property~\eqref{eqn:upper dimension}, and the second simply because the inner ball of a cube sits inside the cube. Also note that by construction, specifically properties~\eqref{prop_cube3} and~\eqref{eqn:testconditionforcubes}, the ratio $C_1/c_1 = 6A_0^3 (C_0/c_0)\leq \delta^{-1}/2$, where $\delta\in (0,1)$ is the base side length of the cubes. Potentially the base side length parameter $\delta$ can be arbitrarily small, therefore making the upper bound in~\eqref{doubling-dilate-cubes1} arbitrarily  large.  Also, the ratio $C_1/c_1$ maybe under control, but that does not imply the outer dilation constant cannot be arbitrarily large, since a priori we could allow the inner dilation constant to also be arbitrarily large. These can be problematic,  therefore we single out the dyadic systems that do not suffer from these problems, and we call them  \emph{regular families of dyadic systems} or \emph{grids}.

\begin{definition}[Regular families of  dyadic systems]\label{def:regular-dyadic-grids}
Given a geometric doubling quasi-metric space $(X,d)$.
 A family  $\{\mathscr{D}^b\}_{b\in\mathscr{B}}$ of Hyt\"onen-Kairema  dyadic systems on $X$  is \emph{regular} if
the outer dilation constants $\{C_1^b\}_{b\in\mathscr{B}}$ and the ratio of the outer and inner dilation constants $\{C_1^b/c_1^b\}_{b\in\mathscr{B}}$ of the systems in the family are uniformly bounded by  constants   depending  only on the quasi-triangle constant $A_0$ of the quasi-metric $d$.
\end{definition}

In the proof of the main theorem in Section~\ref{Proof-Main-Theorem} we will have atomic decompositions in the setting  of a product of spaces of homogenenous type, $X_1\times X_2$,   with  atoms $a$ associated to dyadic grids $\mathcal{D}^a_i$ belonging to   regular families  on $(X_i, d_i,\mu_i)$ for $i=1,2$.  Often we will estimate the measure of  dilates of cubes $Q_i\in\mathscr{D}^a_i$ as in~\eqref{doubling-dilate-cubes1}, and will  say ``by doubling"
\begin{equation}\label{doubling-dilate-cubes}
\mu_i(\lambda Q_i)\lesssim \lambda^{\omega_i} \mu_i (Q_i).
\end{equation}
 The $\lesssim$  will only depend on the geometric constants of the spaces$X_i$ for $i=1,2$, but not on the structural constants of the dyadic grids, because  $\mathcal{D}_i^a$ belong to a regular family of dyadic systems. Elsewhere in the proof of the main theorem the outer dilation constants $ C_1^i$ will come into the estimates, and we will also need them to  be uniformly bounded by a constant depending only on the geometric constants of $X_i$ for $i=1,2$.

\subsection{Orthonormal basis, reproducing formula, and cut-off functions}
\label{sec:onb}

Auscher and Hyt\"onen \cite{AH} constructed a remarkable
orthonormal basis of $L^2(X)$, where $(X,d,\mu)$ is a space of
homogeneous type. To state their result, we first recall the
{\it reference dyadic points} $x_\alpha^k$ as follows.

Let $\delta$ be a fixed small positive parameter ($\delta \leq
10^{-3}A_0^{-10}$, where $A_0$ is the quasi-triangle constant
of the quasi-metric $d$). For $k=0$, let
$\mathscr{X}^0:=\{x_\alpha^0\}_{\alpha  \in \mathscr{A}_0}$ be
a maximal set  of  1-separated points in~$\XX$. Inductively,
for $k\in\mathbb{Z}_+$, let
$\mathscr{X}^k:=\{x_\alpha^k\}_{\alpha
    \in \mathscr{A}_k}
\supseteq \mathscr{X}^{k-1}$ and
$\mathscr{X}^{-k}:=\{x_\alpha^{-k}\}_{\alpha
    \in \mathscr{A}_{-k}} \subseteq
\mathscr{X}^{-(k-1)}$ be maximal $\delta^k$- and
$\delta^{-k}$-separated collections in $\mathscr{X}^{k-1}$ and
$\mathscr{X}^{-(k-1)}$, respectively. The families
$\mathscr{X}^k$ have the separation properties required in
Theorem~\ref{thm:dyadiccubes} for the construction of cubes, with separation constants $c_0=1$, $C_0=2A_0$,  base side length  the given $\delta\in (0,1)$, and
with the additional property that  $\mathscr{X}^{k}\subseteq
\mathscr{X}^{k+1}$ for $k\in\mathbb{Z}$. We denote the
corresponding cubes by $Q_{\alpha}^k$, and the dyadic system $\mathscr{D}^W$.  We will call $\mathscr{D}^W$, the \emph{reference dyadic system} or \emph{grid} underlying the wavelets.

 A randomization $\mathscr{X}^k(\omega )$ of the
families $\mathscr{X}^k$, as discussed in~\cite{HK,HM}, has the  separation properties for
each random parameter $\omega$ (in a certain space $\Omega$
equipped with a probability measure $\mathbb{P}_{\omega}$) needed to construct the
dyadic cubes $Q^k_{\alpha}(\omega)$ according to
Theorem~\ref{thm:dyadiccubes}. However, in \cite[Theorem 2.11]{AH})  they modify the construction
so that the randomized dyadic cubes
$Q^k_{\alpha}(\omega)$ have uniform (in the random parameter $\omega\in\Omega$) inner and
outer balls constants (in fact $c_1(\omega)= \frac16A_0^{-5}$ and $C_1(\omega)=
6A_0^4>1$ for all $\omega\in \Omega$), and   an
 additional ``small
boundary layer property" on average with respect to the
probability measure introduced by the randomization
\cite[Equation (2.3)]{AH}. It is in measuring the smallness of
the boundary layer that a small parameter $\eta>0$ appears, dependent only on the geometric constants of the space $X$.
This parameter $\eta$ is the H\"older regularity of the wavelets defined
in Theorem~\ref{thm:AH_orthonormal_basis}. In this randomized
construction, the reference dyadic point $x_{\alpha}^k$ may
also be viewed as the center of the random cubes
$Q^k_{\alpha}(\omega)$ for all $\omega$ belonging to the
parameter space $\Omega$. For the details of this beautiful
construction see \cite[Section 2]{AH}.

Now denote $\mathscr{Y}^{k}:=\mathscr{X}^{k+1}\backslash
\mathscr{X}^{k}$, and relabel the points $x_\alpha^k$ that
belong to~$\mathscr{Y}^k$ as $y_\alpha^k$, where $\alpha\in
\mathscr{A}_{k+1}\backslash \mathscr{A}_k$ and $k\in
\mathbb{Z}$. To each such point $y^k_{\alpha}$, Auscher and
Hyt\"onen associate a  function $\psi^k_{\alpha}$
that is almost supported near $y_{\alpha}^k$ at scale
$\delta^k$ (these functions are not compactly supported, but
have exponential decay). Also note that to each
Hyt\"onen-Kairema cube $Q_{\alpha}^k$ there corresponds the
point $x_{\alpha}^k$ and to each of the children of
$Q^k_{\alpha}$  there correspond other points
$x_{\beta}^{k+1}$, one of which coincides by construction with
$x_{\alpha}^k$.  Thus the number of indices $\alpha$ in
$\mathscr{A}_{k+1}\backslash \mathscr{A}_k$ corresponding to
$Q_{\alpha}^k$ is exactly  $N(Q_{\alpha}^k) - 1$, where
$N(Q_{\alpha}^k)$ denotes the number of children of
$Q_{\alpha}^k$. This is the right number of wavelets we will
need per cube if our intuition is guided by tensor product
wavelets in $\mathbb{R}^n$, or Haar functions on spaces of
homogeneous type based on Hyt\"onen-Kairema cubes, as
constructed for example in  \cite{KLPW}.  Later on we will
write $\alpha\in \mathscr{Y}^k$ meaning
$\alpha\in \mathscr{A}_{k+1}\backslash \mathscr{A}_k$.

We now state the theorem describing precisely the wavelets of Auscher and Hyt\"onen.

\begin{theorem}
[\cite{AH}, Theorem 7.1]
\label{thm:AH_orthonormal_basis}
    Let $(\XX,d,\mu)$ be a space of homogeneous type with
    quasi-triangle constant~$A_0$,  with  reference dyadic system of cubes $\mathscr{D}^W=\{Q^k_{\alpha}\}_{k\in\mathbb{Z}, \alpha\in\mathscr{A}^k}$ that has  base side length $\delta\in (0,1)$ and  small boundary layer parameter $\eta \in (0,1]$.
    Let
    \begin{equation}\label{eqn:defn_of_a}
        a
        := (1 + 2\log_2 A_0)^{-1}.
    \end{equation}
    There exists an orthonormal  basis
    $\{\psi_\alpha^k\}_{k\in\mathbb{Z}, \alpha\in \mathscr{A}_{k+1}\setminus\mathscr{A}_k}$ of $L^2(X)$ and finite constants $C>0$ and $\nu >0$
    such that for all $k\in\mathbb{Z}$ and $\alpha\in\mathscr{A}_{k+1}\setminus\mathscr{A}_k$ each function
    $\psi^k_{\alpha}$  satisfies the following conditions:
     {\rm  (i)} $\psi^k_{\alpha}$ is  centered at $y_\alpha^k\in  \mathscr{Y}^k$;
       {\rm  (ii)} $\psi^k_{\alpha}$ has  exponential decay determined by parameters $a$ and $\nu$, namely for all $x\in X$,
    \begin{equation}\label{eqn:exponential_decay}
        |\psi_\alpha^k(x)|
        \leq \frac{C}{ \sqrt{\mu \big (B(y_\alpha^k,\delta^k)\big )}}
            \exp\Big(-\nu\Big( \frac{d(y^k_\alpha,x)}{\delta^k}\Big)^a\Big);
    \end{equation}
      {\rm (iii)} $\psi^k_{\alpha}$ has {\rm (local)}  H\"older regularity with H\"older exponent $\eta$, namely  for all $x,y\in X$ such that  $d(x,y)\leq \delta^k$,
    \begin{equation}\label{eqn:Holder_regularity}
        |\psi_\alpha^k(x)-\psi_\alpha^k(y)|
        \leq \frac{C}{\sqrt{\mu \big (B(y_\alpha^k,\delta^k)\big )}}
            \Big( \frac{d(x,y)}{\delta^k}\Big)^\eta
            \exp\Big(-\nu\Big( \frac{d(y^k_\alpha,x)}{\delta^k}\Big)^a\Big);
    \end{equation}
  {\rm (iv)} $\psi^k_{\alpha}$ has vanishing mean, namely
    \begin{equation}\label{eqn:cancellation}
        \int_X \psi_\alpha^k(x)\,d\mu(x) = 0.
    \end{equation}

\end{theorem}
In Theorem~\ref{thm:AH_orthonormal_basis},  the constants $C $, $\nu$,
$\eta$, and $\delta$ are  independent of $k$, $\alpha$,
and~$y_\alpha^k$, they only depend on the geometric constants of the space $X$: quasi-triangle inequality, doubling constant, and upper dimension.
The constant $\delta\in (0,1)$,
 determining the side length of the reference dyadic cubes, is a fixed small parameter, more precisely, $\delta \leq 10^{-3} A_0^{-10}$.

In what follows, we refer to the functions $\psi_\alpha^k$ as \emph{Auscher-Hyt\"onen
wavelets} or simply \emph{wavelets}. The wavelet expansion, convergent in the sense of $L^2(X)$, is given by
\begin{equation}\label{eqn:AH_reproducing formula}
    f(x)
    = \sum_{k\in\mathbb{Z}}\sum_{\alpha \in \mathscr{Y}^k}
        \langle f,\psi_{\alpha}^k \rangle \psi_{\alpha}^k(x).
\end{equation}
Here $\langle f,g\rangle:= \int_X f(x)\overline{g(x)} d\mu(x)$ denotes the $L^2$-pairing.
The Auscher-Hyt\"onen  wavelets
$\{\psi_\alpha^k\}_{k\in\mathbb Z,\alpha\in \mathscr{Y}_k}$
 form an unconditional basis of $L^q(X)$ for all $q$ with $1 < q < \infty$;
see \cite[Corollary~10.4]{AH}. Therefore, the reproducing
formula~\eqref{eqn:AH_reproducing formula} also holds for $f\in
L^q(X)$. Note that for the reproducing
formula~\eqref{eqn:AH_reproducing formula} to hold, it suffices
that the measure~$\mu$ is Borel regular; see
addendum~\cite{AH2}. Also note that it is possible  to build different wavelets based on the same reference dyadic points \cite{AH}.

In the Auscher-Hyt\"onen  construction of wavelets, the reference dyadic grids $\mathscr{D}^W$  form a regular family of dyadic systems according to Definition~\ref{def:regular-dyadic-grids}, because the outer dilation constants and the ratio of the outer and inner dilation constants are respectively, $C_1=
6A_0^4>1$ and   $C_1/c_1= 36A_0^9$, for all the systems in the family.

For a general space of homogeneous type, the H\"older exponent~$\eta$ of the wavelets is bounded above by a constant $\eta_0$  ($0<\eta  < \eta_0$) 
that only depends  on the geometric parameters of the geometrically doubling space $(X,d)$ \cite{AH}. The constant $\eta_0$ is  usually  much smaller than one, even in the case of metric spaces.  In \cite{HT}, Hyt\"onen and Tapiola presented a different construction of the metric wavelets  that  allows-to-obtain H\"older-regularity for any exponent $\eta<1$, strictly below but arbitrarily close to one.

The  wavelets' regularity parameter $\eta$  enters into the definition of the Hardy spaces $H^p(X)$ on  spaces of homogeneous type $(X,d,\mu)$. In particular, $\eta$ together with an upper dimension $\omega$ of the doubling measure $\mu$, determines the range of $p$ for which the Hardy space is defined, namely ${\omega}/(\eta +\omega) < p \leq 1$.  The larger $\eta$ is, the smaller $p$ can be chosen. A similar phenomenon occurs  for the Hardy spaces on product  spaces of homogeneous type, as pointed out in \cite{HLW}, see also Section~\ref{sec:productHp}. This is parallel to the theory on $\mathbb{R}^n$ where the theory of $H^p$-spaces with just the cancellation property is limited to $n/(n+1) <p\leq 1$, and to access smaller values of $p$, the test functions must have larger number of vanishing moments, unavailable in general spaces of homogeneous type.

{The construction of the wavelets hinges on the construction of certain ``splines" on $X$ defined using the probability measure $\mathbb{P}_{\omega}$ on the space $\Omega$.  For every $(k,\alpha)\in \mathbb{Z}\times \mathscr{Y}^k$ Auscher and Hyt\"onen \cite[Equation (3.1)]{AH} define the spline function $s^k_{\alpha}:X\to [0,1]$ by
$$ s^k_{\alpha}(x):= \mathbb{P}_{\omega}\big (x\in \overline{Q}^k_{\alpha}(\omega)\big ).$$
The spline function $s^k_{\alpha}$ are bumps supported on a ball centered at $x^k_{\alpha}$ and radius roughly $\delta^k$, and they satisfy some interpolation, reproducing, and H\"older-continuity properties, described precisely in
\cite[Theorem 3.1]{AH}.}

The splines in turn were used in~\cite{HLW} to construct
smooth cut-off functions.

\begin{lemma}[\cite{HLW}, Lemma 3.8]\label{lem:cut-off-functions}
    For each fixed $x_0\in X$ and $R_0\in (0,\infty)$, there exists
    a smooth cut-off function $h(x)$ such that $0\leq h(x)\leq 1$,
    \begin{equation*}
        h(x)
        \equiv 1 \;\; \mbox{when $x\in B(x_0,R_0/4)$},
        \quad\quad
        h(x)
        \equiv 0 \;\; \mbox{when $x\in B(x_0,A_0^2R_0)^c$},
    \end{equation*}
    and there exists a positive constant $C$, independent of $x_0$,
    $R_0$, $x$, and $y$ (dependent only on geometric constants of the space $X$) such that
    \begin{equation*}
        |h(x)-h(y)|
        \leq C \Big ({d(x,y)}/{R_0}\Big )^{\eta}.
    \end{equation*}
\end{lemma}

Note that the cut-off functions satisfy a global H\"older
regularity condition with the same exponent $\eta$ as the
wavelets in Theorem~\ref{thm:AH_orthonormal_basis}. We will use
these smooth cut-off functions on $X$ in the proof of the key
decomposition Lemma~\ref{lemma-decomposition} for the
wavelets.

\subsection{Test function spaces and distributions}\label{sec:testfunctions}

We now recall the definition of the  test functions and
distributions on $(X,d,\mu)$ that will enter into the
definition of the Hardy spaces on product of spaces of
homogeneous type. In particular, we observe that  the normalized Auscher-Hyt\"onen
wavelets are test functions.

Let $V_r(x) := \mu \big (B(x,r)\big ) \;\; \mbox{for $x\in X$,
$r>0$ and} \;\; V(x,y):=\mu\big (B(x,d(x,y))\big ) \;\;
\mbox{for $x,y\in X$}.$

\begin{definition}[Test functions \cite{HLW}, Definition~3.1]
    \label{def-of-test-func-space}
    Fix $x_0\in X$, $r > 0$, $\beta\in(0,\eta]$ where $\eta
    \leq 1$ is the H\"older regularity exponent from
    Theorem~\ref{thm:AH_orthonormal_basis}, and $\gamma > 0$. A
    $\mu$-measurable function $f$ defined on~$X$ is said to be
    a {\it test function of type $(x_0,r,\beta,\gamma)$
    centered at $x_0\in X$} if $f$ satisfies the following
    three conditions:
    \begin{enumerate}
        \item[(i)] \textup{(Size condition)} There is a constant $C > 0$ such that for all $x\in
            X$ 
            \[
                |f(x)|
                \leq C \,\frac{1}{V_{r}(x_0) + V(x,x_0)}
                \Big(\frac{r}{r + d(x,x_0)}\Big)^{\gamma}.
            \]

        \item[(ii)] \textup{(Local H\"older regularity
            condition)} There is a constant $C > 0$ such that for all $x, y\in X$ with $d(x,y)
            < (2A_0)^{-1}(r + d(x,x_0))$  
            \[
                |f(x) - f(y)|
                \leq C \Big(\frac{d(x,y)}{r + d(x,x_0)}\Big)^{\beta}
                \frac{1}{V_{r}(x_0) + V(x,x_0)} \,
                \Big(\frac{r}{r + d(x,x_0)}\Big)^{\gamma}.
            \]

        \item[(iii)] \textup{(Cancellation condition)}
            \[
                \int_X f(x) \, d\mu(x)
                = 0.
            \]
    \end{enumerate}
\end{definition}

These test functions generalize the test functions in
Definition~\ref{def-of-test-func-space1}, in the case when $\mu
(B'(x,r))\sim r$ and the quasi-metric $d'$ has the H\"older
regularity~\eqref{smetric} with exponent~$\theta$. Notice that
in this case  $\big (V_{r}(x_0) + V(x,x_0)\big ) \sim \big
(r+d'(x,x_0)\big )$, and both definitions coincide. One can
also compare to corresponding definitions in~\cite{HMY1,HMY2}
in the case when the quasi-metric $d$ satisfies the H\"older
regularity \eqref{smetric} with exponent $\theta$ and the
measure satisfies the doubling condition~\eqref{eqn:doubling
condition} and the reverse doubling
condition~\eqref{eqn:reverse doubling}. In these cases the only
difference is that $\beta$ is in $(0,\theta ]$ instead of being
in $(0,\eta ]$; otherwise the definitions are identical.

Let  $G(x_0,r,\beta,\gamma)$ denote the set of all test
functions of type $(x_0,r,\beta,\gamma)$. The norm on
$G(x_0,r,\beta,\gamma)$ is defined by
$
    \|f\|_{G(x_0,r,\beta,\gamma)}
    := \inf\{C>0:\ {\rm(i)\  and \ (ii)}\ {\rm hold} \}.
$

Now fix $x_0\in X$. Let $G(\beta,\gamma) :=
G(x_0,1,\beta,\gamma)$. It is easy to check that
$G(x_1,r,\beta,\gamma) = G(\beta,\gamma)$ with equivalent norms
for each fixed $x_1\in X$ and $r > 0$. Furthermore, it is also
easy to see that if $0 < \beta \leq \eta$ then
$G(\eta,\gamma)\subset G(\beta,\gamma)$ and $G(\eta,\gamma)$ is
a Banach space with respect to the norm on $G(\eta,\gamma)$.

For $0<\beta\leq \eta$, let $\GGs(\beta,\gamma)$ be the completion
of the space $G(\eta,\gamma)$ in the norm of
$G(\beta,\gamma)$. For $f\in \GGs(\beta,\gamma)$, we define
$\|f\|_{\GGs(\beta,\gamma)} := \|f\|_{G(\beta,\gamma)}$. The spaces $\GGs(\beta,\gamma)$ are nested, if $0<\beta\leq \beta'$ and $0<\gamma\leq \gamma'$ then $\GGs(\beta',\gamma') \subset \GGs(\beta,\gamma)$.

The distribution space $(\GGs(\beta,\gamma))'$ is
the set of all bounded linear functionals on 
$\GGs(\beta,\gamma)$. We denote by $\langle f,h\rangle$ the
natural pairing of elements $h\in \GGs(\beta,\gamma)$ and $f\in
(\GGs(\beta,\gamma))'$.

The normalized Auscher-Hyt\"onen wavelets are test functions in $G(\eta,\gamma)$ for any $\gamma>0$.
Later on we will take advantage of this fact, inherited from the exponential decay of the wavelets, and choose $\gamma$ to be large enough.
The reproducing formula holds in the space of test
functions  and distributions with parameters $\beta',\gamma'\in (0,\eta)$. More precisely, the following
propositions hold.

\begin{prop}[\cite{HLW}, Theorem~3.3]
    \label{prop wavelet is test function}
    Suppose $\{\psi_\alpha^k\}_{k\in\mathbb{Z}, \alpha\in
    \mathscr{Y}_k}$ is an orthonormal basis as in
    Theorem~\ref{thm:AH_orthonormal_basis}, with H\"older
    regularity of order~$\eta$. Then for each $k\in\mathbb{Z}$,
    $\alpha\in\mathcal{Y}^k$, and $\gamma>0$, the normalized wavelet
    $\psi_\alpha^k(x)/ \sqrt{\mu \big (B(y_\alpha^k,\delta^k)
    \big)}$ belongs to the set $G(y^k_{\alpha},\delta^k,\eta,\gamma)$ of test functions of type
    $(y_{\alpha}^k,\delta^k,\eta,\gamma)$ centered at $y^k_{\alpha}\in X$.

\end{prop}

\begin{prop}[\cite{HLW}, Theorem~3.4]
    \label{thm reproducing formula test function}
    Suppose that $f\in \GGs(\beta,\gamma)$ with
    $\beta$, $\gamma \in (0,\eta)$.      Then the reproducing
    formula~\eqref{eqn:AH_reproducing formula} holds in
    $\GGs(\beta',\gamma')$ for each $\beta'\in(0,\beta)$ and
    $\gamma'\in(0,\gamma)$.
\end{prop}

As a consequence, the reproducing formula also
holds for distributions.

\begin{corollary}[\cite{HLW}, Corollary~3.5]
    \label{coro reproducing formula distribution}
    The reproducing formula~\eqref{eqn:AH_reproducing formula}
    holds in $(\GGs(\beta',\gamma'))'$, when $\beta', \gamma' \in (0,\eta)$.
\end{corollary}

\subsection{Product setting}
Consider the product setting $(X_1,d_1,\mu_1)\times
(X_2,d_2,\mu_2)$, where  each $(X_i,d_i,\mu_i)$, $i = 1$, 2, is a
space of homogeneous type as defined in Section~1. For $i = 1$,
2, let $A_0^{(i)}$ be the constant in the quasi-triangle
inequality~\eqref{eqn:quasitriangleineq}, let $C_{\mu_i}$ be
the doubling constant as in inequality~\eqref{eqn:doubling
condition}, and let $\omega_i$ be an upper dimension of~$X_i$
as in inequality~\eqref{eqn:upper dimension}.
By Theorem~\ref{thm:AH_orthonormal_basis},   on each space of homogeneous type $(X_i,d_i,\mu_i)$ for $i=1,2$,  there
is a wavelet basis~$\{\psi^{k_i}_{\alpha_i}\}_{k_i\in\mathbb{Z},\alpha_i\in\mathscr{Y}^{k_i}}$, with H\"older
regularity  exponent~$\eta_i\in (0,1]$ as in
inequality~\eqref{eqn:Holder_regularity}, and reference dyadic grid $\mathscr{D}_i^W$ with dilation constants $c^i_1$, $C^i_1$ and their ratio $C^i_1/c^i_1$ depending uniformly on $A_0^{(i)}$.

The spaces of product test functions and distributions on
the product space $X_1\times X_2$ are defined as follows.

\begin{definition}[Product test functions \cite{HLW}, Definition~3.9]
    \label{def-of-test-func-space-product} Suppose
    $(x_0,y_0)\in X_1\times X_2$ and $r_i > 0$, take $\beta_i$
    so that $0 < \beta_i \leq \eta_i$, and take $\gamma_i > 0$,
    for $i = 1$, $2$. A function $f(x,y)$ defined on $X_1\times X_2$
    is said to be a {\it test function of type}
    $(x_0,y_0;r_1,r_2;\beta_1,\beta_2;\gamma_1,\gamma_2)$ if
    the following conditions hold. First, for each fixed $y \in
    X_2,$ $f(x,y),$ as a function of the variable $x$, is a
    test function in $G(x_0,r_1,\beta_1,\gamma_1)$ on $X_1$.
    Second, for each fixed $x \in X_1$, $f(x,y)$, as a function
    of the variable~$y$, is a test function in
    $G(y_0,r_2,\beta_2,\gamma_2)$ on $X_2$. Third, the
    following mixed conditions are satisfied:
\begin{enumerate}
    \item[(i)] (Size condition in $y$ variable) For all  $y\in X_2$,
    $$\Vert f(\cdot,y)\Vert_{G(x_0,r_1,\beta_1,\gamma_1)}\leq
    C \frac{\displaystyle 1}{\displaystyle
    V_{2,r_2}(y_0)+V_2(y_0,y)}\Big(\frac{\displaystyle
    r_2}{\displaystyle r_2+d_2(y,y_0)}\Big)^{\gamma_2},$$
    where $V_{2,r_2}(y_0):=\mu_2(B_{\XX_2}(y_0,r_2)$, and
    $V_2(y_0,y):=\mu_2 \big (B_{\XX_2}(y_0,d_2(y,y_0)) \big
    )$.

    \item[(ii)] (H\"older regularity condition in $y$
        variable) For all $y,y'\in X_2$ with $d_2(y,y')\leq
        \big (r_2+d_2(y,y_0) \big)/2 A_0^{(2)}$, we have
    $$\Vert
    f(\cdot,y)-f(\cdot,y')\Vert_{G(x_0,r_1,\beta_1,\gamma_1)}\leq C
    \Big(\frac{\displaystyle d_2(y,y')}{\displaystyle
    r_2+d_2(y,y_0)}\Big)^{\beta_2} \frac{\displaystyle 1}{\displaystyle
    V_{2,r_2}(y_0)+V_2(y_0,y)}\Big(\frac{\displaystyle r_2}{\displaystyle
    r_2+d_2(y,y_0)}\Big)^{\gamma_2}.$$

   \item[(iii)]  (Size and regularity conditions in $x$ variable) Properties (i)
   and (ii) also hold interchanging the roles of $x$ and $y$.
    \end{enumerate}
\end{definition}

\noindent When $f$ is a test function of type
$(x_0,y_0;r_1,r_2;\beta_1,\beta_2;\gamma_1,\gamma_2)$, we write
$f\in G(x_{0},y_{0};r_{1},r_{2};\beta_{1},\beta_{2};$
$\gamma_{1},\gamma_{2})$.
The expression
$
    \|f\|_{G(x_{0},y_{0};r_{1},r_{2};\beta_{1},\beta_{2};\gamma_{1},\gamma_{2})}
    := \inf\{C:\ {\rm(i),\ (ii)\ and\ (iii)}\ \ {\rm hold}\}
$
defines a norm on
$G(x_{0},y_{0};r_{1},r_{2};\beta_{1},\beta_{2};\gamma_{1},\gamma_{2})$.

We denote by $G(\beta_{1},\beta_{2};\gamma_{1},\gamma_{2})$ the
class
$G(x_{0},y_{0};1,1;\beta_{1},\beta_{2};\gamma_{1},\gamma_{2})$
for any fixed $(x_{0},y_{0})\in X_1\times X_2.$  Then
$G(x_{0},y_{0};r_{1},r_{2};\beta_{1},\beta_{2};\gamma_{1},\gamma_{2})
= G(\beta_{1},\beta_{2};\gamma_{1},\gamma_{2})$, with
equivalent norms, for all $(x_{0},y_{0})\in X_1\times X_2$ and
$r_1>0$, $r_2>0$. Furthermore,
$G(\beta_{1},\beta_{2};\gamma_{1},\gamma_{2})$ is a Banach
space with respect to the norm on
$G(\beta_{1},\beta_{2};\gamma_{1},\gamma_{2})$.

For
$\beta_i \in (0,\eta_i]$ and $\gamma_i>0$, for $i = 1$, 2, let
$\GGp(\beta_1,\beta_2;\gamma_1,\gamma_2)$ be the completion of
the space $G(\eta_1,\eta_2;\gamma_1,\gamma_2)$ in
$G(\beta_1,\beta_2;\gamma_1,\gamma_2)$ in the norm of
$G(\beta_1,\beta_2;\gamma_1,\gamma_2)$.  For
$f\in\GGp(\beta_{1},\beta_{2};\gamma_{1},\gamma_{2}) $, we
define $\|f\|_{\GGp(\beta_{1},\beta_{2};\gamma_{1},\gamma_{2})}
:= \|f\|_{G(\beta_{1},\beta_{2};\gamma_{1},\gamma_{2})}$.

We define the distribution space
$\big(\GGp(\beta_{1},\beta_{2};\gamma_{1},\gamma_{2})\big)'$
to consist of all bounded linear functionals on
$\GGp(\beta_{1},\beta_{2};\gamma_{1},\gamma_{2})$. We denote by
$\langle f,h\rangle$ the natural pairing of elements $h\in
\GGs(\beta_{1},\beta_{2};\gamma_{1},\gamma_{2})$ and $f\in \big
(\GGs(\beta_{1},\beta_{2};\gamma_{1},\gamma_{2}) \big )'$.

Given Auscher-Hyt\"onen wavelets
$\{\psi_{\alpha_i}^{k_i}\}_{k_i\in\mathbb{Z},
\alpha_i\in\mathscr{Y}^{k_i}}$ with H\"older regularity
$\eta_i$ on each space of homogeneous type $(X_i,d_i,\mu_i)$
for $i = 1$, $2$, the corresponding normalized tensor product
wavelets $\widetilde{\psi}_{\alpha_1}^{k_1}(x_1)
\widetilde{\psi}_{\alpha_2}^{k_2}(x_2)$
belong to $\GGp(\beta_{1},\beta_{2};\gamma_{1},\gamma_{2})$ when
$\beta_i\in (0,\eta_i]$ and $\gamma_i>0$ for  $i=1,2$. See \cite[p.124]{HLW}.
Here $\widetilde{\psi}_{\alpha_i}^{k_i}(x_i) :=
{\psi_{\alpha_i}^{k_i}(x_i)} /\sqrt{\mu_i\big
(B_{X_i}(y^{k_i}_{\alpha_i}, \delta_i^{k_i})\big )}$ for
$i = 1$, $2$.

The following reproducing formula holds on the product
space~$X_1\times X_2$.

\begin{theorem}[\cite{HLW}, Theorem~3.11]
    \label{thm product reproducing formula test function} For
    $i = 1$, $2$, let
    $\{\psi_{\alpha_i}^{k_i}\}_{k_i \in \mathbb{Z},
    \alpha_i \in \mathscr{Y}^{k_i}}$ be Auscher-Hyt\"onen
    wavelets with H\"older regularity $\eta_i > 0$ with reference dyadic grids $\mathscr{D}_i^W$ on the
    space of homogeneous type $(X_i,d_i,\mu_i )$, and fix
    constants $\beta_i$, $\gamma_i
    \in (0,\eta_i)$ . Then  the following hold:

    \begin{itemize}
        \item[(a)] The reproducing formula
            \begin{equation}\label{product reproducing formula}
                f(x_1,x_2)
                = \sum_{k_1\in\mathbb{Z}}\sum_{{\alpha_1}\in \mathscr{Y}^{k_1}}
                \sum_{k_2\in\mathbb{Z}}\sum_{{\alpha_2} \in \mathscr{Y}^{k_2}}
                \langle f,\psi_{\alpha_1}^{k_1}\psi_{\alpha_2}^{k_2} \rangle
                \psi_{\alpha_1}^{k_1}(x_1)\psi_{\alpha_2}^{k_2}(x_2)
            \end{equation}
            holds in
            $\GGp(\beta_{1}',\beta_{2}';\gamma_{1}',\gamma_{2}')$,
            for each $\beta_i'\in (0,\beta_i)$ and
            $\gamma_i'\in (0,\gamma_i)$, for $i = 1$, $2$.
        \item[(b)] The reproducing formula
            \eqref{product reproducing formula}
            also holds in  $(\GGp(\beta_{1},\beta_{2};\gamma_{1},\gamma_{2}))'$, the space of distributions.
    \end{itemize}
\end{theorem}

{Furthermore, when $f\in L^q(X_1\times X_2)$ with $q > 1$, the
series \eqref{product reproducing formula} converges
unconditionally in the $L^q(X_1\times X_2)$-norm. This is a
consequence of the Auscher-Hyt\"onen wavelets being an
unconditional basis on $L^q(X_i)$ for $i = 1$, 2;
see~\cite[Corollary 10.4]{AH}.}

\section{Product Hardy spaces, duals, predual, key auxiliary result and theorem}\label{sec:productHp}
\setcounter{equation}{0}

In this section we first recall the Hardy spaces
$H^p(\widetilde{\XX})$, {their duals the Carleson measure
spaces  ${\rm CMO}^p(\widetilde{X})$, and the spaces of bounded
and vanishing mean oscillation, ${\rm BMO}(\widetilde{X})$ and
${\rm VMO}(\widetilde{X})$, respectively dual and predual of
$H^1(\widetilde{X})$}. All these spaces, in the setting of
product spaces of homogeneous type, were introduced
in~\cite{HLW} in terms of a square function defined via the
Auscher-Hyt\"onen wavelet bases and their reference dyadic grids. We  prove a key lemma that
shows each of the Auscher-Hyt\"onen wavelets can themselves be
further decomposed into compactly supported building blocks with appropriate size, smoothness, and cancellation conditions inherited from the wavelets. Finally, we use the key lemma  to prove a key  auxiliary theorem stating that for $1<q<\infty$ and $p_0<p\leq 1$ the set $H^p(\widetilde{\XX})\cap
L^q(\widetilde{\XX})$  is a subset of $L^p(\widetilde{\XX})$
with $L^p$-(semi)norm controlled by the $H^p$-(semi)norm. 
Here $p_0:=\max\{\omega_i/(\omega_i+\eta_i): \, i=1,2\}$ where $\omega_i$ is an upper dimension for $X_i$ and $\eta_i$ is the H\"older regularity exponent of the wavelets on $X_i$, for $i=1,2$,  used on the definition of $H^p(\widetilde{X})$ where $\widetilde{X}=X_1\times X_2)$. 
The key auxiliary results proved in this section will be needed in the proof of the Main Theorem in Section~\ref{sec:atomicHp}.

\subsection{Biparameter Hardy spaces, CMO$^p$, BMO, and VMO, via wavelets}

We focus on the bi-parameter setting $\widetilde{X} =
\XX_1\times\XX_2$, where each factor $(\XX_i,d_i,\mu_i)$
is a space of homogeneous type as defined in
Section~\ref{sec:introduction}, with the constant $\omega_i$
being an upper dimension of $\XX_i$ for $i = 1$, $2$.

The family  $\{\psi_{\alpha_i}^{k_i}\}_{k_i\in\mathbb{Z},\alpha_i\in\mathscr{Y}^{k_i}}$
is an Auscher-Hyt\"onen orthonormal  wavelet basis  on $\XX_i$  with reference dyadic grid~$\mathscr{D}_i^W$, exponential decay constant $a_i$ and $\nu_i$, and order of
regularity $\eta_i\in (0,1)$ for $i = 1$, 2, as in
Theorem~\ref{thm:AH_orthonormal_basis}. All the dyadic rectangles in this section are of the form $R=Q^{k_1}_{\alpha_1}\times Q^{k_2}_{\alpha_2}$ where $Q^{k_i}_{\alpha_i}\in \mathscr{D}_i^W$ for $i=1,2$.

We denote by $\GG$ and $(\GG)'$ for short the product test function
spaces $\GGp(\beta_{1}',\beta_{2}';\gamma_{1}',\gamma_{2}')$ and
spaces of distributions
$\big(\GGp(\beta_{1}',\beta_{2}';\gamma_{1}',\gamma_{2}')\big)^{'}$,
respectively, where $ \beta_i', \gamma_i' \in (0,\eta_i)$ for $i
= 1$, 2. Note that we fix some $\beta_i'$, $\gamma_i'$ in
$(0,\eta_i)$ and work with those test functions and the
distributions in the dual space. At the end of the day it does
not matter which $\beta_i',\gamma_i'$ were chosen, as long as
they belong to the interval $(0,\eta_i)$. The product wavelets $\psi^{k_1}_{\alpha_1}\psi^{k_2}_{\alpha_2}\in \GG$ and therefore if $f\in(\GG )'$ the notation $\langle f, \psi^{k_1}_{\alpha_1}\psi^{k_2}_{\alpha_2}\rangle$  means the action of the functional $f$ on the product wavelet, which is an appropriate test function.
We have ``color-coded" the parameters $\beta_i'$ and $\gamma_i'$ in definition of $\GG$ and $(\GG)'$  not to confuse them with the parameters $\beta_i$ and $\gamma_i$ for which the wavelets $\psi^{k_i}_{\alpha_i}$  belong to $G(\beta_i,\gamma_i)$, namely all $\beta_i\in (0,\eta_i)$ and $\gamma_i>0$ for $i=1,2$.  In the proofs below, we will want to choose the wavelets' parameter $\gamma_i$ as large as necessary. The space of distributions $(\GG)'$  appear in the definition of  the product $H^p$, ${\rm CMO}^p$, ${\rm BMO}$, and ${\rm VMO}$-spaces presented in this section as well as in the definition of atomic $H^{p,q}_{{\rm at}}$-spaces in Section~\ref{sec:atomicHp}.

In \cite{HLW}, the Hardy spaces $H^p(\XX_1\times\XX_2)$ are
defined  as follows  for $p_0<p\leq 1$, where  we let   $p_0:=\max\{\omega_i/(\omega_i + \eta_i): i=1,2\} $.
\begin{definition}[\cite{HLW}, Definition 5.1]\label{def-Hp}
    Suppose $p_0 < p \leq 1$. 
     The {\it Hardy space} $H^p(X_1\times X_2)$ is defined  to be the collection of distributions in $(\GG)'$ whose square function in  $L^p(X_1\times X_2)$,
    $$
    H^p(X_1\times X_2):=\big\lbrace f \in (\GG)': S(f)\in
    L^p(\XX_1\times\XX_2)\big\rbrace .
    $$
    Here  the {\it product Littlewood-Paley square function $S(f)$ of $f$
    related to the given orthonormal basis $\{\psi^k_{\alpha}\}_{k\in\mathbb{Z},\alpha\in\mathscr{Y}^k}$ and
    reference dyadic grids $\mathscr{D}^W_i$  on $X_i$ for $i=1,2$}, is defined by
    \begin{equation}\label{g function}
        S(f)(x_1,x_2)
        :=\Big\{ \sum_{k_1\in\mathbb{Z}}\sum_{\alpha_1\in \mathscr{Y}^{k_1}}
            \sum_{k_2\in\mathbb{Z}}\sum_{\alpha_2\in \mathscr{Y}^{k_2}} \Big|
            \langle f, \psi_{\alpha_1}^{k_1}\psi_{\alpha_2}^{k_2}\rangle \,
            \widetilde{\chi}_{Q_{\alpha_1}^{k_1}}(x_1)
            \widetilde{\chi}_{Q_{\alpha_2}^{k_2}}(x_2) \Big|^2 \Big\}^{1/2}
    \end{equation}
    with $Q^{k_i}_{\alpha_i}\in \mathscr{D}^W_i$ and $\widetilde{\chi}_{Q_{\alpha_i}^{k_i}}(x_i) :=
    \chi_{Q_{\alpha_i}^{k_i}}(x_i)\, \mu_i(Q_{\alpha_i}^{k_i})^{-1/2}$
    for $i = 1$, 2.
    For $f\in H^p(X_1\times X_2)$,  define the $H^p$-(semi)norm\footnote{For $p<1$, the semi-norm $\|\cdot\|_{H^p(X_1\times X_2)}$  satisfies all the axioms of a norm except the triangle inequality, instead it satisfies
    $\|f+g\|^p_{H^p(X_1\times X_2)}\leq \|f\|^p_{H^p(X_1\times X_2)} + \|g\|^p_{H^p(X_1\times X_2)}$.}
    \[
        \|f\|_{H^p(X_1\times X_2)}
        := \|S(f)\|_{L^p(X_1\times X_2)}.
    \]
\end{definition}

Definition~\eqref{g function} corresponds to \cite[Definition
4.7, equation (4.10)]{HLW}, where the product square function is called
$\widetilde{S}$ instead of $S$.

In \cite{HLW}  the Carleson measure space ${\rm CMO}^p(X_1\times X_2)$ are defined as follows.

\begin{definition}[\cite{HLW}, Definition 5.2]\label{def-CMOp}
    Suppose $p_0< p \leq 1$. 
    The {\it Carleson measure space} ${\rm CMO}^p(X_1\times X_2)$ is defined by
    \[
        {\rm CMO}^p(X_1\times X_2)
        := \big\lbrace f \in (\GG)': \mathcal{C}_p(f) < \infty\big\rbrace.
    \]
    Here the  quantity $\mathcal{C}_p(f)$ is  defined by
    \begin{equation}\label{Cp(f)quantity}
        \mathcal{C}_p(f)
        := \sup_{\Omega}\Big\lbrace \frac{1}{\mu(\Omega)^{\frac2p-1}}
            \sum_{R=Q^{k_1}_{\alpha_1}\times Q^{k_2}_{\alpha_1}\subset \Omega}
            |\langle f,\psi^{k_1}_{\alpha_1}\psi^{k_2}_{\alpha_2}\rangle |^2\Big\rbrace^{1/2},
   \end{equation}
   where $\Omega$ runs over all open sets in $X_1\times X_2$
   with finite measure, and it is understood, here and in the
   sequel, that the indices $k_i\in \mathbb{Z}$ and
   $\alpha_i\in \mathcal{Y}^{k_i}$ for $i = 1$, 2. The {\it space
   $\bmo$ of functions of bounded mean oscillation} is defined by
   \[
       {\rm BMO}(X_1\times X_2)
       := {\rm CMO}^1(X_1\times X_2).
   \]
   \end{definition}

One of the main results in~\cite{HLW} establishes the duality
between the Hardy spaces and the Carleson measure spaces.

\begin{theorem}[\cite{HLW}, Theorem 5.3]\label{thm:CMOp-duality-Hp}
    Suppose $p_0 < p \leq 1$. 
    Then
    $
        (H^p(X_1\times X_2))'
        = {\rm CMO}^p(X_1\times X_2).
    $
    In particular, when $p = 1$ we have
    $
        \big (H^1(X_1\times X_2)\big )'
        = {\rm BMO}(X_1\times X_2).
    $
\end{theorem}

The vanishing mean oscillation space~$\vmo(X_1\times X_2)$ was
introduced in~\cite{HLW}, and it was shown in the same paper to
be the predual of~$H^1(X_1\times X_2)$. For the convenience of
the reader we record the definition and the duality theorem.

\begin{definition}[\cite{HLW}, Definition 5.9]\label{def-VMO}
    The {\it space ${\rm VMO}(X_1\times X_2)$ of functions of
    vanishing mean oscillation} is the subspace of ${\rm
    BMO}(X_1\times X_2)$ whose elements satisfy the following three
    properties:
    \begin{itemize}
    \item[(a)] $\displaystyle{\quad \lim_{\delta\to 0^+}
        \sup_{\mu(\Omega)<\delta}\Big\lbrace
        \frac{1}{\mu(\Omega)}\sum_{R=Q^{k_1}_{\alpha_1}\times
        Q^{k_2}_{\alpha_1}\subset \Omega} |\langle
        f,\psi^{k_1}_{\alpha_1}\psi^{k_2}_{\alpha_2}\rangle
        |^2\Big\rbrace^{1/2}=0}$;
    \item[(b)] $\displaystyle{\quad \lim_{N\to \infty}
        \sup_{{\rm diam}(\Omega)> N}\Big\lbrace
        \frac{1}{\mu(\Omega)}\sum_{R=Q^{k_1}_{\alpha_1}\times
        Q^{k_2}_{\alpha_1}\subset \Omega} |\langle
        f,\psi^{k_1}_{\alpha_1}\psi^{k_2}_{\alpha_2}\rangle
        |^2\Big\rbrace^{1/2}=0}$; and
    \item[(c)] $\displaystyle{\quad \lim_{N\to \infty}
        \sup_{\Omega:\, \Omega\subset\big (B(x_1,N)\times
        B(x_2,N)\big )^c} \Big\lbrace
        \frac{1}{\mu(\Omega)}\sum_{R=Q^{k_1}_{\alpha_1}\times
        Q^{k_2}_{\alpha_1}\subset \Omega} |\langle
        f,\psi^{k_1}_{\alpha_1}\psi^{k_2}_{\alpha_2}\rangle
        |^2\Big\rbrace^{1/2}=0}$.
    \end{itemize}
    Here the suprema run over all open sets $\Omega$ in $X_1\times
    X_2$ with finite measure, and either with small measure in~(a), with
    large diameter in~(b), or living far away from an arbitrary fixed
    point $(x_1,x_2)\in X_1\times X_2$ in~(c).
\end{definition}

\begin{theorem}[\cite{HLW}, Theorem 5.10]\label{thm:VMO-duality-H1}
    The Hardy space $H^1(X_1\times X_2)$ is the dual of the
    space of vanishing mean oscillation ${\rm VMO}(X_1\times
    X_2)$. Namely,
    $
        \big ({\rm VMO}(X_1\times X_2)\big )'
        = H^1(X_1\times X_2).
    $
\end{theorem}

Note that the definitions for the $H^p$, $CMO^p$, $BMO$, and $VMO$ spaces all use given Auscher-Hyt\"onen wavelets and their underlying reference grids in $X_i$ for $i=1,2$. Whether these definitions are independent of the chosen wavelets  and reference grids is an important question, answered in the affirmative in this paper.

\subsection{Key decomposition lemma and $H^p \cap L^q \subset L^p$ theorem}\label{sec:key-lemma-and-theorem}
We point out that $\GG$,  and thus 
$\displaystyle{H^p(X_1\times X_2)\cap L^q(X_1\times X_2)}$ for $q> 1$, 
are dense in $H^p(X_1\times X_2)$ with respect to the 
$H^p(X_1\times X_2)$-(semi)norm, see  \cite[p.40-41]{HLW}.
We now show that
functions in the dense subset  $\displaystyle{H^p(X_1\times X_2)\cap L^q(X_1\times X_2)}$ also lie in $L^p(X_1\times
X_2)$, in other words  for $q>1$, 
$$H^p(X_1\times X_2)\cap L^{{q}}(X_1\times X_2) \subset L^p(X_1\times X_2),$$
 with $L^p$-(semi)norm  controlled by the $H^p$-(semi)norm. As an aside recall that the $L^p$-(semi)norm  is not a norm when $0<p<1$, satisfying  instead of the triangle inequality the following inequality: $\|f+g\|_{L^p(X_1\times X_2)}^p\leq \|f\|_{L^p(X_1\times X_2)}^p+\|g\|_{L^p(X_1\times X_2)}^p$.\\

Our key auxiliary  theorem in this section is the following.

\begin{theorem}\label{theorem-of-fLp-lessthan-fHp-on-product-case}
Given spaces of homogeneous type $(X_i, d_i, \mu_i)$ with  an upper dimension $\omega_i$, with reference dyadic grids $\mathscr{D}^W_i$, and  associated Auscher-Hyt\"onen wavelet basis $\{\psi^{k_i}_{\alpha_i}\}_{k_i\in\mathbb{Z}, \alpha_i\in \mathcal{Y}^{k_i}}$  with H\"older regularity $\eta_i\in (0,1)$,   for $i =
    1$,~$2$. 
  Let $p_0:=\max\{ \omega_i/(\omega_i+ \eta_i): \, i=1,2\}$, suppose $p_0 < p \leq 1$, and take $q >1$.  If a function $f\in H^p(\XX_1\times\XX_2)\cap
    L^{{q}}(\XX_1\times\XX_2)$, then $f\in
    L^p(\XX_1\times\XX_2)$ and there exists a constant $C_p >
    0$, independent of the $L^{{q}}$-norm of~$f$, such that
    \[
        \|f\|_{L^p( X_1\times X_2)}
        \leq C_p\|f\|_{H^p(\XX_1\times\XX_2)}.
    \]
\end{theorem}

As a consequence of Theorem
\ref{theorem-of-fLp-lessthan-fHp-on-product-case}, we have the
following result. 

\begin{corollary}\label{coro H1 in L1}
 Let $q>1$ then   $\displaystyle{H^1(X_1\times X_2) \cap L^q(X_1\times X_2 )}$ is a subset of $L^1(X_1\times X_2)$.
\end{corollary}

To prove  Theorem~\ref{theorem-of-fLp-lessthan-fHp-on-product-case}, we first establish an
auxiliary result, Lemma~\ref{lemma-decomposition}, on the
decomposition of the orthonormal basis functions
$\psi_\alpha^k$ into building blocks with compact support and other
convenient properties. These building blocks will inherit   from the wavelets, appropriately scaled, size and smoothness conditions as well as cancellation.
We follow  the approach of Nagel and Stein (see \cite[Section 3.5]{NS}). 
\begin{lemma}\label{lemma-decomposition}
Let $(X,d,\mu )$ be a space of homogeneous type with $A_0$ the quasi-triangle constant of the quasi-metric $d$, and $\omega$ an upper dimension of the Borel regular doubling measure $\mu$.
   Fix  parameters {$\gamma >\omega$} and $\overline{C} >1$.
    Suppose that $\psi_\alpha^k$ is a basis function (a wavelet)  as in
    Theorem~\ref{thm:AH_orthonormal_basis},  with exponential decay exponents $\nu>0$ and   $a=(1+2\log_2A_0)^{-1}$  and with
H\"older-regularity exponent $\eta$. 
    Then  there exist functions $\varphi^{\gamma,\overline{C}}_{\ell, k,\alpha}$ for each integer $\ell\geq 0$
    such that for all $x\in X$ and  for each $k\in\mathbb{Z}$, $\alpha\in \mathscr{Y}^k$, we have the
    following decomposition for the normalized wavelets
    $\widetilde{\psi}_\alpha^k :=  \psi_\alpha^k(x)/\sqrt{\mu \big (B(y_\alpha^k,\delta^k) \big )}$,
     \begin{equation}\label{decomposition of wavelet into atom}
    \widetilde{\psi}_\alpha^k(x)
        = \sum_{\ell=0}^\infty (2^{\ell}\overline{C})^{-\gamma} \varphi^{\gamma, \overline{C}}_{\ell,k,\alpha}(x).
    \end{equation}
    Here each $\varphi^{\gamma}_{\ell,k,\alpha}$ satisfies the following  properties.
\begin{itemize}
    \item[(i)] {\rm (Compact  support)} $\;\;\supp\varphi^{\gamma,\overline{C}}_{\ell,k,\alpha} \subset B(y_\alpha^k, {2A_0^2}\,  \overline{C} 2^{\ell}\,\delta^k).$
     \item[(ii)] {\rm (Boundedness)} There is a constant $C_{\gamma}>0$ such that for all $x\in X$
$$|\varphi^{\gamma,\overline{C}}_{\ell,k,\alpha}(x)| \leq C_{\gamma} (\overline{C}2^{\ell})^{\omega}/
\mu \big (B(y_\alpha^k, \overline{C}2^{\ell}\delta^k)\big ).$$ 
  \item[(iii)] {\rm (Local H\"older regularity)}   There is a constant $C_{\gamma}>0$ such that   for all $x,y\in X$ with $d(x,y)\leq \delta^k,$
    $$ | \varphi^{\gamma,\overline{C}}_{\ell,k,\alpha}(x) - \varphi^{\gamma,\overline{C}}_{\ell,k,\alpha}(y)|
     \leq C_{\gamma}\,(\overline{C}2^\ell\delta^k)^{-\eta} \, (\overline{C}2^{\ell})^{\omega} 
      d(x,y)^{\eta}/\mu \big (B(y_\alpha^k,  \overline{C}2^{\ell} \delta^k)\big ).$$          
     \item[(iv)] {\rm (Cancellation)} $\;\;\int_X \varphi^{\gamma}_{\ell,k,\alpha}(x)\,d\mu(x) = 0$.
    \end{itemize}
 Here $C_{\gamma}$ is a positive constant independent of
    $y_\alpha^k$,  $\delta^k$, and~$\ell$. However $C_{\gamma}$ will depend on the fixed $\gamma >0$ and the geometric constants of the space $X$.
    The equality \eqref{decomposition    of wavelet into atom} holds pointwise, as well as in
    $L^q(X)$ for $q\in (1,\infty)$.
\end{lemma}

Lemma~\ref{lemma-decomposition} allows for two parameters, a decaying parameter $\gamma>\omega$ and a dilation parameter $\overline{C}>1$. Later on we will pick $\gamma$ large enough so that some geometric series converge and we will need $\overline{C}$ to match  dilation
parameters for the $(p,q)$-atoms which are independent of the wavelets and based on possibly separate dyadic grids.
When $\overline{C}=1$ we simply write $\varphi^{\gamma}_{\ell,k,\alpha}$.

In the local H\"older regularity  condition~(iii) in Lemma~\ref{lemma-decomposition},   the range of validity, $d(x,y)\leq \delta^k$,  is inherited from the wavelets local regularity condition as in Theorem~\ref{thm:AH_orthonormal_basis}(iii). In the proof of Lemma~\ref{lemma-decomposition} we will see that a type of H\"older regularity like the one test functions have, see Definition~\ref{def-of-test-func-space}(ii), with range of validity  $d(x, y) < \big (2A_0)^{-1}(\delta^k + d(x, y^{k}_{\alpha})\big )$  provided $x\in B(y_{\alpha}^k, {A_0^2}\, {\overline{C}}2^{\ell} \delta^k)\setminus B(y_{\alpha}^k, {\overline{C}}2^{\ell -1}\delta^k{/4})$, will also hold because the wavelets are test functions by Theorem~\ref{prop wavelet is test function}. We will need this estimate in the proof of the Main Theorem in Section~\ref{sec:atomicHp}.

What is gained in this decomposition is  the compact support of the building blocks, as opposed to the exponential decay of the wavelets being decomposed. What is lost is the orthonormality  of the wavelets, however the building blocks  will have an appropriate ``almost-orthogonality" property that will be needed in the proof of  Theorem~\ref{theorem-of-fLp-lessthan-fHp-on-product-case}. This almost-orthogonality of the building blocks  is captured in Lemma~\ref{lemma LittlewoodPaley} stated in page~\pageref{Littlewood-Paley Key Lemma} and proved after the the proof of Theorem~\ref{theorem-of-fLp-lessthan-fHp-on-product-case}  in page~\pageref{proof-Lemma-LP}.

\begin{proof}[Proof of Lemma~\ref{lemma-decomposition}]

Fix {$\gamma >\omega$}, $k\in\mathbb{Z}$, and $\alpha\in\mathscr{Y}_k$. Let
\begin{align}
    \Lambda_0^{\overline{C}}(x)
    &:=  h_0(x)\,  
        \widetilde{\psi}_\alpha^k(x)
        \quad\text{and} \label{def:Lambda_0}\\
    \Lambda_\ell^{\overline{C}}(x)
    &:= \big(h_{\ell}(x) - h_{\ell-1}(x)\big )\,
        \widetilde{\psi}_\alpha^k(x) \ 
        \ {\rm for}\ \ \ell\geq1. \label{def:Lambda_ell}
\end{align}
The cut-off functions $h_{\ell}\in C^{\eta}(X)$ are given by Lemma~\ref{lem:cut-off-functions} based on $x_0=y^k_{\alpha}$ and  with parameter $R_0=\overline{C}2^{\ell}\delta^k$ for each $\ell\geq 0$.  They have the following properties
for $\ell>0$: first
 $0\leq h_{\ell}(x)\leq 1$;  second
 \begin{equation}\label{eqn:cut-off-support}
 h_{\ell}(x) \equiv 1 \;\; \mbox{when $x\in B(y^k_{\alpha},\overline{C}2^{\ell}\delta^k/4)$}, \quad\quad h_\ell(x)\equiv 0 \;\; \mbox{when $x\in B(y^k_{\alpha},A_0^2 \,\overline{C}2^{\ell} \delta^k)^c$};
 \end{equation}
and third,  there exists a constant $C>0$ independent of $y^k_{\alpha}$ and  $\ell$, depending only on the geometric constants of the space $X$, such that  for all $x,y\in X$ the following global H\"older regularity holds:
\begin{equation}\label{eqn:cut-off-Holder-regularity}
|h_{\ell}(x)-h_{\ell}(y)| \leq C \Big ({d(x,y)}/ \,{\overline{C}2^{\ell}\delta^k}\Big )^{\eta}.
 \end{equation}
{By definition, the function $\Lambda_0$ is supported on $B(y_{\alpha}^k, {A_0^2}\,\overline{C} \delta^k)$ and the function
$\Lambda_{\ell}^{\overline{C}}$ for $\ell\geq 1$ is supported on the annulus $B(y_{\alpha}^k, {A_0^2}\,\overline{C} 2^{\ell} \delta^k)\setminus B(y_{\alpha}^k, \overline{C}2^{\ell -1}\delta^k{/4})$.} By a telescoping sum argument we see that
\[\sum_{\ell =0}^L \Lambda_{\ell}^{\overline{C}}(x) =  {h_L(x)}\,
\widetilde{\psi}_\alpha^k(x)   
\; \mbox{and is identical to} \;
\widetilde{\psi}_\alpha^k(x) 
\; \mbox{on $B(y_\alpha^k,\overline{C}2^{L}\delta^k{/4})$.}\]

It follows  that
$\widetilde{\psi}_\alpha^k(x) 
= \sum_{\ell\geq 0} \Lambda_\ell^{\overline{C}}(x)$ pointwise. Moreover, for all $x\in X$ and every~$\gamma>0$,
\begin{equation}\label{eqn:sizeLambda}
    |\Lambda_\ell^{\overline{C}}(x)|
    \lesssim_{\gamma} \frac{(\overline{C}2^{\ell})^{-\gamma}}{\mu \big (B(y_\alpha^k, \delta^k) \big )}
    \lesssim_{\gamma}  \frac{(\overline{C}2^{\ell})^{\omega-\gamma}}{\mu \big (B(y_\alpha^k, \overline{C}2^{\ell}\delta^k) \big )}. 
\end{equation}
The second inequality  by the doubling property of the measure. The first inequality can be seen since $\psi_\alpha^k(x)$ has the exponential decay
property~\eqref{eqn:exponential_decay}, { $|h_{\ell}(x)-h_{\ell-1}(x)|\in [0,1]$},
and $\Lambda_{\ell}^{\overline{C}}$ is supported on $B(y_{\alpha}^k, {A_0^2}\, \overline{C}2^{\ell} \delta^k)\setminus B(y_{\alpha}^k, \overline{C}2^{\ell -1}\delta^k{/4})$.  Note that for $\nu, a >0$ the function $e^{-\nu z^a }z^{\gamma}$ defined for $z\geq 0$ is a bounded function for each $\gamma>0$, with an upper bound depending on $\gamma>0$.

Next, following the argument in \cite[p.550-551]{NS},  define $a_\ell:=\int_X \Lambda_\ell^{\overline{C}}(x)\,d\mu(x)$.
Using~\eqref{eqn:sizeLambda} 
 it is clear that
$a_\ell=O\big ((\overline{C}2^{\ell})^{\omega-\gamma}\big )$.
Define $s_\ell := \sum_{0\leq j\leq \ell} a_j$, note that by Lebesgue domination theorem,
$$\sum_{\ell\geq 0} a_\ell =
\int_X \widetilde{\psi}_\alpha^k(x)   
\,d\mu(x)=0,$$
therefore we have
$s_\ell = -\sum_{j>\ell}a_j$, which gives
$s_\ell=O\big ( (\overline{C}2^{\ell})^{\omega-\gamma}\big )$.

We now define the function $ \widetilde{\Lambda}_{\ell}^{\overline{C}}:X\to \mathbb{R}$ by
 $$  \widetilde{\Lambda}_{\ell}^{\overline{C}}(x) :=
\Lambda_{\ell}^{\overline{C}}(x)-a_\ell \,\xi_\ell(x)+ s_\ell \big(\xi_\ell(x)-\xi_{\ell+1}(x) \big) = \Lambda_{\ell}^{\overline{C}}(x)+s_{\ell-1}\, \xi_\ell(x)- s_\ell \,\xi_{\ell+1}(x).$$
Here  for each $\ell\geq 0$ the function  $\xi_{\ell}$ is the  $L^1$-normalization of the function  {$h_{\ell}$}  \begin{equation}\label{def:xi_ell}
    \xi_\ell(x)
    := h_{\ell}(x)   
        \Big[\int_X h_{\ell}(z) 
        \,d\mu(z) \Big]^{-1},
\end{equation}
Finally we define the functions  $\varphi^{\gamma,\overline{C}}_{\ell,k,\alpha}$
in the decomposition of the wavelets
\begin{equation}\label{def:varphi}
    \varphi^{\gamma,\overline{C}}_{\ell,k,\alpha}(x)
    := (\overline{C}2^{\ell})^{\gamma}\widetilde{\Lambda}_{\ell}^{\overline{C}}(x).
\end{equation}
Note that $\widetilde{\Lambda}_{\ell}^{\overline{C}}$ does not depend on $\gamma$, although it depends on the fixed $k$ and  $\alpha$.
It is easy to verify that the decomposition~\eqref{decomposition of wavelet into atom} holds. Namely
\begin{align*}
    \sum_{\ell\geq 0} (\overline{C}2^{\ell})^{-\gamma}\varphi^{\gamma,\overline{C}}_{\ell,k,\alpha}(x)
    &= \sum_{\ell\geq 0} \widetilde{\Lambda}_{\ell}^{\overline{C}}(x)\\
    &= \sum_{\ell\geq 0} \Lambda_{\ell}^{\overline{C}}(x) - \sum_{\ell\geq 0}
        a_\ell \xi_\ell(x)+ \sum_{\ell \geq 0}
        s_\ell \big(\xi_\ell(x)-\xi_{\ell+1}(x) \big)\\
    &= \widetilde{\psi}_\alpha^k(x),  
\end{align*}
where the last equality follows from the facts that
$\widetilde{\psi}_\alpha^k(x)  
= \sum_{\ell\geq 0}\Lambda_\ell^{\overline{C}}(x)$ and  $\sum_{\ell\geq 0} a_\ell
\xi_\ell = \sum_{\ell\geq 0} s_\ell(\xi_\ell(x) - \xi_{\ell+1}(x) )$, using summation by
 parts and noting that $a_{\ell}=s_\ell-s_{\ell -1}$.

Now we verify that $\varphi^{\gamma,\overline{C}}_{\ell,k,\alpha}$ satisfies  properties
(i), (ii), (iii), and~(iv).

In fact, from the definition of $\varphi^{\gamma,\overline{C}}_{\ell,k,\alpha}$ it is easy
to see that properties~(i) and~(iv) hold. We now turn to property~(ii). From the size estimate \eqref{eqn:sizeLambda}
we have that
\begin{equation}\label{size-Lambda-gamma}
    |\Lambda_\ell^{\overline{C}}(x)|
    \lesssim_{\gamma} \frac{ (\overline{C}2^{\ell})^{\omega-\gamma} }
    {\mu\big (B(y_\alpha^k, \overline{C}2^{\ell}\delta^k)\big )}
\end{equation}
for each $\gamma > 0$, where $\omega$ is an upper dimension of
the measure $\mu$.  
Next, it follows from the definition of the function $\xi_\ell$ that
\[
    |\xi_\ell(x)|
    \lesssim  \frac{1}{\mu \big (B(y_\alpha^k,\overline{C}2^{\ell}\delta^k)\big )}
\]
because   $0\leq {h_{\ell}(x)} \leq 1$ and  $\mu \big (B(y_{\alpha}^k,\overline{C}2^{\ell-1}\delta^k{/4})\big )\leq \int_X {h_{\ell} }(z)\,d\mu(z) \leq \mu \big (B(y_{\alpha}^k,{A_0^2}\,\overline{C}2^{\ell}\delta^k)\big )$. Furthermore, using the doubling property of $\mu$, we conclude that   
\begin{equation}\label{integral-widetilde-xi}
\int_X h_{\ell}(z)\, d\mu(z) \sim \mu\big (B(y_{\alpha}^k,\overline{C}2^{\ell}\delta^k)\big ).
\end{equation}
Consequently, recalling that $a_{\ell}= O\big ( (\overline{C}2^{\ell})^{\omega-\gamma})$ and $s_{\ell}= O\big ((\overline{C}2^{\ell})^{\omega-\gamma})$,  we conclude that
property~(ii) holds. 

Similarly, from the H\"older
regularity~\eqref{eqn:Holder_regularity} of the wavelet $\psi_{\alpha}^k$ and estimate~\eqref{eqn:cut-off-Holder-regularity} of the cut-off functions $h_{\ell}$,  
 together with the definition of the
function~$\xi_\ell$, we obtain that property~(iii) holds.
More precisely, we need to verify that  there is a constant $C_{\gamma}>0$ depending only on  the geometric constants of $X$ and on $\gamma$,  such that for all $x,y\in X$ {with $d(x,y)\leq \delta^k$}, and for all $\ell$, $\alpha$, and $k$ the following inequality holds
\[ |\varphi^{\gamma \overline{C}}_{\ell,k,\alpha}(x) - \varphi^{\gamma,\overline{C}}_{\ell,k,\alpha}(y)| \leq  \frac{C_{\gamma} \, (\overline{C}2^\ell \delta^k)^{-\eta}\ (\overline{C}2^{\ell})^{ \omega}}{\mu \big (B(y_\alpha^k,\overline{C}2^\ell
    \delta^k)\big )}
 \, d(x,y)^{\eta}.    \]
Without loss of generality we can assume that $d(x,y)>0$, in other words $x\neq y$.  Using  definition \eqref{def:varphi} of the atoms $\varphi^{\gamma}_{\ell, k,\alpha}$ and the triangle inequality we get that
\[   |\varphi^{\gamma,\overline{C}}_{\ell,k,\alpha}(x) - \varphi^{\gamma,\overline{C}}_{\ell,k,\alpha}(y)| \leq  (\overline{C}2^{\ell})^{\gamma}\big (
|\Lambda_{\ell}^{\overline{C}}(x)-\Lambda_{\ell}^{\overline{C}}(y)| + |s_{\ell-1}| \,|\xi_{\ell}(x)-\xi_{\ell}(y)| +|s_{\ell}| \, |\xi_{\ell +1}(x)-\xi_{\ell +1}(y)| \big ).
\]
Since 
$s_{\ell}=O\big ((\overline{C}2^{\ell})^{-\gamma})$, it suffices to show that there is a constant $C_{\gamma}>0$ such that for all $x,y\in X$ {with $d(x,y)\leq \delta^k$} the following two inequalities hold
\begin{eqnarray}
(\overline{C}2^{\ell})^{\gamma} |\Lambda_{\ell}^{\overline{C}}(x)-\Lambda_{\ell}^{\overline{C}}(y)| & \leq & \frac{C_{\gamma}\, (\overline{C}2^\ell \delta^k)^{-\eta}\ (\overline{C}2^{\ell})^{\omega}}{\mu \big (B(y_\alpha^k,\overline{C}2^\ell   \delta^k)\big )} \, d(x,y)^{\eta},  \label{estimate-for-Lambda}\\
|\xi_{\ell}(x)-\xi_{\ell}(y)| & \leq &  \frac{C_{\gamma}\, (\overline{C}2^\ell \delta^k)^{-\eta} }
{\mu \big (B(y_\alpha^k,\overline{C}2^\ell   \delta^k)\big )} \, d(x,y)^{\eta}. \label{estimate-for-xi}
\end{eqnarray}

We first estimate~\eqref{estimate-for-xi}.
Using definition~\eqref{def:xi_ell} of $\xi_{\ell}$, estimate \eqref{integral-widetilde-xi}, and the fact that $h_{\ell}$
satisfies estimate~\eqref{eqn:cut-off-Holder-regularity} for all $x,y\in X$, we obtain
\[ |\xi_{\ell}(x)-\xi_{\ell}(y)|  \lesssim   \frac{ |h_{\ell}(x)-h_{\ell}(y)|}
{\mu\big (B(y^k_{\alpha},\overline{C}2^{\ell}\delta^k)\big )}
\lesssim  \frac{ (\overline{C}2^{\ell} \delta^k)^{-\eta} 
}{\mu \big (B(y_\alpha^k,\overline{C}2^\ell   \delta^k)\big )} \, d(x,y)^{\eta}.
\]
This is more than what we wanted to show, since $x$ and $y$ are not required to be $\delta^k$-close to each other, and the similarity constants are independent of $\gamma$.

We now estimate \eqref{estimate-for-Lambda}. 
We argue in the case when $\ell>0$ and note that when $\ell=0$ a similar calculation, somewhat simpler, yields the  desired estimate.
By definition~\eqref{def:Lambda_ell} of $\Lambda_{\ell}^{\overline{C}}$ when $\ell>0$,
we conclude that
\begin{eqnarray*}
| \Lambda_{\ell}^{\overline{C}}(x)-\Lambda_{\ell}^{\overline{C}}(y)| & \leq & |h_{\ell}(x)\widetilde{\psi}^k_{\alpha}(x) - h_{\ell}(y)\widetilde{\psi}^k_{\alpha}(y)|
+ |h_{\ell-1}(x)\widetilde{\psi}^k_{\alpha}(x) - h_{\ell-1}(y)\widetilde{\psi}^k_{\alpha}(y)|.
\end{eqnarray*}
For all $\ell>0$ we  estimate using the triangle inequality
\begin{eqnarray*}
|h_{\ell}(x)\, \widetilde{\psi}^k_{\alpha}(x) - h_{\ell}(y)\, \widetilde{\psi}^k_{\alpha}(y)|
 & \leq &
 \| h_{\ell}\|_{L^{\infty}(X)} |\widetilde{\psi}^k_{\alpha}(x) - \widetilde{\psi}^k_{\alpha}(y)| +
 \|\widetilde{\psi}^k_{\alpha}\|_{L^{\infty}(X)} |h_{\ell}(x)-h_{\ell}(y)|.
 \end{eqnarray*}
 Using the exponential decay and H\"older regularity estimates~\eqref{eqn:exponential_decay} and~\eqref{eqn:Holder_regularity} for the wavelet $\psi^k_{\alpha}$, together with the fact that $\| h_{\ell}\|_{L^{\infty}(X)}\leq 1$  and the H\"older regularity estimate~\eqref{eqn:cut-off-Holder-regularity} of~$h_{\ell}$, we conclude that, when  $d(x,y)\leq \delta^k$,
\begin{eqnarray*}
|h_{\ell}(x) \,\widetilde{\psi}^k_{\alpha}(x) - h_{\ell}(y)\, \widetilde{\psi}^k_{\alpha}(y)|  & \lesssim &
\frac{  \exp\Big(-\nu\Big( \frac{d(y^k_\alpha,x)}{\delta^k}\Big)^a\Big)}{\mu \big (B(y_\alpha^k,\delta^k)\big )}
         \big ( \delta^{-k\eta}   d(x,y)^\eta +  (\overline{C}2^{\ell}\delta^k)^{-\eta}d(x,y)^{\eta} \big )
         \\
           & & \hskip -.8in \lesssim_{\Gamma}  \; \frac{\delta^{-k\eta} (\overline{C}2^{\ell})^{\omega} \big(1+  (\overline{C}2^{\ell})^{-\eta}\big)}{{\mu \big (B(y_\alpha^k,\overline{C}2^{\ell}\delta^k)\big )}}
                  d(x,y)^\eta \Big (\frac{\delta^k}{\delta^k+d(y^k_{\alpha},x)}\Big )^{\Gamma},
\end{eqnarray*}
for all $\Gamma>0$. 
Where we used the doubling property~\eqref{eqn:upper dimension}  in the last inequality.  When $x$ is in   the support of $\Lambda_{\ell}^{\overline{C}}$, namely $B(y_{\alpha}^k, {A_0^2}\, \overline{C}2^{\ell} \delta^k)\setminus B(y_{\alpha}^k,\overline{C}2^{\ell -1}\delta^k{/4})$,  then  $d(x,y^k_{\alpha})\, \delta^{-k}\sim \overline{C}2^{\ell}$, we conclude that for all $\Gamma>0$
  \begin{eqnarray*}
(\overline{C}2^{\ell})^{\gamma}| \Lambda_{\ell}^{\overline{C}}(x)-\Lambda_{\ell}^{\overline{C}}(y)| & \lesssim_{\Gamma}&
 \frac{(\overline{C}2^{\ell}\delta^k)^{-\eta}\, (\overline{C}2^{\ell})^{\omega}}{{\mu \big (B(y_\alpha^k,\overline{C}2^{\ell}\delta^k)\big )}} d(x,y)^\eta
\, {(\overline{C}2^{\ell})^{\gamma} \big ((\overline{C}2^{\ell})^{\eta} +1\big )}\Big (\frac{1}{1+d(y^k_{\alpha},x)\,\delta^{-k}}\Big )^{\Gamma}\\
& \lesssim_{\Gamma} & \frac{(\overline{C}2^{\ell}\delta^k)^{-\eta}\, (\overline{C}2^{\ell})^{\omega}}{{\mu \big (B(y_\alpha^k,\overline{C}2^{\ell}\delta^k)\big )}} d(x,y)^\eta
\, (\overline{C}2^{\ell})^{\gamma + \eta -\Gamma} .
 \end{eqnarray*}
 Picking $\Gamma = \gamma+\eta$ we get estimate~\eqref{estimate-for-Lambda} at least when  $x$ is in the support of $\Lambda_{\ell}^{\overline{C}}$ and $d(x,y)\leq \delta^k$.  Clearly when both $x$ and $y$ are not in the support of  $\Lambda_{\ell}^{\overline{C}}$ then $\Lambda_{\ell}^{\overline{C}}(x)-\Lambda_{\ell}^{\overline{C}}(y) =0$. The only remaining case is when $y$ is in the support of $\Lambda_{\ell}^{\overline{C}}$ and $x$ is not. The calculations above are symmetric in $x$ and $y$, interchanging their roles we  conclude that
when $d(x,y)\leq \delta^k$ then
$$(\overline{C}2^{\ell})^{\gamma}| \Lambda_{\ell}^{\overline{C}}(x)-\Lambda_{\ell}^{\overline{C}}(y)| \lesssim_{\gamma} \frac{(\overline{C}2^{\ell}\delta^k)^{-\eta}\, (\overline{C}2^{\ell})^{\omega}}{{\mu \big (B(y_\alpha^k,\overline{C}2^{\ell}\delta^k)\big )}} d(x,y)^\eta.$$
This proves estimate~\eqref{estimate-for-Lambda} and shows that condition~(iii) in the lemma holds.

By Proposition~\ref{prop wavelet is test function},  $\widetilde{\psi}^k_{\alpha}$ is a test function of type $(y^k_{\alpha}, \delta^k,\eta,\gamma+\eta)$. Using the test-function properties instead of the local H\"older regularity of the wavelets as we just did,  one can verify in a similar manner  that when  $x\in {\rm supp} (\Lambda_{\ell}^{\overline{C}})$ and $d(x,y)\leq (2A_0)^{-1}(\delta^k+d(x,y^k_{\alpha}))$
then
\begin{equation}\label{test-function-estimate-Lambda-x}
(\overline{C}2^{\ell})^{\gamma}| \Lambda_{\ell}^{\overline{C}}(x)-\Lambda_{\ell}^{\overline{C}}(y)| \lesssim_{\gamma} \frac{(\overline{C}2^{\ell}\delta^k)^{-\eta}}{{\mu \big (B(y_\alpha^k,\delta^k )\big ) + \mu\big (B(x,d(x,y^k_{\alpha}))\big )}} d(x,y)^\eta.
\end{equation}

Finally we can verify that the convergence  in equality \eqref{decomposition
    of wavelet into atom} is not just pointwise, but also in
    $L^q(X)$ for $q\in (1,\infty)$. Indeed, let $ \psi_\alpha^{k,N}(x) = \sqrt{\mu(B(y_\alpha^k,\delta^k))}
         \sum_{\ell=0}^N (\overline{C}2^{\ell})^{-\gamma}\varphi^{\gamma,\overline{C}}_{\ell,k,\alpha}(x)$. Then, using the already proven boundedness and support properties~(i) and~(ii) of $\varphi^{\gamma,\overline{C}}_{\ell,k,\alpha}$ in Lemma~\ref{lemma-decomposition}, we readily see that, 
     \begin{align*}
         \| \psi_\alpha^{k}-\psi_\alpha^{k,N}\|_{L^q(X)} & \leq \sqrt{\mu \big (B(y_\alpha^k,\delta^k)\big) }
         \sum_{\ell=N+1}^{\infty} (\overline{C}2^{\ell})^{-\gamma} \|\varphi^{\gamma,\overline{C}}_{\ell,k,\alpha}\|_{L^q(X)} \\
         & \lesssim   (\overline{C})^{\omega-\gamma} \mu \big (B(y_\alpha^k,\delta^k)\big )^{\frac{1}{2}}     \sum_{\ell=N+1}^{\infty} 2^{(-\gamma + \omega)\ell}
       \mu\big (B(y_\alpha^k,\overline{C}2^{\ell}\delta^k)\big )^{-\frac{1}{q'}}\\
         &\lesssim (\overline{C})^{\omega-\gamma } \mu \big (B(y_\alpha^k,\delta^k)\big )^{\frac{1}{2}-\frac{1}{q'}}  \sum_{\ell=N+1}^{\infty} 2^{(-\gamma + \omega)\ell}.
         \end{align*}
        As $N\to \infty$,  the series on the right-hand-side  converges to zero, since $\gamma > \omega$.
       In the last inequality we simply observed that  $\mu\big (B(y_\alpha^k,\overline{C}2^{\ell}\delta^k)\big )^{-{1}/{q'}} \leq \mu \big (B(y_\alpha^k,\delta^k)\big )^{-{1}/{q'}} $ since the power is negative.
\end{proof}

We now present the proof of the key auxiliary theorem.

\begin{proof}[Proof of Theorem \ref{theorem-of-fLp-lessthan-fHp-on-product-case}]
Suppose that $f\in H^p(X_1\times X_2)\cap L^{q}(X_1\times X_2)$ and let $\mu$ denote the product measure $\mu_1\times \mu_2$.
Then, by the reproducing formula~\eqref{product reproducing
formula}, Lemma~\ref{lemma-decomposition}  with $\overline{C}_i=1$ for $i=1,2$, and Fubini for summations, we have
\begin{align}
    f(x_1,x_2)
    &= \sum_{k_1\in\mathbb{Z}} \; \sum_{\alpha_1\in \mathscr{Y}^{k_1}} \;\sum_{k_2\in\mathbb{Z}} \; \sum_{\alpha_2\in \mathscr{Y}^{k_2}}
        \langle f,\psi_{\alpha_1}^{k_1}  \psi_{\alpha_2}^{k_2} \rangle
        \, \psi_{\alpha_1}^{k_1}(x_1) \, \psi_{\alpha_2}^{k_2}(x_2)\nonumber\\
    &=: \sum_{\ell_1,\ell_2\geq 0} 2^{-\ell_1\gamma_1}\,2^{-\ell_2\gamma_2}
        f_{\ell_1,\ell_2}(x_1,x_2).
        \label{special-repro-identity}
\end{align}
Where $f_{\ell_1,\ell_2}$ is defined by
\begin{equation}\label{fell1ell2}
  f_{\ell_1,\ell_2}(x_1,x_2) :=  \sum_{k_1\in\mathbb{Z}} \; \sum_{\alpha_1\in \mathscr{Y}^{k_1}}  \; \sum_{k_2\in\mathbb{Z}} \;\sum_{\alpha_2\in \mathscr{Y}^{k_2}}
        \langle f,\psi_{\alpha_1}^{k_1}\psi_{\alpha_2}^{k_2} \rangle
        \, \kappa_1\,\varphi^{\gamma_1}_{\ell_1,k_1,\alpha_1}(x_1)\,\kappa_2
        \, \varphi^{\gamma_2}_{\ell_2,k_2,\alpha_2}(x_2),
        \end{equation}
here we are denoting $\varphi^{\gamma_i}_{\ell_i,k_i,\alpha_i}:=\varphi^{\gamma_i, 1}_{\ell_i,k_i,\alpha_i}$ and  $\kappa_i:= \sqrt{\mu_i \big (B(y_{\alpha_i}^{k_i},\delta^{k_i})\big )}$  for $i = 1, 2$ {(we are abusing notation, to be more precise we should write $\kappa^{k_i}_{\alpha_i}$ instead of simple $\kappa_i$)}.
The parameter  {$\gamma_i$} is an arbitrary constant larger than the
upper dimension of $X_i$, that is $\gamma_i >\omega_i$, for
$i=1,2$,  and to be determined later. All these series converge
{unconditionally} in the $L^{q}(X_1\times X_2)$-norm when
$q>1$, allowing us to reorder the series at will.

Now for $j\in\mathbb{Z}$, we let $\Omega_j$ be a level set for $S(f)$, more precisely
\begin{equation}\label{Omega_j}
    \Omega_j
    :=\bigg\{(x_1,x_2)\in X_1\times X_2: S(f)(x_1,x_2)>2^j\bigg\}.
\end{equation}
Notice that $\Omega_{j+1}\subset \Omega_j$ for all $j\in \mathbb{Z}$ {and that
by the well-known layer-cake\footnote{{Assume $F\in L^p(X,\mu )$ then
$\|F\|^p_{L^p(\mu)}=\int_0^{\infty} p\lambda^{p-1}\mu\{x\in X: |F(x)| > \lambda\} \, d\lambda$.}}
formula for the $L^p$-(semi)norm of $S(f)$ it holds that
\begin{equation}\label{Sf-norm-Lp-coronas}
 \|S(f)\|_{L^p(X_1\times X_2)}^p \sim_p \sum_{j\in\mathbb{Z}} 2^{pj}\mu (\Omega_j ).
 \end{equation}  }
 Also, by Tchebichev's inequality, when $f\in L^p(X_1\times X_2)$,
 \begin{equation}\label{Tchebichev-square}
 \mu (\Omega_j ) \leq 2^{-jp}\int_{\Omega_j} |S(f)(x_1,x_2)|^p \,d\mu_1(x_1)\, d\mu_2(x_2).
 \end{equation}

If $f=0$ in $L^q(X_1\times X_2)$ then $S(f)=0$ in $L^q(X_1\times X_2)$ and the theorem is trivially true.
Assume  $f\neq 0$ in $L^q(X_1\times X_2)$, notice that this implies that $S(f)\neq 0$ in $L^q(X_1\times X_2)$,
and it ensures that there is $j_0\in \mathbb{Z}$ such  that $\mu(\Omega_j)>0$ for all $j\leq j_0$.

 Recall that the  reference dyadic grids  underlying the wavelets on $X_i$ are denoted $\mathcal{D}^{W}_{i}$ for $i=1,2$. Given dyadic cubes $Q^{k_i}_{\alpha_i}\in \mathscr{D}^W_i$   for $i=1,2$,  let $R=R^{k_1,k_2}_{\alpha_1,\alpha_2}$  denote the  dyadic rectangle in $X_1\times X_2$  they determine, that is,  $R^{k_1,k_2}_{\alpha_1,\alpha_2}:=Q_{\alpha_1}^{k_1}\times
Q_{\alpha_2}^{k_2}$. Let
\begin{equation} \label{setBj}
    \mathcal{B}_j
    := \Big\{R \; \mbox{dyadic rectangle}:      \mu( R \cap\Omega_j )>\mu(R)/2, \;
        \mu( R \cap\Omega_{j+1} )\leq \mu(R)/2\Big\}.
\end{equation}

In particular, since $S(f)\neq 0$ in $L^q(X_1\times X_2)$, each dyadic rectangle $R^{k_1,k_2}_{\alpha_1,\alpha_2}$ belongs to exactly  one set $\mathcal{B}_j$. {We can reorder the quadruple sum in \eqref{fell1ell2} over $(k_1, k_2, \alpha_1, \alpha_2)\in \mathbb{Z}^2\times \mathscr{Y}^{k_1}\times \mathscr{Y}^{k_2}$ by first adding over $j\in\mathbb{Z}$ and second adding over those $(k_1, k_2, \alpha_1, \alpha_2)$ such that  $R_{\alpha_1,\alpha_2}^{k_1,k_2}\in \mathcal{B}_j$, obtaining}. 
\begin{equation}\label{special-repro-identity-further}
    f_{\ell_1,\ell_2}(x_1,x_2)
    = \sum_{j\in\mathbb{Z}}\sum_{R_{\alpha_1,\alpha_2}^{k_1,k_2}\in \mathcal{B}_j}
        \langle f,\psi_{\alpha_1}^{k_1}\psi_{\alpha_2}^{k_2} \rangle\,
        \kappa_1\,\varphi^{\gamma_1}_{\ell_1,k_1,\alpha_1}(x_1)\,\kappa_2\,
        \varphi^{\gamma_2}_{\ell_2,k_2,\alpha_2}(x_2).
\end{equation}

Next, we will show below that for each $j\in\mathbb{Z}$,
\begin{align}\label{claim Lp}
    & \hskip -1in \Big\|\sum_{R_{\alpha_1,\alpha_2}^{k_1, k_2}\in \mathcal{B}_j}
        \langle f,\psi_{\alpha_1}^{k_1}\psi_{\alpha_2}^{k_2} \rangle \,
            \kappa_1\, \varphi^{\gamma_1}_{\ell_1,k_1,\alpha_1}\kappa_2\, \varphi^{\gamma_2}_{\ell_2,k_2,\alpha_2}
        \Big\|_{L^p(X_1\times X_2)}^p \nonumber\\
    &\hskip1cm\lesssim {(1+\ell_1 \omega_1+\ell_2 \omega_2 )^{1-\frac{p}{q}}}2^{\ell_1 \omega_1 (1+\frac{p}{{q'}})}2^{\ell_2 \omega_2 (1+\frac{p}{{q'}})}
        2^{jp}\mu(\Omega_{j}).
\end{align}
Together with the special reproducing
formula~\eqref{special-repro-identity} {and estimate~\eqref{Sf-norm-Lp-coronas}}, inequality~\eqref{claim
Lp} yields the conclusion of
Theorem~\ref{theorem-of-fLp-lessthan-fHp-on-product-case}. More precisely, since $0<p\leq 1$,
\begin{align*}
 \|f\|_{L^p(X_1\times X_2)}^p
    &\leq  \sum_{\ell_1,\ell_2\geq 0} 2^{-\ell_1\gamma_1p}2^{-\ell_2\gamma_2p}
        \|f_{\ell_1,\ell_2}\|_{L^p(X_1\times X_2)}^p\\
    &\leq  \sum_{\ell_1,\ell_2\geq 0} 2^{-\ell_1\gamma_1p}2^{-\ell_2\gamma_2p} \sum_{j\in\mathbb{Z}}
        \big\| \hskip -.1in \sum_{R^{k_1,k_2}_{\alpha_1,\alpha_2}\in \mathcal{B}_j}
       \hskip -.15 in \langle f,\psi_{\alpha_1}^{k_1}\psi_{\alpha_2}^{k_2} \rangle
        \kappa_1\varphi^{\gamma_1}_{\ell_1,k_1,\alpha_1}\kappa_2\varphi^{\gamma_2}_{\ell_2,k_2,\alpha_2}
        \big\|_{L^p(X_1\times X_2)}^p\\
    &\lesssim \sum_{\ell_1,\ell_2\geq 0} 2^{-\ell_1\gamma_1p} 2^{-\ell_2\gamma_2p}
        {(1+\ell_1 \omega_1+\ell_2 \omega_2 )^{1-\frac{p}{q}}}2^{\ell_1 \omega_1 (1+\frac{p}{{q'}})} 2^{\ell_2 \omega_2 (1+\frac{p}{{q'}})}
        \sum_{j\in\mathbb{Z}} 2^{jp}\mu(\Omega_j)\\
    &\lesssim \|S(f)\|_{L^p(X_1\times X_2)}^p = \|f\|_{H^p(X_1\times X_2)}^p.
    \end{align*}
Where we have chosen $\gamma_i > \omega_i(1/p+1/{q'})$  for~$i=1,2$\label{constraintThm4.2}, to ensure convergence of the relevant  series over $\ell_1$ and $\ell_2$. {Note that since $1/p\geq 1$,
 this constraint implies that $\gamma_i>\omega_i$ for $i=1,2$, a constraint needed in Lemma~\ref{lemma-decomposition}.}

Thus, it suffices to verify the claim~\eqref{claim Lp}. To this end, we  define
the $\epsilon_0$-\emph{enlargement} $\widetilde{\Omega}_j:= \widetilde{ \Omega}_j^{\epsilon_0}$ of the open set  $\Omega_j$ by
\begin{equation}\label{enlargement Omega j}
    \widetilde{\Omega}_j
    :=\bigg\{(x_1,x_2) \in X_1\times X_2: M_s(\chi_{\Omega_{j}})(x_1,x_2)> \epsilon_0 :=\frac{1}{2 C_{\mu_1} C_{\mu_2} }
    \Big ( \frac{c_1^1}{C^1_1}\Big )^{\omega_1}\Big ( \frac{c_1^2}{C_1^2}\Big )^{\omega_2}
    \bigg\}.
\end{equation}
Here $c_1^i, C_1^i$ are the dilation constants of the grids $\mathscr{D}^W_i$ and  $M_s$ is the \emph{strong maximal function}
\[ M_sg(x_1,x_2):= \sup_{B_1\times B_2 \ni (x_1,x_2)}  \frac{1}{\mu_1(B_1)\mu_2(B_2)} \int_{B_1\times B_2} |g(y_1,y_2)| \,d\mu_1(y_1) d\mu_2(y_2),\]
defined for  functions $g\in L^1_{{\rm loc}}(\XX_1\times\XX_2)$, and where $B_i$ are balls in $X_i$ for $i=1,2$.

The constant $\epsilon_0$ in \eqref{enlargement Omega j}  is determined by the doubling constants of the measures $\mu_i$, the upper dimensions $\omega_i$, and the ratio of the dilation constants $c_1^i=(A_0^{(i)})^{-5}/6$ and $C_1^i=6(A_0^{(i)})^4$ involved in the radius of the inner and outer balls sandwiching the reference dyadic cubes  for the wavelets, as in property~\eqref{prop_cube3},  for $i=1,2$.   More precisely,
$\epsilon_0$ is a constant depending only on the geometric constants of $X_i$ for $i=1,2$,
\begin{equation}\label{epsilon-0}
\epsilon_0 = \Big ( 2 C_{\mu_1} C_{\mu_2} \big(36(A_0^{(1)})^9\big)^{\omega_1} \big (36(A_0^{(2)})^9\big )^{\omega_2}\Big )^{-1}.
\end{equation}
Furthermore  $\epsilon_0\in (0,1)$  and   is chosen so that if $R\in \mathcal{B}_j$ then $R\subset \widetilde{\Omega}_j$. More precisely, if $R\in \mathcal{B}_j$ then by definition $\mu (R\cap \Omega_j ) / \mu(R) >1/2$.
The dyadic rectangle $R=Q_1\times Q_2$ and, for $i=1,2$, each dyadic cube $Q_i\in \mathscr{D}^W_i$ contains $B'_i$, its inner ball, and is contained in $B''_i$, its outer ball, that is
$B'_i\subset Q_i \subset B''_i$. Moreover, $\mu_i(B''_i ) \leq C_{\mu_i} \Big ( \frac{C_1^i}{c_1^i}\Big )^{\omega_i} \mu_i (B'_i)$ by the doubling property~\eqref{eqn:upper dimension}  of the measure $\mu_i$ for $i=1,2$. Hence
\[ \frac{1}{2} < \frac{ \mu (R\cap \Omega_j )}{\mu(R) } \leq \frac{ \mu \big ( (B_1''\times B_2'') \cap \Omega_j \big )}{\mu_1 (B_1') \, \mu_2 (B_2')} \leq C_{\mu_1} C_{\mu_2} \Big ( \frac{C_1^1}{c_1^1}\Big )^{\omega_1} \Big ( \frac{C_1^2}{c_1^2}\Big )^{\omega_2 } \frac{ \mu \big ( (B_1''\times B_2'') \cap \Omega_j \big )}{\mu_1 (B_1'') \, \mu_2 (B_2'')}.\]
We conclude that $B_1''\times B_2''\subset \widetilde{\Omega}_j$ and therefore $R=Q_1\times Q_2 \subset \widetilde{\Omega}_j$. \label{RinBj}

By definition every open set $\Omega$ is contained in its $\epsilon$-enlargement
\begin{equation}\label{epsilon-enlargement}
\widetilde{\Omega}^{\epsilon}:=\big\{(x_1,x_2)\in X_1\times X_2: M_s(\chi_{\Omega})(x_1,x_2)> \epsilon \big\}
\end{equation}
for $\epsilon \in (0,1)$, that is $\Omega \subset \widetilde{\Omega}^{\epsilon}$. In particular $\Omega_j\subset \widetilde{\Omega}_j$  and hence $\mu(\Omega_j)\leq \mu( \widetilde{\Omega}_j)$ for all $j\geq 0$. {More interestingly, by weak-$L^2$ properties of the strong maximal function we get
\begin{equation}\label{weak-L^2-Ms}
\mu( \widetilde{\Omega}_j) \leq C\Big (\frac{\| \chi_{\Omega_j}\|_{L^2(X_1\times X_2)} }{\epsilon_0}\Big )^2 = \frac{C}{\epsilon_0^2} \; \mu (\Omega_j).
\end{equation}}
We also  define the $(\ell_1,\ell_2)$-\emph{enlargement} $\widetilde{\Omega}_{j,\ell_1,\ell_2}$  of $\widetilde{\Omega}_j$. Recall that $2^{\ell_i} Q_i:= B(y_{\alpha_i}^{k_i}, 2^{\ell_i} C^i_1\delta^{k_i})$, where $C_1^i$ is the dilation constant determining the radius of the outer ball of the dyadic cube $Q_i\in\mathcal{D}^W_i$ for each $i=1,2$.  Let
\begin{equation}\label{(ell1,ell2)-enlargement}
    \widetilde{\Omega}_{j,\ell_1,\ell_2}
    := \bigcup_{R=Q_1\times Q_2
        \subset\widetilde{\Omega}_j}
        2^{\ell_1}Q_1\times
        2^{\ell_2}Q_2.
\end{equation}
It is clear from this definition that $\widetilde{\Omega}_j\subset \widetilde{\Omega}_{j,\ell_1,\ell_2}$ for all $\ell_1, \ell_2 \geq 0$. Note that $ \widetilde{\Omega}_{j,\ell_1,\ell_2}$ is a subset of $\{ (x_1,x_2)\in X_1\times X_2: M_s(\chi_{\widetilde{\Omega}_j})(x_1,x_2) \geq 2^{-\ell_1\omega_1-\ell_2\omega_2}\}$. 
Indeed, for every $(x_1,x_2)\in \widetilde{\Omega}_{j,\ell_1,\ell_2}$ there most be a dyadic rectangle $R=Q_1\times Q_2 \in \widetilde{\Omega}_j$ such that $(x_1,x_2) \in 2^{\ell_1}Q_1\times 2^{\ell_2}Q_2$. Also for $2^{\ell_1}Q_1\times 2^{\ell_2}Q_2$ we get
\[
\frac{\mu\big (\widetilde{\Omega}_j\cap (2^{\ell_1}Q_1\times 2^{\ell_2}Q_2)\big )}{\mu (2^{\ell_1}Q_1\times 2^{\ell_2}Q_2)}  \geq \frac{\mu \big (\widetilde{\Omega}_j\cap (Q_1\times Q_2)\big )}{2^{\ell_1\omega_1+\ell_2\omega_2}\mu ( Q_1\times Q_2)}
= \frac{1}{2^{\ell_1\omega_1+\ell_2\omega_2}}.
\]
Hence $M_s(\chi_{\widetilde{\Omega}_j})(x_1,x_2)\geq 2^{-\ell_1\omega_1-\ell_2\omega_2}$.
We conclude that
\begin{equation}\label{eqn:measure-enlargements}
\mu (\widetilde{\Omega}_{j,\ell_1,\ell_2}) \lesssim (1+\ell_1\omega_1 +\ell_2\omega_2)2^{\ell_1\omega_1} 2^{\ell_2\omega_2} \mu (\widetilde{\Omega}_j ).
\end{equation}
by an argument similar to \cite[p.191, line 17]{CF}, using the $L\log_+ L$ to weak $L^1$ estimate for the strong maximal function applied to $f=\chi_{\widetilde{\Omega}_j}$, namely
\begin{equation}\label{LlogL}
\mu \{(x_1,x_2)\in \widetilde{X}: M_s(f)(x_1,x_2)>\lambda\} \lesssim  \int_{\widetilde{X}} \frac{|f(x_1,x_2)|}{\lambda } \log \bigg ( 1+\frac{|f(x_1,x_2)|}{\lambda }\bigg ) \, d\mu(x_1,x_2).
\end{equation}
 The $L\log_+ L$ to weak  $L^1$ estimate~\eqref{LlogL} for the strong maximal function can be deduced for the strong dyadic maximal function (defined as $M_s$ but instead of product of balls we consider products of dyadic cubes in $X_1$ and $X_2$) from the weak $(1,1)$ estimates on each individual dyadic maximal function on $X_i$ for $i=1,2$, see \cite[Theorem 1]{Fa} and also~\cite{F2}.
By \cite[Theorem 3.1(ii)]{KLPW}  we can control pointwise the strong maximal function $M_s$ (with respect to  balls) by a finite sum of strong dyadic maximal functions (with respect to  adjacent systems of dyadic cubes \cite[Section 2.4]{KLPW},
the equivalent to the 1/3 trick in $\mathbb{R}$ for spaces of homogeneous type). Therefore getting the desired estimate~\eqref{LlogL}.

For each set $\mathcal{B}_j$ of dyadic rectangles, we  define the function $f_{\mathcal{B}_j}:X_1\times X_2\to \mathbb{R}$ to be
\begin{equation}\label{function-fBj}
f_{\mathcal{B}_j}(x_1,x_2) := \sum_{R^{k_1,k_2}_{\alpha_1,\alpha_2}\in \mathcal{B}_j}\langle f,\psi_{\alpha_1}^{k_1}\psi_{\alpha_2}^{k_2} \rangle \, \psi_{\alpha_1}^{k_1}(x_1) \, \psi_{\alpha_2}^{k_2}(x_2),
\end{equation} and hence by definition of the square function
\begin{equation}\label{SfBj}
 S(f_{\mathcal{B}_j})(x_1,x_2)= \Big (\sum_{R^{k_1,k_2}_{\alpha_1,\alpha_2}\in \mathcal{B}_j}
 \big|\big\langle f,\widetilde{\psi}_{\alpha_1}^{k_1}\widetilde{\psi}_{\alpha_2}^{k_2}\big\rangle\big|^2 \chi_{R^{k_1,k_2}_{\alpha_1,\alpha_2}}(x_1,x_2)  \Big )^{\frac{1}{2}},
 \end{equation}
where $\widetilde{\psi}^{k_i}_{\alpha_i}={\psi_{\alpha_i}^{k_i}/ \kappa_i}$ denotes the normalized wavelets for $i=1,2$.

Note that  by construction, the function  $\varphi^{\gamma_i}_{\ell_i,k_i,\alpha_i}$  has compact support on
$B(y_{\alpha_i}^{k_i}, {2(A_0^{(i)})^2}\, 2^{\ell_i}\delta^{k_i})$ which is contained in $B(y_{\alpha_i}^{k_i},2^{\ell_i}C^i_1\delta^{k_i})$ for $i=1,2$.  {The last statement holds since in the Auscher-Hyt\"onen construction the dilation constant $C^i_1$ determining the radius of the outer balls is $C^i_1=6(A_0^{(i)})^4 > 2(A_0^{(i)})^2$ for each $i=1,2$ \cite[Theorem 2.11]{AH}.}
As explained in page~\pageref{RinBj}, if  $R^{k_1,k_2}_{\alpha_1,\alpha_2}\in \mathcal{B}_j$, then $R^{k_1,k_2}_{\alpha_1,\alpha_2}\in \widetilde{\Omega}_j$, and thus the support of
$\varphi^{\gamma_1}_{\ell_1,k_1,\alpha_1}(x_1)\,\varphi^{\gamma_2}_{\ell_2,k_2,\alpha_2}(x_2)$ is
contained in $\widetilde{\Omega}_{j,\ell_1,\ell_2}$.

Therefore, by H\"older's inequality {with exponents $s=q/p>1$ and $s'=q/(q-p)$,}
\begin{align}
    &\Big\|\sum_{R_{\alpha_1,\alpha_2}^{k_1,k_2}\in \mathcal{B}_j}
        \langle f,\psi_{\alpha_1}^{k_1}\psi_{\alpha_2}^{k_2} \rangle
        \kappa_1\varphi^{\gamma_1}_{\ell_1,k_1,\alpha_1}\kappa_2\varphi^{\gamma_2}_{\ell_2,k_2,\alpha_2}
        \Big\|_{L^p(X_1\times X_2)}^p \nonumber\\
    &\hskip.5cm\leq
        \mu(\widetilde{\Omega}_{j,\ell_1,\ell_2})^{1-\frac{p}{q}}\Big\|
        \sum_{R^{k_1,k_2}_{\alpha_1,\alpha_2}\in \mathcal{B}_j}
        \langle f,\psi_{\alpha_1}^{k_1}\psi_{\alpha_2}^{k_2} \rangle
        \kappa_1\varphi^{\gamma_1}_{\ell_1,k_1,\alpha_1}\kappa_2\varphi^{\gamma_2}_{\ell_2,k_2,\alpha_2}
        \Big\|_{L^{q}(X_1\times X_2)}^{p}. \label{Lp-estimate-Bj}
\end{align}
To estimate the $L^q$-norm of the sum in the right-hand-side of~\eqref{Lp-estimate-Bj} we use a duality argument. Hence, for all $g\in L^{{q'}}(X_1\times
X_2)$ with $\|g\|_{L^{{q'}}(X_1\times X_2)}\leq 1$,   we estimate  the inner product
\begin{align*}
    &\Big|\Big\langle \sum_{R_{\alpha_1,\alpha_2}^{k_1,k_2}\in \mathcal{B}_j}
        \langle f,\psi_{\alpha_1}^{k_1}\psi_{\alpha_2}^{k_2} \rangle
        \kappa_1\varphi^{\gamma_1}_{\ell_1,k_1,\alpha_1} \kappa_2\varphi^{\gamma_2}_{\ell_2,k_2,\alpha_2},\
        g   \Big\rangle\Big|\\
    &\hskip .5in =\Big|\sum_{R^{k_1,k_2}_{\alpha_1,\alpha_2}\in \mathcal{B}_j} \kappa_1^2 \kappa_2^2
        \big\langle f,\widetilde{\psi}_{\alpha_1}^{k_1}\widetilde{\psi}_{\alpha_2}^{k_2} \big\rangle
        \, \langle \varphi^{\gamma_1}_{\ell_1,k_1,\alpha_1}\varphi^{\gamma_2}_{\ell_2,k_2,\alpha_2},\ g\rangle
        \Big|\\
    & \hskip .5in \leq \sum_{R^{k_1,k_2}_{\alpha_1,\alpha_2}\in \mathcal{B}_j} \mu_1(Q_{\alpha_1}^{k_1})\mu_2(Q_{\alpha_2}^{k_2})
        \Big|\big\langle f,\widetilde{\psi}_{\alpha_1}^{k_1}\widetilde{\psi}_{\alpha_2}^{k_2}  \big\rangle\Big|
        \, \Big|\langle \varphi^{\gamma_1}_{\ell_1,k_1,\alpha_1}\varphi^{\gamma_2}_{\ell_2,k_2,\alpha_2},\ g\rangle
        \Big|\\
    & \hskip .5in \leq \mathop{\int}_{X_1\times X_2}\sum_{R^{k_1,k_2}_{\alpha_1,\alpha_2}\in \mathcal{B}_j}
        \Big|\big\langle f,\widetilde{\psi}_{\alpha_1}^{k_1}\widetilde{\psi}_{\alpha_2}^{k_2} \big\rangle\Big|
        \, \Big|\langle \varphi^{\gamma_1}_{\ell_1,k_1,\alpha_1}\varphi^{\gamma_2}_{\ell_2,k_2,\alpha_2},\ g\rangle
        \Big| \, \chi_{R_{\alpha_1,\alpha_2}^{k_1k_2}}(x_1,x_2) \, d\mu_1(x_1) \,d\mu_2(x_2).
  \end{align*}
In the last inequality  we used that  $\mu_1(Q_{\alpha_1}^{k_1})\mu_2(Q_{\alpha_2}^{k_2})=\int_{X_1\times X_2}\chi_{R^{k_1,k_2}_{\alpha_1,\alpha_2}}(x_1,x_2)\, d\mu_1(x_1)\, d\mu_2(x_2) $.
We  continue estimating,  first  applying the Cauchy-Schwarz inequality on the sum,  second applying H\"older's inequality, with exponents $q>1$ and $q'$, to the integral, and third using the notation introduced in~\eqref{function-fBj} and~\eqref{SfBj},
\begin{align}
&\Big|\Big\langle \sum_{R_{\alpha_1,\alpha_2}^{k_1,k_2}\in \mathcal{B}_j}
        \langle f,\psi_{\alpha_1}^{k_1}\psi_{\alpha_2}^{k_2} \rangle
        \kappa_1\varphi^{\gamma_1}_{\ell_1,k_1,\alpha_1} \kappa_2\varphi^{\gamma_2}_{\ell_2,k_2,\alpha_2},\
        g   \Big\rangle\Big| \nonumber \\
    &\hskip 1in \leq \bigg(\mathop{\int}_{X_1\times X_2} \bigg (\sum_{R^{k_1,k_2}_{\alpha_1,\alpha_2}\in \mathcal{B}_j}
        \Big|\big\langle f,\widetilde{\psi}_{\alpha_1}^{k_1}\widetilde{\psi}_{\alpha_2}^{k_2} \big\rangle\Big|^2
        \chi_{R_{\alpha_1,\alpha_2}^{k_1,k_2}}(x_1,x_2) \bigg )^{{\frac{q}{2}}}d\mu_1(x_1)\, d\mu_2(x_2)\bigg)^{\frac{1}{q}} \nonumber \\
    &\hskip .8in \times \bigg(\mathop{\int}_{X_1\times X_2}\bigg (\sum_{R^{k_1,k_2}_{\alpha_1,\alpha_2}\in \mathcal{B}_j}
         \Big|\langle \varphi^{\gamma_1}_{\ell_1,k_1,\alpha_1}\varphi^{\gamma_2}_{\ell_2,k_2,\alpha_2},\ g\rangle
        \Big|^2 \chi_{R_{\alpha_1,\alpha_2}^{k_1,k_2}}(x_1,x_2) \bigg )^{{\frac{q'}{2}}}d\mu_1(x_1) \, d\mu_2(x_2) \bigg)^{\frac{1}{q'}} \nonumber\\
    &\hskip 1in \lesssim_q 2^{\ell_1 \omega_1}2^{\ell_2 \omega_2}  \|S(f_{\mathcal{B}_j})\|_{L^q(X_1\times X_2)}.
    \label{dual-estimate}
\end{align}
The last inequality is deduced from  the  fact that   $\|g\|_{L^{q'}(X_1\times X_2)}\leq 1$ and the following
Littlewood--Paley estimate, whose proof will be provided after  finishing the proof of Theorem~\ref{theorem-of-fLp-lessthan-fHp-on-product-case}.
\begin{lemma}\label{lemma LittlewoodPaley}
There is a constant $C>0$ {\rm (}depending on the geometric constants and on $q>1${\rm )} such that for  all functions $g\in L^{{q'}}(X_1\times X_2)$ and all positive integers $\ell_1$ and $\ell_2$,
\begin{equation}\label{Littlewood-Paley Key Lemma}
 \bigg \| \Big (\sum_{R^{k_1,k_2}_{\alpha_1,\alpha_2}\in \mathcal{B}_j}
         \Big|\langle \varphi^{\gamma_1}_{\ell_1,k_1,\alpha_1}\varphi^{\gamma_2}_{\ell_2,k_2,\alpha_2},\ g\rangle
        \Big|^2 \chi_{R_{\alpha_1,\alpha_2}^{k_1,k_2}}\Big )^{1/2} \bigg \|_{L^{{q'}}(X_1\times X_2)} \lesssim_q
         2^{\ell_1\omega_1}2^{\ell_2\omega_2} \|g\|_{L^{{q'}}(X_1\times X_2)}.
\end{equation}
\end{lemma}

 The dual estimate~\eqref{dual-estimate}  implies that
\begin{align}
    &\Big\|\sum_{R^{k_1,k_2}_{\alpha_1,\alpha_2}\in \mathcal{B}_j}
        \langle f,\psi_{\alpha_1}^{k_1}\psi_{\alpha_2}^{k_2} \rangle
        \kappa_1\varphi^{\gamma_1}_{\ell_1,k_1,\alpha_1}\kappa_2\varphi^{\gamma_2}_{\ell_2,k_2,\alpha_2}
        \Big\|_{L^{{q}}(X_1\times X_2)}\label{dual-estimate2}
         \lesssim_q 2^{\ell_1 \omega_1}2^{\ell_2 \omega_2}  \|S(f_{\mathcal{B}_j})\|_{L^q(X_1\times X_2)} \\
    & =   2^{\ell_1 \omega_1}2^{\ell_2 \omega_2}
       \bigg(\mathop{\int}_{X_1\times X_2}\bigg \{\sum_{R^{k_1,k_2}_{\alpha_1,\alpha_2}\in \mathcal{B}_j}
        \Big|\big\langle f,\widetilde{\psi}_{\alpha_1}^{k_1}\widetilde{\psi}_{\alpha_2}^{k_2}\big\rangle\Big|^2
        \chi_{R_{\alpha_1,\alpha_2}^{k_1,k_2}}(x_1,x_2) \bigg \}^{{\frac{q}{2}}} d\mu_1(x_1)\,d\mu_2(x_2)\bigg)^{\frac{1}{q}}\nonumber \\
     &\lesssim_q 2^{\ell_1 \omega_1}2^{\ell_2 \omega_2} \bigg(\mathop{\int}_{X_1\times X_2}\bigg \{\sum_{R=R^{k_1,k_2}_{\alpha_1,\alpha_2}\in \mathcal{B}_j}
        \Big|\big\langle f,\widetilde{\psi}_{\alpha_1}^{k_1}\widetilde{\psi}_{\alpha_2}^{k_2} \big\rangle\Big|^2 \nonumber \\
    & \hskip 2.0in \times    \Big | {M_s\big (\chi_{R\cap (\widetilde{\Omega}_j\setminus\Omega_{j+1})}  \big ) (x_1,x_2)}\Big |^2
         \bigg \}^{{\frac{q}{2}}} d\mu_1(x_1)\,d\mu_2(x_2)\bigg)^{\frac{1}{q}} \nonumber
            \end{align}
{In the last inequality we have used the definitions~\eqref{setBj},     of  the set $\mathcal{B}_j$, and~\eqref{enlargement Omega j}, of  the enlargement set $\widetilde{\Omega}_j$ via the strong maximal function, to deduce that $$\chi_R(x_1,x_2)\lesssim\big |M_s\big (\chi_{R\cap (\widetilde{\Omega}_j\setminus\Omega_{j+1})})  \big  (x_1,x_2)\big |^2.$$ More precisely, recall that if $R=Q_1\times Q_2$
belongs to  $\mathcal{B}_j$ then   it is a subset of $\widetilde{\Omega}_j$. Hence $R\cap (\widetilde{\Omega}_j\setminus\Omega_{j+1}) = R\setminus\Omega_{j+1}$, and since $R\in \mathcal{B}_j$ it is also true that  $\mu (R\cap\Omega_{j+1}) \leq \frac{1}{2} \mu (R)$. Therefore  $\mu (R\setminus\Omega_{j+1}) \geq \frac{1}{2} \mu (R)$.  \label{R-minus-Omega}
As before, denote by $B_i'$ and $B_i''$ the inner and outer balls    of the  dyadic cubes $Q_i$ for $i=1,2$, recall that $B_i'\subset Q_i \subset B_i''$, therefore $R\setminus\Omega_{j+1}\subset B_1''\times B_2''$. Using the doubling property~\eqref{eqn:upper dimension} we get for $R\in \mathcal{B}_j$
\begin{eqnarray*}
\frac{1}{\mu (B_1''\times B_2'')} \int_{B_1''\times B_2''}\chi_{R \cap (\widetilde{\Omega}_j\setminus \Omega_{j+1})}(z_1,z_2) \,d\mu(z_1,z_2)
& =  & \frac{\mu (R\setminus \Omega_{j+1})}{\mu (B_1''\times B_2'')}
\; \geq \; \frac12 \frac{\mu (R)}{\mu (B_1''\times B_2'')}\\
& & \hskip -2.8in \geq   \frac12\frac{ \mu( B_1'\times B_2')}{\mu (B_1''\times B_2'')}
\; = \; \frac{1}{2} \;\frac{\mu_1(B_1')}{\mu_1(B_1'')}\;\frac{\mu_2(B_2')}{ \mu_2(B_2'')}
\; \geq \; \frac{1}{2 C_{\mu_1}C_{\mu_2}}\Big [ \frac{c^1_1}{C^1_1}\Big ]^{\omega_1}\Big [ \frac{c_1^2}{C_1^2}\Big ]^{\omega_2} = \epsilon_0.
\end{eqnarray*}
 Therefore, for all  $R\in \mathcal{B}_j$ and for all  $(x_1,x_2)\in R$, we get  \[M_s\big (\chi_{R\cap (\widetilde{\Omega}_j\setminus\Omega_{j+1})} \big ) (x_1,x_2)\geq  \epsilon_0>0.\]
 Hence we obtain
 $\chi_R(x_1,x_2)=\chi_R^2(x_1,x_2)\lesssim \big |M_s\big (\chi_{R\cap (\widetilde{\Omega}_j\setminus\Omega_{j+1})})  \big  (x_1,x_2)\big |^2$,} as claimed. Note that the similarity constant is~$\epsilon_0^{-2}$, which only depends  on the geometric constants of~$X_i$ for~$i=1,2$, by definition~\eqref{epsilon-0}.

Recall the Fefferman-Stein vector-valued strong maximal function estimate  in \cite{FS},
 given $q,r>1$, there is a constant $C_q>0$ such that for appropriate sequences of functions $\{f_k\}_{k\geq 1}$
\begin{equation}\label{vector-valued-FS}
\Big \| \big \{ \sum_{k=1}^{\infty} M_s (f_k)^r\big \}^{1/r}\Big \|_{L^q(X_1\times X_2)} \leq C_q  \Big \| \big \{ \sum_{k=1}^{\infty} |f_k|^r\big \}^{1/r}\Big \|_{L^q(X_1\times X_2)}.
\end{equation}
We use estimate~\eqref{vector-valued-FS} with $r=2$ and $q>1$, and we denote $\widetilde{X}=X_1\times X_2$,
to conclude that
 \begin{align}
      \|S(f_{\mathcal{B}_j})\|_{L^q(\widetilde{X})} 
 &\lesssim_q 
 \bigg(\mathop{\int}_{\widetilde{X}}\bigg (\sum_{R=R^{k_1,k_2}_{\alpha_1,\alpha_2}\in \mathcal{B}_j}
        \Big|\big\langle f,\widetilde{\psi}_{\alpha_1}^{k_1}\widetilde{\psi}_{\alpha_2}^{k_2}\big\rangle\Big|^2 \chi_{R\cap (\widetilde{\Omega}_j\setminus\Omega_{j+1})}  (x_1,x_2)
         \bigg )^{{\frac{q}{2}}} d\mu_1(x_1) \,d\mu_2(x_2)\bigg)^{\frac{1}{q}} \nonumber\\
 &= 
  \bigg(\mathop{\int}_{\widetilde{\Omega}_j\setminus\Omega_{j+1}}\bigg (\sum_{R=R^{k_1,k_2}_{\alpha_1,\alpha_2}\in \mathcal{B}_j}
        \Big|\big\langle f,\widetilde{\psi}_{\alpha_1}^{k_1}\widetilde{\psi}_{\alpha_2}^{k_2} \big\rangle\Big|^2 \chi_{R}  (x_1,x_2)
         \bigg )^{{\frac{q}{2}}} d\mu_1(x_1)\, d\mu_2(x_2)\bigg)^{\frac{1}{q}}\nonumber \\
    & = 
     \bigg(\mathop{\int}_{\widetilde{\Omega}_j\setminus\Omega_{j+1}} |S(f_{\mathcal{B}_j})(x_1,x_2)|^q \, d\mu_1(x_1) \, d\mu_2(x_2)\bigg)^{\frac{1}{q}}.
          \label{eqn:Lq-estimate-a/lambda}
                    \end{align} 
The function $f_{\mathcal{B}_j}$ was defined in~\eqref{function-fBj}, and its square function $S(f_{\mathcal{B}_j})$ in \eqref{SfBj}. Note that pointwise $S(f_{\mathcal{B}_j})\leq S(f)$. Moreover when  $(x_1,x_2)\notin \Omega_{j+1}$  by definition $S(f)(x_1,x_2) \leq 2^{j+1}$. Therefore,
\begin{equation}\label{eqn:SBj-less-muOmegaj}
\|S(f_{\mathcal{B}_j})\|_{L^q(X_1\times X_2)}
  \lesssim_q 2^{j}\mu(\widetilde{\Omega}_j)^{1/q}. \end{equation}
All together we conclude that
\begin{equation} \label{Lq-estimate-Bj}
\Big\|\sum_{R^{k_1,k_2}_{\alpha_1,\alpha_2}\in \mathcal{B}_j}
        \langle f,\psi_{\alpha_1}^{k_1}\psi_{\alpha_2}^{k_2} \rangle
        \kappa_1\varphi^{\gamma_1}_{\ell_1,k_1,\alpha_1}\kappa_2\varphi^{\gamma_2}_{\ell_2,k_2,\alpha_2}
        \Big\|_{L^{{q}}(X_1\times X_2)}
         \lesssim_q 2^{\ell_1 \omega_1}2^{\ell_2 \omega_2}      2^{j}\mu(\widetilde{\Omega}_j)^{\frac{1}{q}}.
 \end{equation}

Finally,  first using estimates~\eqref{Lp-estimate-Bj} and ~\eqref{Lq-estimate-Bj}, and second using estimate~\eqref{eqn:measure-enlargements}, we get the $L^p$-estimate claimed in~\eqref{claim Lp}
\begin{align*}
 &   \Big\|\sum_{R^{k_1,k_2}_{\alpha_1,\alpha_2}\in \mathcal{B}_j}
        \langle f,\psi_{\alpha_1}^{k_1}\psi_{\alpha_2}^{k_2} \rangle
        \kappa_1\varphi^{\gamma_1}_{\ell_1,k_1,\alpha_1}\kappa_2\varphi^{\gamma_2}_{\ell_2,k_2,\alpha_2}
        \Big\|_{L^p(X_1\times X_2)}^p
    \lesssim_q \mu(\widetilde{\Omega}_{j,\ell_1,\ell_2})^{1-\frac{p}{q}} 2^{\ell_1\omega_1p}2^{\ell_2\omega_2p} 2^{jp}
        \mu(\widetilde{\Omega}_j)^{\frac{p}{q}}\\
    & \hskip 3cm \lesssim_q (1+\ell_1 \omega_1+\ell_2 \omega_2 )^{1-\frac{p}{q}}2^{\ell_1 \omega_1 (1-\frac{p}{q})}\,2^{\ell_2 \omega_2 (1-\frac{p}{q})}\,
        \mu(\widetilde{\Omega}_{j})^{1-\frac{p}{q}}\,
        2^{\ell_1 \omega_1p}\,2^{\ell_2 \omega_2p}\,
        2^{jp}\,
        \mu(\widetilde{\Omega}_j)^{\frac{p}{q}}\\
         &\hskip 3cm\lesssim_q (1+\ell_1 \omega_1+\ell_2 \omega_2 )^{1-\frac{p}{q}}2^{\ell_1 \omega_1 (1+\frac{p}{q'})}2^{\ell_2 \omega_2 (1+\frac{p}{q'})}
        2^{jp} \mu(\Omega_{j}).
\end{align*}
Where the last estimate follows from $\mu(\widetilde{\Omega}_j)\lesssim \mu (\Omega_j)$ by~\eqref{weak-L^2-Ms}.
Note that all constants depend only on geometric constants of $X_i$ for $i=1,2$, sometimes via the parameter $\epsilon_0$ defined in~\eqref{epsilon-0}.
This estimate finishes the proof of the claim~\eqref{claim Lp},
and hence
Theorem~\ref{theorem-of-fLp-lessthan-fHp-on-product-case} is
proved.
\end{proof}

\begin{proof}[Proof of Lemma~\ref{lemma LittlewoodPaley}]\label{proof-Lemma-LP}

Estimate \eqref{Littlewood-Paley Key Lemma} can be  established
using an argument similar to the one made when proving the second inequality in the product Plancherel--P\'olya inequalities from~\cite[Theorem 4.9, equation~(4.13)]{HLW}. More specifically, there are sufficiently large integers $N_i>0$ for $i=1,2$, and a constant $C_q>0$ (depending only on the geometric constants of $X_i$ for $i=1,2$ and $q>1$) such that for all $g\in L^{q'}(X_1\times X_2)$ the following inequality holds:
\begin{align}
&\bigg\|\bigg\{ \sum_{k_1,k_2}\sum_{\alpha_1\in \mathscr{Y}^{k_1},\alpha_2\in \mathscr{Y}^{k_2}}  \Big|\langle \varphi^{\gamma_1}_{\ell_1,k_1,\alpha_1}\varphi^{\gamma_2}_{\ell_2,k_2,\alpha_2},\ g\rangle
        \Big|^2 \chi_{R_{\alpha_1,\alpha_2}^{k_1,k_2}} \bigg\}^{\frac{1}{2}}\bigg\|_{L^{q'}(X_1\times X_2)} \label{PP1}\\
&   \leq C_q  2^{\ell_1w_1}2^{\ell_2w_2}
      \bigg\| \bigg\{  \sum_{k_1,k_2}\sum_{\substack{\alpha_1\in \mathscr{X}^{k_1+N_1}\\ \alpha_2\in \mathscr{X}^{k_2+N_2} }}  \inf_{\substack{z_1\in Q_{\alpha_1}^{k_1+N_1}\\z_2\in Q_{\alpha_2}^{k_2+N_2}} }
 |D^{(1)}_{k_1}D^{(2)}_{k_2}(g)(z_1,z_2)|^2\chi_{Q_{\alpha_1}^{k_1+N_1}}\chi_{Q_{\alpha_2}^{k_2+N_2}}  \bigg\}^{\frac{1}{2}} \bigg\|_{L^{q'}(X_1\times X_2)}, \nonumber
 \end{align}
where $D^{(1)}_{k_1}$ is the integral operator in $X_1$ with kernel $D^{(1)}_{k_1}(x,y) = \sum_{\beta_1\in  \mathscr{Y}^{k_1}} \psi_{\beta_1}^{k_1}(x)\psi_{\beta_1}^{k_1}(y)$, and similarly for $D^{(2)}_{k_2}$.
The statement in~\cite[Theorem 4.9]{HLW} refers to  Plancherel-P\'olya inequalities  with the wavelets $\psi^{k_i}_{\alpha_i}$ instead  of the functions $\varphi^{\gamma_i}_{\ell_i,k_i,\alpha_i}$ on the left-hand-side of equation~\eqref{PP1}. However, carefully tracing the proof of~\cite[Equation (4.13)]{HLW}, one realizes that  all is required are the size, smoothness, and cancellation conditions  of the functions~$\varphi^{\gamma_i}_{\ell_i,k_i,\alpha_i}$ (proved in Lemma~\ref{lemma-decomposition})  and of the kernels~$D^{(i)}_{k_i}(x,y)$ for $i=1,2$ (proved in~\cite[Lemma~3.6]{HLW}).  The key observations are first, for every $(y_1,y_2)\in X_1\times X_2$
\[ \big\langle \varphi^{\gamma_1}_{\ell_1,k_1,\alpha_1}\varphi^{\gamma_2}_{\ell_2,k_2,\alpha_2}, D^{(1)}_{k_1}D^{(2)}_{k_2}(\cdot,y_1,\cdot, y_2)\big\rangle_{X_1\times X_2} = \big\langle \varphi^{\gamma_1}_{\ell_1,k_1,\alpha_1},D^{(1)}_{k_1}(\cdot, y_1)\big\rangle_{X_1} \,\big\langle
\varphi^{\gamma_2}_{\ell_2,k_2,\alpha_2}, D^{(2)}_{k_2}(\cdot, y_2)\big\rangle_{X_2}.\]
Second,  the following almost-orthogonality estimate is valid for $i=1,2$: for all integers $k_i$ and  $k'_i$  let $\delta'_i:= \delta_i^{\min\{k_i,k'_i\}}$,  where $\delta_i$ is the base side length for the reference dyadic cubes in $X_i$. Then 
for each positive integer $N_i$, each $\gamma>0$, each point $z\in Q^{k'_i+N_i}_{\alpha'_i}\subset X_i$ and each center point $x^{k'_i+N_i}_{\alpha'_i}\in Q^{k'_i+N_i}_{\alpha'_i}$
\begin{equation}\label{almost-orthogonality-claim}
|\langle \varphi^{\gamma_i}_{\ell_i,k_i,\alpha_i}(\cdot), D^{(i)}_{k'_i}(\cdot, z)\rangle | \lesssim
\frac{2^{\ell_i\omega_i}\delta_i^{|k_i-k'_i|\eta}}{V_{\delta'_i}(x^{k_i}_{\alpha_i})+V_{\delta'_i}(x^{k'_i+N_i}_{\alpha'_i}) + V(x^{k_i}_{\alpha_i},x^{k'_i+N_i}_{\alpha'_i})}\Big (\frac{\delta'_i}{\delta'_i+ d_i(x^{k_i}_{\alpha_i},x^{k'_i+N_i}_{\alpha'_i})}\Big )^{\gamma}.
\end{equation}
Where  $V_{r_i}(x_i)=\mu_i\big (B_{X_i}(x_i,r_i)\big )$, $V(x_i,y_i) = \mu_i\big (B_{X_i}(x_i,d_i(x_i,y_i)\big )$, and the similarity constants depend only on the geometric constants of $X_i$ for $i=1,2$.
This estimate is analogue to estimate~\cite[Equation (4.4)]{HLW} with the functions $\varphi$ instead of the wavelets on the left-hand-side of the inner product. It is in proving  estimate~\eqref{almost-orthogonality-claim} that  the size, smoothness, and cancellation properties of the functions $\varphi^{\gamma_i}_{\ell_i,k_i,\alpha_i}$ are needed. Also  needed are the corresponding properties  for the kernels of the operators~$D^{(i)}_{k'_i}$ stablished in~\cite[Lemma~3.6]{HLW}.  

The right-hand-side of~\eqref{PP1} is pointwise bounded by the same expression where the infimum  in the sum is replaced by the supremum.
Another application of Plancherel-P\'olya as stated in \cite[Equation~(4.12)]{HLW} shows that
for all  positive integers $N_1$ and $N_2$ there is a constant $C_q>0$ (depending only on geometric constants and $q>1$) such that
\begin{equation}\label{PP2}
\bigg\| \bigg\{  \sum_{k_1,k_2}\sum_{\substack{\alpha_1\in \mathscr{X}^{k_1+N_1}\\ \alpha_2\in \mathscr{X}^{k_2+N_2} }}  
\sup_{\substack{z_1\in Q_{\alpha_1}^{k_1+N_1}\\z_2\in Q_{\alpha_2}^{k_2+N_2}} } |D^{(1)}_{k_1}D^{(2)}_{k_2}(g)(z_1,z_2)|^2\chi_{Q_{\alpha_1}^{k_1+N_1}}\chi_{Q_{\alpha_2}^{k_2+N_2}}\bigg\}^{1\over2} \bigg\|_{L^{q'}(\widetilde{X})} \leq C_q \| S(g)\|_{L^{q'}(\widetilde{X})}.
 \end{equation}
Here  $S(g)$ is the product Littlewood-Paley square function of $g$  as in Definition~\ref{g function}. This time there are wavelets on both sides of~\eqref{PP2} exactly as in~\cite[Equation~(4.12)]{HLW}.

Based on~\eqref{PP1} and the product Plancherel--P\'olya inequality~\eqref{PP2} we see that the left-hand side of \eqref{Littlewood-Paley Key Lemma}
is bounded by the  $L^{q'}$-norm of  $S(g)$.
From Theorem 4.8 in~\cite{HLW},  since we are in the case $q'>1$, we obtain that
$$\| S(g) \|_{L^{q'}(X_1\times X_2)} 
 \leq C_{q'} \|g\|_{L^{q'}(X_1\times X_2)}.  $$
Putting all the pieces together we get estimate~\eqref{Littlewood-Paley Key Lemma}, with a constant $C>0$ that depends only on the  geometric constants of $X_i$ for $i=1,2$ and on $q>1$. This
finishes the proof of Lemma~\ref{lemma LittlewoodPaley}.
\end{proof}

\section{Atomic product Hardy spaces}\label{sec:atomicHp}

We now provide an atomic decomposition for $H^p( X_1\times X_2
).$ More precisely, we will find an atomic decomposition for
each function $f\in L^q( X_1\times X_2 )\cap H^p( X_1\times X_2 )$ with
$1<q<\infty$ and  $p_0<p\leq 1$, where the decomposition converges  both in the
$L^q( X_1\times X_2 )$-norm and $H^p( X_1\times X_2 )$-(semi)norm.
Recall that $p_0:=\max\{{\omega_i}/{(\omega_i+\eta_i)}: i=1,2\}$.
To achieve this decomposition we will need a Journ\'e-type
covering lemma and a suitable definition of product
$(p,q)$-atoms on $X_1\times X_2$, valid for  $(X_i,d_i,\mu_i)$ spaces of
homogeneous type in the sense of Coifman and Weiss for $i=1,2$.
We will also define atomic product Hardy spaces $H^{p,q}_{{\rm
at}}(X_1\times X_2)$, and as a consequence of the main theorem
we will show these spaces coincide with $H^p(X_1\times X_2)$
for all $q>1$.

The definition of the product Hardy spaces $H^p(X_1\times
X_2)$ utilizes Auscher-Hyt\"onen wavelet bases on each space of
homogeneous type $X_i$, with H\"older regularity $\eta_i\in (0,1]$,  and
corresponding reference dyadic grids $\mathscr{D}^W_i$, for  $i=1,2$, provided $p>p_0$.
In this section we
will show that  functions in $H^p(X_1\times X_2)\cap
L^q(X_1\times X_2)$ can be decomposed into  product
$(p,q)$-atoms based on the wavelets' reference dyadic grids $\mathscr{D}^W_i$ for $i=1,2$.
  Product $(p,q)$-atoms do not require wavelets in
their definition,  but there is an underlying dyadic grid
associated to each atom. We will show that   product
$(p,q)$-atoms,  based on regular families of  dyadic grids,   are
in $H^p(X_1\times X_2)$ with uniform bounds on their $H^p$-(semi)norm
dependent only on the geometric constants of the spaces $X_i$ for $i=1,2$.
  These observations allow us to deduce
that the product $H^p$, ${\rm CMO}^p$, ${\rm BMO}$, and ${\rm
VMO}$-spaces, defined a priori using Auscher-Hyt\"onen
wavelets, are independent of the wavelets and the reference dyadic grids
chosen (and indeed of the reference dyadic points $\{x^k_\alpha\}$
chosen), yielding Corollary~B and Corollary~C stated
in the introduction.   

We would like to point out that the convergence in both the
$L^2(X_1\times X_2 )$-norm and $H^p(X_1\times X_2)$-(semi)norm is
crucial for proving the boundedness of Calder\'on-Zygmund
operators from $H^p( X_1\times X_2 )$ to $L^p( X_1\times X_2 )$
as described in \cite{HLLin}.

\subsection{Journ\'e-type covering lemma}

{In the product theory the Journ\'e-type covering lemmas play a
fundamental role.} The Journ\'e covering lemma was established
by Journ\'e~\cite{J} on $\mathbb{R}\times \mathbb{R}$, and by
Pipher~\cite{P} on $\mathbb{R}^{n_1}\times \dots
\times\mathbb{R}^{n_k}$. Recently, following the same ideas and
techniques as in \cite{P}, a Journ\'e-type covering lemma was
developed for $X_1\times X_2$ by the first two authors and Lin
\cite{HLLin} for certain spaces of homogeneous type.

In this section, for $i=1,2$,  $(X_i,d_i,\mu_i)$ denotes a space of homogeneous type in
the sense of Coifman and Weiss with  $\omega_i$ an upper
dimension, $A_0^{(i)}$ the quasi-triangle constant, $C_{\mu_i}$ the doubling constant, and with an underlying
dyadic grid  $\mathscr{D}_i$ whose structural constants are $c_0^i$, $C_0^i$, $c_1^i$, $C_1^i$, and $\delta_i$, as in Theorem~\ref{thm:dyadiccubes}.

Let $\Omega\subset X_1\times X_2$ be an open set of finite
measure and for $i = 1$, 2, let $m_i(\Omega)$ denote the family
of dyadic rectangles $R=Q_1\times
Q_2$ in $\Omega$ which are maximal in the $i$th
``direction'', here $Q_i\in\mathscr{D}_i$.  Also denote by $m(\Omega)$ the set of all
maximal dyadic rectangles contained in $\Omega$. Note that neither $m(\Omega )$ nor $m_1(\Omega)$ nor $m_2(\Omega)$ are disjoint collections of rectangles, this is one of the main difficulties when
dealing with the product and multi-parameter settings.

 Given a dyadic rectangle $R=Q_1\times Q_2\in m_1(\Omega)$,
let $\widehat{Q}_2=\widehat{Q}_2(Q_1)$ be the largest dyadic cube\label{def:widehatQ2Q1}
in $\mathscr{D}_2$ containing $Q_2$ such that
\begin{equation}\label{m(Omega)}
\mu\big(\big(Q_1\times \widehat{Q}_2\big)\cap \Omega\big)>{1\over 2}\mu(Q_1\times \widehat{Q}_2),
\end{equation}
where $\mu=\mu_1\times\mu_2$ is the measure on $X_1\times X_2$.
Similarly, given a dyadic rectangle $R=Q_1\times Q_2\in m_2(\Omega)$, let
$\widehat{Q}_1=\widehat{Q}_1(Q_2)$ be the largest dyadic cube in $\mathscr{D}_1$
containing $Q_1$ such that
$$\mu\big(\big(\widehat{Q}_1\times Q_2\big)\cap \Omega\big)>{1\over 2}\mu(\widehat{Q}_1\times Q_2).$$

We now state the Journ\'e-type covering lemma on $X_1\times X_2$.

\begin{lemma}[\cite{HLLin}, Lemma 2.2]\label{theorem-cover lemma}
{For $i = 1$, $2$, let $(X_i,d_i,\mu_i)$ be spaces of
    homogeneous type in the sense of Coifman and Weiss as
    described in the Introduction, with quasi-metrics~$d_i$ and
    Borel regular doubling measures~$\mu_i$, each space with an underlying dyadic grid $\mathscr{D}_i$. }
Let  $\Omega$ be an open subset in $X_1\times X_2$ with finite
measure. Let $w:[0,\infty)\to [0,\infty)$ be any {fixed}
increasing function such that $\sum_{j=0}^\infty
jw(C_02^{-j})<\infty$, where $C_0$ is any given positive
constant. Then there exists a positive constant $C$ {\rm (}dependent
on the fixed increasing function~$w$, the geometric constants of the spaces $X_i$,
and the structural constants of the underlying
dyadic grids via the ratios of the dilation constants $C_1^i/c_1^i$, for $i=1,2${\rm )} such that
\begin{eqnarray}\label{1 direction}
\sum_{R=Q_1\times Q_2\in
m_1(\Omega)}\mu(R) \,w\Big({\ell(Q_2)\over\ell(\widehat{Q}_2)}\Big)\leq
C\mu(\Omega)
\end{eqnarray}
and
\begin{eqnarray}\label{2 direction}
\sum_{R=Q_1\times Q_2\in
m_2(\Omega)}\mu(R)\, w\Big({\ell(Q_1)\over\ell(\widehat{Q}_1)}\Big)\leq
C\mu(\Omega).
\end{eqnarray}

\end{lemma}
In applications, we may take $w(t)=t^\delta$ for
any $\delta>0$ and the underlying dyadic grids may be  reference dyadic grids for the wavelets, or may belong to a regular family of dyadic grids that contains them. In these cases
the constant $C=C_{\delta}$ depends only on $\delta$ and the geometric constants of the spaces $X_i$ for $i=1$, 2.

In \cite{HLLin} the setting is on the product of two  space of homogeneous type with
a regularity condition on the metrics and a reverse doubling
condition on the measures. However the proof of the Journ\'e-type lemma uses
only  the doubling property of the measures, and  goes through
in the present setting. 
{In the same paper  the authors introduced $(p,q)$-atoms in their
setting, similar to those we define in this paper. Our $(p,q)$-atoms will have  additional enlargement  parameters $(\ell_1,\ell_2)\in \mathbb{Z}_+^2$ that were not present in \cite{HLLin}. }

\subsection{Product $(p,q)$-atoms and atomic Hardy spaces}

First we define product  $(p,q)$-atoms for all $p\in (0,1]$ and $q>1$. Second we define product atomic Hardy spaces,
 $H^{p,q}_{{\rm at}}( X_1\times X_2 )$, for all $q>1$ and for all $p$ with $p_0<p\leq 1$, where $p_0:=\max\{{\omega_i}/{(\omega_i+\eta_i)}: \; i=1,2\}$.

\begin{definition}[Product $(p,q)$-atoms]\label{def-of-p-q-atom}
Suppose that  $0<p\leq1$ and
$1<q<\infty$. For $i = 1, 2$, let $(X_i,d_i,\mu_i)$ be spaces of
    homogeneous type in the sense of Coifman and Weiss, with upper dimension $\omega_i$.
A function $a(x_1,x_2)$ defined on $ X_1\times X_2 $ is a \emph{product $(p,q)$-atom} 
 if it satisfies the following conditions.
\begin{itemize}
\item[(1)]  (Support condition on open set) There  are an open set $\Omega$  of $ X_1\times X_2 $ with finite measure and   integers $\ell_1,\ell_2\geq 0$, 
such that
${\rm supp}\,a\subset\widetilde{\Omega}_{\ell_1,\ell_2}$, where  $\widetilde{\Omega}_{\ell_1,\ell_2}$ is the $(\ell_1,\ell_2)$-enlargement of $\widetilde{\Omega}$, the $\epsilon_0$-enlargement of $\Omega$, defined respectively in~\eqref{(ell1,ell2)-enlargement} and in~\eqref{epsilon-enlargement}, with $\epsilon_0$ as defined in~\eqref{epsilon-0}.\item[(2)] (Size condition) There is a constant $C_q>0$ such that
$$\|a\|_{L^q(X_1\times X_2)}\leq C_q \, \big ( (1+\ell_1 \omega_1+\ell_2 \omega_2 ) 2^{\ell_1\omega_1+\ell_2\omega_2}\mu(\widetilde{\Omega}) \big )^{1/q-1/p} .$$

\item[(3)]  (Further decomposition into rectangle atoms with cancellation) There are underlying dyadic grids $\mathscr{D}^a_i $ on $X_i$ for $i=1,2$, such that the function $a$ can be decomposed into \emph{rectangle $(p,q)$-atoms} $a_R$ associated to a dyadic rectangle $R=Q_1\times Q_2$,  with $Q_i\in \mathscr{D}^a_i$ and satisfying the following conditions.
   \begin{itemize}
\item[(i)] (Support condition) Let $C_i=2(A_0^{(i)})^2>0$ for $i=1,2$. For all rectangles atoms $a_R$, we have that
\[ {\rm supp}\,a_R\subset C_1 2^{\ell_1}Q_1\times C_22^{\ell_2}Q_2 \subset \widetilde{\Omega}_{\ell_1,\ell_2}.\]

\item[(ii)] (Cancellation condition on each variable)
\[ \int_{X_i}a_R(x_1,x_2)\, d\mu_i(x_i)=0 \; \mbox{for a.e. \ $x_j\in X_j$ and  $(i,j)\in \{(1,2), (2,1)\}$}.\]

\item[(iii-a)] (Decomposition and size condition for $2\leq q<\infty$)  If $q\geq 2$ then
    \[a=\sum\limits_{R\in m(\Omega)}a_R \]
and there is a constant $C_q>0$ such that
\[\Big(\sum\limits_{R\in m(\Omega)}\|a_R\|_{L^q(X_1\times X_2)}^q\Big)^{1/q} \leq C_q\, \big ((1+\ell_1 \omega_1+\ell_2 \omega_2 )2^{\ell_1\omega_1+\ell_2\omega_2}\mu(\widetilde{\Omega}) \big )^{1/q-1/p}.\]

\item[(iii-b)] (Decomposition and size condition for $1< q < 2$)  If $q\in (1,2)$ then
\[a=\sum\limits_{R\in m_1(\Omega)}a_R+\sum\limits_{R\in m_2(\Omega)}a_R,\]
and for all $\delta>0$, there exists a constant $C_{q,\delta}>0$ such that  we have, for each~$(i,j)\in\{(1,2),(2,1)\},$
$$
 \hskip .5in   \bigg(\sum_{R\in m_i(\Omega)} \Big({\ell(Q_j)\over\ell(\widehat{Q}_j)}\Big)^\delta \|a_R\|_{L^q(X_1\times X_2)}^q
    \bigg)^{1/q} \leq C_{q,\delta}\, \big ((1+\ell_1 \omega_1+ \ell_2 \omega_2 )2^{\ell_1\omega_1+\ell_2\omega_2}\mu(\widetilde{\Omega}) \big )^{1/q-1/p}.
 $$
\end{itemize}
The constants $\epsilon_0$, $C_q$, $C_{q,\delta}$ depend only on the geometric constants of $X_i$ for $i=1,2$ and as indicated on $q$ and $\delta$. The families of rectangles $m(\Omega)$, $m_i(\Omega)$ for $i=1,2$ were defined in page~\pageref{m(Omega)}. We will call the  integers  $\ell_i\geq 0$, \emph{enlargement parameters} of the atom.

\end{itemize}
\end{definition}

We remark that, when $ X_1\times X_2 =\mathbb{R}^n\times
\mathbb{R}^m$, $(p,2)$-atoms with conditions (i), (ii) and
(iii-a) (with $q = 2$, and $\ell_1=\ell_2=0$) were introduced by R. Fefferman
\cite{F1}.  { When $(X_i,d_i,\mu_i)$ are spaces of homogeneous type with the quasi-metric $d_i$ satisfying the regularity condition~\eqref{smetric} and the doubling measure
$\mu_i$ satisfying a reverse doubling condition~\eqref{eqn:doubling condition}, for $i=1,2$,  the $(p,q)$-atoms with $\ell_1=\ell_2=0$, were defined in \cite[Definition 2.3]{HLLin}.}   
 In \cite[Definition~5.3]{KLPW} the product $(1,2)$-atoms as in Definition~\ref{def-of-p-q-atom} were used when  $\ell_1=\ell_2=0$.

Note that there are no wavelets and  no regularity parameters $\eta_i$ involved in the definition of the $(p,q)$-atoms. In item (3) of Definition~\ref{def-of-p-q-atom} any pair of underlying dyadic grids is acceptable,  as long as properties
(i)-(iii) are met.  However we will be interested when the underlying dyadic grids $\mathscr{D}_i^a$ belong to a regular family of dyadic grids on $X_i$ that contains all  possible reference dyadic grids $\mathscr{D}_i^W$ for all possible wavelets on $X_i$ for $i=1,2$.

The open set $\Omega$ is a placeholder and the maximal rectangles in item (3) do refer to $\Omega$. The positive constants $C_i=2(A_0^{(i)})^2$ for $i=1,2$ in item~(3)(i)   are the same for all $(p,q)$-atoms. 
However the enlargement parameters, $\ell_i$ for $i=1,2$, in item~(1)
may change from $(p,q)$-atom to $(p,q)$-atom.  We will see, in the proof of the atomic decomposition for $H^p(X_1\times X_2)$, that the $(p,q)$-atoms will be indexed by a parameter $j\in\mathbb{Z}$ and by the enlargement parameters $\ell_i\geq 0$ for $i=1,2$.  

We can now define atomic product Hardy spaces $H^{p,q}_{{\rm at}}(X_1\times X_2)$.
\begin{definition}[Atomic product Hardy spaces]\label{def:atomicHp} 
For $i = 1$, $2$, let $(X_i,d_i,\mu_i)$ be spaces of
    homogeneous type in the sense of Coifman and Weiss as
    described in the Introduction, with quasi-metrics~$d_i$ and
    Borel regular doubling measures~$\mu_i$.  Let ${\omega}_i$
    be an upper dimension for~$X_i$, and let $\eta_i$ be the
    exponent of regularity of a family of Auscher-Hyt\"onen wavelets in $X_i$.
Let
$p_0:=\max\{ \omega_i/(\omega_i+ \eta_i): i=1,2\}$, 
suppose that $p_0<p\leq 1$
and $1<q<\infty$. Then
\[H^{p,q}_{{\rm at}}(X_1\times X_2):=\{f\in (\GG)': \;  f=\sum_{j=-\infty}^\infty\lambda_ja_j, \;\;\;\sum_{j=-\infty}^{\infty}|\lambda_j|^p<\infty\},\]
where for each ${j\in\mathbb{Z}}$, the function $a_j$  is a  $(p,q)$-atom with underlying  dyadic grids $\mathscr{D}_i^{a_j}$ for $i=1,2$, belonging to a regular family of dyadic grids on $X_i$ that contains the reference dyadic grids of all possible  Auscher-Hyt\"onen wavelets on $X_i$.  Furthermore, the convergence of the series is in $(\GG)'$.
We define a (semi)norm on $H^{p,q}_{{\rm at}}(X_1\times X_2)$ as follows
\[\|f\|_{H^{p,q}_{{\rm at}}(X_1\times X_2)}:=\inf \Big\{ \Big (\sum_{j=-\infty}^{\infty}|\lambda_j|^p \Big )^{\frac{1}{p}}:
            \,f=\sum_{j=-\infty}^{\infty}\lambda_ja_j\Big\},\]
       where the infimum is taken over all possible atomic decompositions  of $f$. 
       \end{definition}
Recall that $(\GG)'$ is short for the
spaces of distributions
$\big(\GGp(\beta_{1}',\beta_{2}';\gamma_{1}',\gamma_{2}')\big)^{'}$,
respectively, where we have fixed $ \beta_i', \gamma_i' \in (0,\eta_i)$  and $\eta_i$ is the regularity exponent of the Auscher-Hyt\"onen wavelets on $X_i$  for $i = 1$, 2. In the one parameter theory, in the corresponding definition of atomic Hardy space $H^p_{{\rm at}}(X)$, it is required that  $f\in (\mathcal{C}_{\frac{1}{p}-1}(X))'$ the dual of the Campanato space,  see \cite[discussion surrounding Lemma~2.6 in p.3448]{HHL}.}

The underlying dyadic grids  can change from atom to atom. The underlying dyadic grids~$\mathscr{D}_i^a$  for $i=1,2$, for a given atom $a$, can be any dyadic grids belonging to  regular family of dyadic grids on $X_i$ that contains all the reference dyadic grids  associated to all possible wavelets on $X_i$ for $i=1,2$. In particular they may not coincide with the reference dyadic grids $\mathscr{D}_i^W$  associated to the wavelet basis on $X_i$ for $i=1,2$, used in the definition of  the product Hardy space $H^p(X_1\times X_2)$. This ensures that by definition, the product atomic Hardy spaces $H^{p,q}_{{\rm at}}(X_1\times X_2)$ are independent of the reference  dyadic grids and wavelets used in the definition of $H^p(X_1\times X_2)$. We may as well restrict the regular family of dyadic grids on each $X_i$ in the definition of atomic Hardy spaces to be the collection of reference dyadic grids for all possible wavelets on $X_i$ for $i=1,2$.

We will show in Section~\ref{MainThm-Corollaries} that $H^{p,q}_{{\rm at}}(X_1\times X_2)$ is the same space for all $q>1$, hence we can safely write $H^p_{{\rm at}}(X_1\times X_2)$. Moreover we will show that $H^p_{{\rm at}}(X_1\times X_2)=H^p(X_1\times X_2)$. In \cite{HHL} they work with $(p,2)$-atoms only, and therefore $H^p_{{\rm at}}(X_1\times X_2)$ is by definition what we denote $H^{p,2}_{{\rm at}}(X_1\times X_2)$.
Note  that if $f\in H^{p,q}_{{\rm at}}(X_1\times X_2)\cap L^q(X_1\times X_2)$ the convergence of the atomic series also holds in $L^q(X_1\times X_2)$ and that $H^{p,q}_{{\rm at}}(X_1\times X_2)\cap L^q(X_1\times X_2)$ is dense in $H^{p,q}_{{\rm at}}(X_1\times X_2)$ in the atom (semi)norm.

\subsection{Main theorem on atomic decomposition, and corollaries}\label{MainThm-Corollaries}

The main result in this section, Theorem~\ref{theorem-Hp atom
decomp}, is to show that $L^q(X_1\times X_2)\cap H^p(X_1\times
X_2 )$ has an atomic decomposition. This theorem was cited and
used in \cite[Theorem 5.4]{KLPW}, in the case $p=1$ and $q=2$,
to establish dyadic structure theorems for $H^1(X_1\times X_2)$
and $\bmo (X_1\times X_2)$.

Theorem~\ref{theorem-Hp atom decomp} was stated in the
introduction and called Main Theorem. For the convenience of
the reader we restate the theorem here, being more precise about the dyadic grids.

\begin{theorem}[Main Theorem]\label{theorem-Hp atom decomp}
    For $i = 1$, $2$, let $(X_i,d_i,\mu_i)$ be spaces of
    homogeneous type in the sense of Coifman and Weiss as
    described in the Introduction, with quasi-metrics~$d_i$ and
    Borel regular doubling measures~$\mu_i$. Let ${\omega}_i$
    be an upper dimension for~$X_i$,  let $\eta_i$ be the
    exponent of regularity of the Auscher-Hyt\"onen wavelets
    used in the construction of the Hardy
    space~$H^p(X_1\times X_2)$,  let $p_0:=\max\{ \omega_i/(\omega_i+ \eta_i): i=1,2\}$,
    and let $\mathscr{D}_i^W$ be the reference dyadic grids for the wavelets in $X_i$. Suppose that $p_0<p\leq
    1$, $1<q<\infty$, and {$f\in L^q(X_1\times X_2)$}.     Then $f\in H^p( X_1\times X_2 )$ if  and only if  $f$ has an
    atomic decomposition, that is,
    \begin{eqnarray}\label{atom decom}
        f=\sum_{j=-\infty}^\infty\lambda_ja_j.
    \end{eqnarray}
    Where, first  the  functions $a_j$ are $(p, q)$-atoms with respect to an underlying dyadic grid $\mathscr{D}^{a_j}_i$ belonging to a regular family of dyadic grids on $X_i$  that contains all possible reference grids for wavelets on $X_i$ for $i=1,2$,
    second $\sum_{j=-\infty}^{\infty}|\lambda_j|^p<\infty,$
    and third  {the series converges in 
    $L^q( X_1\times X_2 )$}.  Moreover, the series also
    converges in $H^p(X_1\times X_2)$ and
    \begin{eqnarray*}
        \|f\|_{H^p( X_1\times X_2 )}
        \sim \inf \Big\{ \Big (\sum_{j=-\infty}^{\infty}|\lambda_j |^p \Big )^{\frac{1}{p}}:
            \,f=\sum_{j=-\infty}^{\infty}\lambda_ja_j\Big\},
    \end{eqnarray*}
    where the infimum is taken over all decompositions as in \eqref{atom decom} and
    the implicit constants are independent of the $L^q( X_1\times
    X_2 )$ and $H^p( X_1\times X_2 )$-{\rm (}semi{\rm )}norms of $f$, only dependent on the geometric constants of $X_i$ for $i=1,2$.
\end{theorem}

We repeat,  the underlying dyadic grid  $\mathscr{D}^a_i$ needed
for each atom may or not coincide with the reference dyadic grid $\mathscr{D}^W_i$
associated to the underlying Auscher-Hyt\"onen wavelets  on $X_i$ for $i=1,2$, used in
the definition of $H^p(X_1\times X_2)$.

As corollaries of the Main Theorem~\ref{theorem-Hp atom
decomp} we conclude first that $H^{p,q}_{{\rm at}}(X_1\times
X_2)$ coincides with $H^p(X_1\times X_2)$ for all $q>1$, and
second that the Hardy spaces $H^p(X_1\times X_2)$ defined via
specific Auscher-Hyt\"onen wavelet bases based on specific
reference dyadic grids on $X_i$ for $i=1,2$, are indeed independent of
the choices of both wavelet bases and reference dyadic grids.

\begin{corollary}[Corollary A in the
Introduction]\label{Corollary1}
    For all  $1<q<\infty$ and  $ p_0< p\leq 1$ then
    $$H^{p,q}_{{\rm at}}(X_1\times X_2)= H^p(X_1\times X_2).$$
\end{corollary}

\begin{proof}
By Theorem~\ref{theorem-Hp atom decomp} for each $q>1$,
\[
    H^{p,q}_{{\rm at}}(X_1\times X_2)\cap L^q(X_1\times X_2)
    = L^q(X_1\times X_2)\cap H^p(X_1\times X_2 ),
\]
the closure of the right-hand-side in the $H^p$-(semi)norm is
$H^p(X_1\times X_2)$, and the closure of the left-hand-side in
the atom (semi)norm is $H^{p,q}_{{\rm at}}(X_1\times X_2)$. Both
(semi)norms are equivalent by Theorem~\ref{theorem-Hp atom
decomp}, therefore we conclude that $H^p(X_1\times
X_2)=H^{p,q}_{{\rm at}}(X_1\times X_2)$. This is precisely what
we wanted to prove.
\end{proof}

For any p with $p_0<p\leq 1$ we  now define $H^p_{{\rm at}}(X_1\times X_2)$, the
 \emph{atomic product $H^p$-space},  by 
$$H^p_{{\rm at}}(X_1\times X_2):= H^{p,q}_{{\rm at}}(X_1\times X_2),$$  
for any given $q>1$. The atomic product $H^p$-space  
is well-defined by Corollary~\ref{Corollary1}.

{\begin{corollary}[Corollary B in the
Introduction]\label{corollary-independence2}
    Let $p > p_0$, then the Hardy spaces
    $H^p(X_1\times X_2)$ as defined in~\cite{HLW} are
    independent of the particular choices of the
    Auscher-Hyt\"onen wavelets and of the reference dyadic grids used in
    their construction.
\end{corollary}

\begin{proof}
    Given $p >p_0$, define $H^p(X_1\times X_2)$ as
    in~\cite{HLW}, using a particular choice of  reference dyadic grids, $\mathscr{D}_i^W$ for $i=1,2$,
    and a particular choice of  basis of Auscher-Hyt\"onen
    wavelets defined on those grids. For $p>1$ we already know
    that $H^p(X_1\times X_2)=L^p(X_1\times X_2)$, see
    \cite{HLW}. For $p_0< p\leq 1$, choose $q > 1$. By the Main
    Theorem, the set $H^p(X_1\times X_2) \cap L^q(X_1\times
    X_2)$ coincides with the set of functions in $L^q(X_1\times
    X_2)$ that have atomic decompositions in terms of
    $(p,q)$-atoms. Each $(p,q)$-atom $a$ in a decomposition, has underlying   dyadic grids $\mathscr{D}_i^a$ for $i=1,2$,
    possibly different from $\mathscr{D}_i^W$, but belonging to regular families of dyadic grids on $X_i$ that contain all possible reference dyadic grids on $X_i$.
     The atomic decompositions are a priori unrelated
    to the Auscher-Hyt\"onen wavelets and their reference dyadic grids. Further, $H^p(X_1\times
    X_2) \cap L^q(X_1\times X_2)$ is dense in $H^p(X_1\times
    X_2)$ in the $H^p$-(semi)norm. Note that the
    closure is independent  on the choice of square function
    (which depends on the choice of wavelets and hence of
    reference dyadic grids) in the $H^p$-(semi)norm, because we
    can instead use the equivalent atom (semi)norm. Thus
    $H^p(X_1\times X_2)$ is independent of the particular
    choice of  reference dyadic grids and the particular choice of  basis
    of Auscher-Hyt\"onen wavelets defined on these grids, as
    required.
\end{proof}
}

{As a further corollary of these results and the duality
theorems, Theorem~\ref{thm:CMOp-duality-Hp} and
Theorem~\ref{thm:VMO-duality-H1}, we conclude that Carleson
measure spaces ${\rm CMO}^p(X_1\times X_2)$, the space of
bounded  mean oscillation ${\rm BMO}(X_1\times X_2)$, and the
space of vanishing mean oscillation${\rm VMO}(X_1\times X_2)$
are all independent of the chosen wavelets and reference dyadic grids.

\begin{corollary}[Corollary C in the
Introduction]
    Let $p_0<p\leq 1$, then the
    Carleson measure spaces ${\rm CMO}^p(X_1\times X_2)$, the
    space of bounded mean oscillation ${\rm BMO}(X_1\times
    X_2)$, and the space of vanishing mean oscillation ${\rm
    VMO}(X_1\times X_2)$, as defined in~\cite{HLW}, are
    independent of the particular choices of the
    Auscher-Hyt\"onen wavelets and of the reference dyadic grids used in
    their construction.
\end{corollary}

\begin{proof}
By Theorem~\ref{thm:CMOp-duality-Hp},  if $p_0<p\leq 1$ then  ${\rm CMO}^p(X_1\times X_2)$ is the
dual of $H^p(X_1\times X_2)$. By
Corollary~\ref{corollary-independence2}, the Hardy space
$H^p(X_1\times X_2)$ is independent of the particular choice of
reference dyadic grids and the  particular choice of basis of
Auscher-Hyt\"onen wavelets  defined on these grids, therefore
so will be its dual ${\rm CMO}^p(X_1\times X_2)$. Also by
Definition~\ref{def-CMOp} we know that ${\rm BMO}(X_1\times
X_2)={\rm CMO}^1(X_1\times X_2)$, and by
Theorem~\ref{thm:VMO-duality-H1} we know that $\big ({\rm
VMO}(X_1\times X_2)\big )'=H^1(X_1\times X_2)$, hence since
$H^1(X_1\times X_2)$ is independent of chosen reference dyadic grids and
wavelets so will be ${\rm BMO}(X_1\times X_2)$ and ${\rm
VMO}(X_1\times X_2)$.
\end{proof}}

\subsection{Proof of the main theorem}\label{Proof-Main-Theorem}

In the proof of the Main Theorem~\ref{theorem-Hp atom decomp}, given a function $f\in H^p(X_1\times X_2)\cap L^q(X_1\times X_2)$ we will show it can be decomposed into $(p,q)$-atoms based upon the  reference dyadic grids, $\mathscr{D}_i^W$ for $i=1,2$,  corresponding to the underlying  wavelets.
For the converse, it will suffice to verify that a given $(p,q)$-atom $a$, based on possibly different  dyadic grids $\mathscr{D}_i^a$  belonging to a regular family of dyadic gris that contains all possible reference dyadic grids for wavelets on $X_i$ for $i=1,2$, must belong to $H^p(X_1\times X_2)$ with uniform control on its $H^p$-(semi)norm.  We will have to carefully balance the geometry on both sets of dyadic grids with the size, support, and  cancellation properties of the functions $\varphi^{\gamma,\overline{C}_i}_{\ell,k_i,\alpha_i}$ for $i=1,2$ (building blocks for the wavelet $\psi^{k_i}_{\alpha_i}$ found in Lemma~\ref{decomposition of wavelet into atom})   and the rectangular $(p,q)$-atoms
$a_R$. For example, when estimating the inner product $\langle \varphi^{\gamma,\overline{C}_1}_{\ell,k_,\alpha_1}(\cdot), a_R(\cdot, x_2)\rangle_{L^2(X_1)}$  for $\mu_2$-a.e.   $x_2\in X_2$,
as we do in page~\pageref{inner-product}. To achieve this balance we will choose $\overline{C}_i= C_i2^{\ell_i}$ where
$C_i=2(A_0^{(i)})^2$ and $\ell_i$ for $i=1,2$ are the enlargement parameters appearing in the definition of the $(p,q)$-atom.

\begin{proof}[Proof of Theorem \ref{theorem-Hp atom decomp}]

($\Rightarrow$) Following the proof of
Theorem~\ref{theorem-of-fLp-lessthan-fHp-on-product-case}, for
$f\in H^p(X_1\times X_2)\cap L^q(X_1\times X_2)$, we have
by~\eqref{special-repro-identity}
and~\eqref{special-repro-identity-further}, that for some sufficiently large
$\gamma_i >0$  (in fact  for~$\gamma_i>\omega_i(1/p+1/q')$), letting
$\overline{C}_i=1$,
and denoting
$\varphi^{\gamma_i,1}_{\ell_i,k_i,\alpha_i}=\varphi^{\gamma_i}_{\ell_i,k_i,\alpha_i}$,  for $i=1,2$,
\begin{align*}
    f(x_1,x_2)&= \sum_{\ell_1,\ell_2\geq 0} 2^{-\ell_1\gamma_1-\ell_2\gamma_2}
        f_{\ell_1,\ell_2}(x_1,x_2)\\
        &=  \sum_{\ell_1,\ell_2\geq 0} 2^{-\ell_1\gamma_1-\ell_2\gamma_2}
       \sum_{j\in\mathbb{Z}}\sum_{R^{k_1,k_2}_{\alpha_1,\alpha_2}\in \mathcal{B}_j}
        \langle f,\psi_{\alpha_1}^{k_1}\psi_{\alpha_2}^{k_2} \rangle\,
        \kappa_1\,\varphi_{\ell_1,k_1,\alpha_1}^{\gamma_1}(x_1)\,\kappa_2\,
        \varphi_{\ell_2,k_2,\alpha_2}^{\gamma_2}(x_2).
\end{align*}
Here the series converges unconditionally in the $L^q(X_1\times
X_2)$-norm. As before, the constants~$\kappa_i=\sqrt{\mu_i\big
(B(y^{k_i}_{\alpha_i},\delta^{k_i})\big )}$ for $i=1,2$, the dyadic
rectangle
$R^{k_1,k_2}_{\alpha_1,\alpha_2} = Q_{\alpha_1}^{k_1}\times
Q_{\alpha_2}^{k_2}$,  with $Q_{\alpha_i}^{k_i}\in \mathscr{D}_i^W$ for $i=1,2$, and the set $\mathcal{B}_j$ was defined by
\eqref{setBj}.
We now set
\begin{eqnarray}\label{eqn:f(x1,x2)}
    f(x_1,x_2)
    &=& \sum_{\ell_1,\ell_2\geq 0} \sum_{j\in\mathbb{Z}} 2^{-\ell_1\gamma_1-\ell_2\gamma_2}
        \lambda_{j,\ell_1,\ell_2} \,
        a^{\gamma_1,\gamma_2}_{j,\ell_1,\ell_2}(x_1,x_2),
\end{eqnarray}
where the  functions $a_{j,\ell_1,\ell_2}^{\gamma_1,\gamma_2}$
will be $(p,q)$-atoms {with respect to the reference dyadic grids $\mathscr{D}_i^W$  for ~$i=1,2$ associated to the wavelets (as shown below), provided $\gamma_1$ and $\gamma_2$ are sufficiently large,} and are defined by
\begin{eqnarray*}\label{atom ak}
    a^{\gamma_1,\gamma_2}_{j,\ell_1,\ell_2}(x_1,x_2)
    := {1\over\lambda_{j,\ell_1,\ell_2} }
        \sum_{R^{k_1,k_2}_{\alpha_1,\alpha_2}\in \mathcal{B}_j}
        \langle f,\psi_{\alpha_1}^{k_1}\psi_{\alpha_2}^{k_2} \rangle\,
        \kappa_1\,\varphi_{\ell_1,k_1,\alpha_1}^{\gamma_1}(x_1)\,\kappa_2\,
        \varphi_{\ell_2,k_2,\alpha_2}^{\gamma_2}(x_2),
\end{eqnarray*}
and the coefficients $\lambda_{j,\ell_1,\ell_2}$ are defined
differently according to whether $q<2$ or not.

First, when $2\leq q<\infty$, define the coefficient
$\lambda_{j,\ell_1,\ell_2}$ as follows:
\begin{align}\label{atom lambda k q big}
    \lambda_{j,\ell_1,\ell_2}:= 2^{\ell_1\omega_1+\ell_2\omega_2}\, \| S(f_{\mathcal{B}_j})\|_{L^q(X_1\times X_2)} \,
  \big ((1+\ell_1 \omega_1+\ell_2 \omega_2 )2^{\ell_1\omega_1+\ell_2\omega_2} \mu(\widetilde{\Omega}_{j} )\big )^{{1\over p}-{1\over q}}.     
\end{align}
Second, when $1 < q < 2$, define the coefficient
$\lambda_{j,\ell_1,\ell_2}$ as follows:
\begin{align}\label{atom lambda k q small}
    \lambda_{j,\ell_1,\ell_2}:=  2^{\ell_1\omega_1 + \ell_2\omega_2} \, \|
        S(f_{\mathcal{B}_j})\|_{L^2(X_1\times X_2)}\,
\big ((1+\ell_1 \omega_1+\ell_2 \omega_2 )2^{\ell_1\omega_1+\ell_2\omega_2} \mu(\widetilde{\Omega}_{j} )\big )^{{1\over p}-{1\over 2}} 
\end{align}
Here $f_{\mathcal{B}_j}$ was defined in \eqref{function-fBj}, and hence
$S(f_{\mathcal{B}_j})=\Big
(\sum_{R^{k_1,k_2}_{\alpha_1,\alpha_2}\in
\mathcal{B}_j}\Big|\Big\langle f,\widetilde{\psi}_{\alpha_1}^{k_1}\widetilde{\psi}_{\alpha_2}^{k_2}
\Big\rangle\Big|^2 \chi_{R^{k_1,k_2}_{\alpha_1,\alpha_2}}
\Big )^{\frac{1}{2}} $, where $\widetilde{\psi}_{\alpha_i}^{k_i}={\psi}_{\alpha_i}^{k_i}/\kappa_i$ denotes the normalized wavelets for $i=1,2$. The open set $\widetilde{\Omega}_{j,\ell_1,\ell_2}$ is the $(\ell_1,\ell_2)$-enlargement of~$\widetilde{\Omega}_{j}$ defined in~\eqref{(ell1,ell2)-enlargement}, the open set $\widetilde{\Omega}_{j}$ is the $\epsilon_0$-enlargement of $\Omega_j$ defined in~\eqref{enlargement Omega j}, and the  level set $\Omega_j$ is defined in~\eqref{Omega_j}. The constant $\epsilon_0>0$ was defined in~\eqref{epsilon-0} and is purely dependent  on the geometric constants of the spaces $X_i$ for $i=1,2$.

 Notice that when $1<q<\infty$  estimate~\eqref{dual-estimate2} provides
\begin{eqnarray}\label{eqn:a-bounded-by-SBj}
\|a^{\gamma_1,\gamma_2}_{j,\ell_1,\ell_2}\|_{L^q(X_1\times X_2)}^q & \lesssim_q &
 \lambda_{j,\ell_1,\ell_2}^{-q} 2^{q(\ell_1\omega_1 +\ell_2\omega_2)}
 \| S(f_{\mathcal{B}_j})\|_{L^q(X_1\times X_2)}^q, 
 \end{eqnarray}
where the similarity depends only on the geometric constants of $X_i$ for $i=1,2$ and on $q>1$.

When $2\leq q < \infty$,  
using ~\eqref{atom lambda k q big}, the definition of   the coefficient $\lambda_{j,\ell_1,\ell_2}$, provides the following  $L^q$-estimate for the atom:
\begin{equation}\label{Lq estimate atom p>=2}
\|a^{\gamma_1,\gamma_2}_{j,\ell_1,\ell_2}\|_{L^q(X_1\times X_2)}^q  \lesssim_q \,
\big ((1+\ell_1 \omega_1+\ell_2 \omega_2 )2^{\ell_1\omega_1+\ell_2\omega_2} \mu(\widetilde{\Omega}_{j} )\big )^{{1}-{q\over p}}.  
 \end{equation}
 In particular when $q=2$ we obtain the following $L^2$-estimate for the atom:
\begin{equation}\label{L2 estimate atom p<2}
\|a^{\gamma_1,\gamma_2}_{j,\ell_1,\ell_2}\|_{L^2(X_1\times X_2)}^2  \lesssim  \,
\big ((1+\ell_1 \omega_1+\ell_2 \omega_2 )2^{\ell_1\omega_1+\ell_2\omega_2} \mu(\widetilde{\Omega}_{j} )\big )^{{1}-{2\over p}}.   
 \end{equation}

We now verify that the functions $a_{j,\ell_1,\ell_2}^{\gamma_1,\gamma_2}$ are $(p,q)$-atoms with respect to the reference dyadic grids $\mathscr{D}_i^W$ for $i=1,2$ associated to the underlying wavelets, with the open set $\Omega_j$ playing the role of $\Omega$ in Definition~\ref{def-of-p-q-atom},  and with enlargement parameters $\ell_1,\ell_2 \geq 0$.

First we check  that $a^{\gamma_1,\gamma_2}_{j,\ell_1,\ell_2}$ satisfies condition (1) of
Definition~\ref{def-of-p-q-atom}. Recall that  $\varphi^{\gamma_i}_{\ell_i,k_i,\alpha_i}(x_i)$
is supported on the ball  $B(y_{\alpha_i}^{k_i}, {2(A_0^{(i)})^2} \, 2^{\ell_i}\delta^{k_i}) \subset X_i$
 for each $i=1,2$.
Hence, if $R\in \mathcal{B}_j$, then the support of
$\varphi_{\ell_1,k_1,\alpha_1}^{\gamma_1}(x_1)\,\varphi_{\ell_2,k_2,\alpha_2}^{\gamma_2}(x_2)$ is
contained in the open set $\widetilde{\Omega}_{j,\ell_1,\ell_2}=(\widetilde{\Omega}_j)_{\ell_1,\ell_2}$, as explained in page~\pageref{RinBj}.  Note that since $f\in L^q(\widetilde{X})$, where $\widetilde{X}=X_1\times X_2$, for $1<q<\infty$, then ${\Omega}_{j}$ and $\widetilde{\Omega}_{j,\ell_1,\ell_2}$ have finite measure. More precisely, by estimates~\eqref{eqn:measure-enlargements} and~\eqref{weak-L^2-Ms}  and by Tchebichev's inequality~\eqref{Tchebichev-square},
\begin{align*}
\mu(\widetilde{\Omega}_{j,\ell_1,\ell_2}) &\lesssim (1+\ell_1 \omega_1+\ell_2 \omega_2 )2^{\ell_1\omega_1+\ell_2\omega_2} \mu (\Omega_j) \leq (1+\ell_1 \omega_1+\ell_2 \omega_2 )2^{\ell_1\omega_1+\ell_2\omega_2} 2^{-jq} \| S(f)\|_{L^q(\widetilde{X})}^q\\
& \lesssim_q (1+\ell_1 \omega_1+\ell_2 \omega_2 )2^{\ell_1\omega_1+\ell_2\omega_2} 2^{-jq} \|f\|_{L^q(\widetilde{X})}^q <\infty.
\end{align*}
Thus condition (1) of
Definition~\ref{def-of-p-q-atom} holds.

Second we verify that $a^{\gamma_1,\gamma_2}_{j,\ell_1,\ell_2}$ satisfies condition (2) of
Definition~\ref{def-of-p-q-atom}.
For $2\leq q<\infty$ this is estimate~\eqref{Lq estimate atom p>=2}.
For $1<q<2$, since $a^{\gamma_1,\gamma_2}_{j,\ell_1,\ell_2}$ is supported in $\widetilde{\Omega}_{j,\ell_1,\ell_2}$, applying H\"older's inequality  with exponent $s=2/q>1$, and using \eqref{eqn:measure-enlargements} and the $L^2$-estimate~\eqref{L2 estimate atom p<2}, yields
\begin{eqnarray*}
\|a_{j,\ell_1,\ell_2}^{\gamma_1,\gamma_2}\|_{L^q(X_1\times X_2)} &\leq &
\|a^{\gamma_1,\gamma_2}_{j,\ell_1,\ell_2}\|_{L^2(X_1\times X_2)}\,\mu(\widetilde{\Omega}_{j,\ell_1,\ell_2})^{{1\over q}-{1\over 2}} \\
& &\hskip -1in \lesssim \;  \big ((1+\ell_1 \omega_1+\ell_2 \omega_2 )2^{\ell_1\omega_1+\ell_2\omega_2} \mu(\widetilde{\Omega}_{j} )\big )^{{1\over 2}-{1\over p}}\, 
\big ((1+\ell_1 \omega_1+\ell_2 \omega_2 )2^{\ell_1\omega_1+\ell_2\omega_2} \mu(\widetilde{\Omega}_{j} )\big )^{{1\over q}-{1\over 2}} \\
& & \hskip -1in \lesssim \; \big ( (1+\ell_1 \omega_1+\ell_2 \omega_2 )2^{\ell_1\omega_1+\ell_2\omega_2} \mu(\widetilde{\Omega}_{j} )\big )^{{1\over q}-{1\over p}}. 
\end{eqnarray*}
As a consequence, we get that $a^{\gamma_1,\gamma_2}_{j,\ell_1,\ell_2}$ satisfies condition (2) of
Definition~\ref{def-of-p-q-atom}.

Third, it remains to check that $a^{\gamma_1,\gamma_2}_{j,\ell_1,\ell_2}$ satisfies  condition (3) of
Definition~\ref{def-of-p-q-atom}.
To see this, we can further decompose $a^{\gamma_1,\gamma_2}_{j,\ell_1,\ell_2}$
into  rectangular atoms $a^{\gamma_1,\gamma_2}_{j,\ell_1,\ell_2, \overline{R}}$ defined by
$$
a^{\gamma_1,\gamma_2}_{j,\ell_1,\ell_2, \overline{R}}(x_1,x_2) :={1\over
\lambda_{j,\ell_1,\ell_2}} \sum_{R=R^{k_1,k_2}_{\alpha_1,\alpha_2}\in \mathcal{B}_j, \, \tau(R)=\overline R }
        \langle f,\psi_{\alpha_1}^{k_1}\psi_{\alpha_2}^{k_2} \rangle\,
        \kappa_1\,\varphi_{\ell_1,k_1,\alpha_1}^{\gamma_1}(x_1)\,\kappa_2\,
        \varphi_{\ell_2,k_2,\alpha_2}^{\gamma_2}(x_2),
$$
where $\overline{R} = \overline{Q}_1\times \overline{Q}_2$ with $\overline{Q}_i\in \mathscr{D}_i^W$, a dyadic cube associated to the wavelets on $X_i$ for $i=1,2$.  Here $\tau: \mathcal{B}_j\to m(\Omega_j)$ denotes a function  that assigns to each $R\in\mathcal{B}_j$ a rectangle  $\tau(R)=\overline{R}\in m(\Omega_j)$,  so that $R\subset \overline{R}$. This will be important when verifying condition~(3)(iii-a)  in Definition~\ref{def-of-p-q-atom}. Likewise when verifying condition~(3)(iii-b) in Definition~\ref{def-of-p-q-atom} we will assign each $R\in\mathcal{B}_j$ to only one $\overline{R}\in m_1(\Omega_j)\cup m_2(\Omega_j)$ with $R\subset \overline{R}$.

We can verify that ${\rm supp}\,a^{\gamma_1,\gamma_2}_{j,\ell_1,\ell_2, \overline{R}} \subset 2(A_0^{(1)})^2\, 2^{\ell_1}\overline{Q}_1\times 2(A_0^{(2)})^2\,2^{\ell_2} \overline{Q}_2$,  by definition of the rectangle atoms and the support conditions of the functions $\varphi_{\ell_i,k_i,\alpha_i}^{\gamma_i}$, for $i=1,2$.  We deduce that
$ \int_{X_i} a^{\gamma_1,\gamma_2}_{j,\ell_1,\ell_2, \overline{R}}(x_1,x_2)\,d\mu_i(x_i)=0 $ for a.e. $x_j\in X_j$,
by the cancellation conditions of the functions  $\varphi_{\ell_i,k_i,\alpha_i}^{\gamma_i}$
 for $(i,j)\in \{(1,2), (2,1)\}$, and the facts that the integrand~$a^{\gamma_1,\gamma_2}_{j,\ell_1,\ell_2, \overline{R}}\in L^q(X_1\times X_2)$ for $q>1$ and has compact support.
These show that the support and cancellation  conditions~(3)(i) and~(3)(ii) of Definition~\ref{def-of-p-q-atom} hold, with support constants $C_i=2(A_0^{(i)})^2$, as required,   and enlargement constants ${\ell_i}\geq 0$, for~$i=1,2$.

We now show that
$a^{\gamma_1,\gamma_2}_{j,\ell_1,\ell_2}$ satisfies the decomposition and size conditions~(3)(iii-a), when $2\leq q<\infty$, and~(3)(iii-b),  when $1<q<2$, of Definition~\ref{def-of-p-q-atom}.

For $2\leq q<\infty$,  first observe that $a^{\gamma_1,\gamma_2}_{j,\ell_1,\ell_2}=\sum_{\overline{R}\in m({\Omega}_{j})} a^{\gamma_1,\gamma_2}_{j,\ell_1,\ell_2, \overline{R}}
$, this is true because each $R\in \mathcal{B}_j$ is assigned to exactly one $\overline{R}\in m(\Omega)$, namely to $\overline{R}=\tau(R)$.
 Second, we have by definition of the rectangular atom and the triangle inequality
\begin{align*}
&\|a^{\gamma_1,\gamma_2}_{j,\ell_1,\ell_2,\overline{R}}\|_{L^q(X_1\times X_2)}
= \sup_{g: \ \|g\|_{  L^{q'}(X_1\times X_2)}\leq1} \big|\big\langle  a^{\gamma_1,\gamma_2}_{j,\ell_1,\ell_2,\overline{R}} , g \big\rangle\big|\\
 & \hskip .2in \leq
\sup_{g: \ \|g\|_{  L^{q'}(X_1\times X_2)}\leq1}\lambda_{j,\ell_1,\ell_2}^{-1} \sum_{R=R^{k_1,k_2}_{\alpha_1,\alpha_2}\in \mathcal{B}_j, \, \tau(R)=\overline R }
       \big| \langle f,\widetilde {\psi}_{\alpha_1}^{k_1}\widetilde{\psi}_{\alpha_2}^{k_2} \rangle\big|\,
        \kappa_1^2\,\kappa_2^2\,  \big|\langle \varphi_{\ell_1,k_1,\alpha_1}^{\gamma_1}
        \varphi_{\ell_2,k_2,\alpha_2}^{\gamma_2},\ g\rangle\big|.
\end{align*}
Therefore, first raising to the $q$ power, and second using the Cauchy-Schwarz inequality
on the sum together with
Lemma~\ref{lemma LittlewoodPaley} as we did when estimating~\eqref{dual-estimate},  we conclude  that
\begin{align*}
&\|a^{\gamma_1,\gamma_2}_{j,\ell_1,\ell_2,\overline{R}}\|_{L^q(X_1\times X_2)}^q\\
& \lesssim_q  \sup_{g: \ \|g\|_{  L^{q'}(X_1\times X_2)}\leq1} \lambda_{j,\ell_1,\ell_2}^{-q}    \bigg( \sum_{R=R^{k_1,k_2}_{\alpha_1,\alpha_2}\in \mathcal{B}_j,\,  \tau(R)=\overline R }
       \big| \langle f,\widetilde {\psi}_{\alpha_1}^{k_1}\widetilde{\psi}_{\alpha_2}^{k_2} \rangle\big|\,
        \kappa_1^2\,\kappa_2^2\,  \big|\langle \varphi_{\ell_1,k_1,\alpha_1}^{\gamma_1}
        \varphi_{\ell_2,k_2,\alpha_2}^{\gamma_2},\ g\rangle\big|\bigg)^q\\
 &  \lesssim_q   2^{(\ell_1\omega_1+\ell_2\omega_2)q} \lambda_{j,\ell_1,\ell_2}^{-q}  \int_{X_1\times X_2}  \bigg( \sum_{R=R^{k_1,k_2}_{\alpha_1,\alpha_2}\in \mathcal{B}_j,\,  \tau(R)=\overline R }  \big| \langle f,\widetilde {\psi}_{\alpha_1}^{k_1}\widetilde{\psi}_{\alpha_2}^{k_2} \rangle\big|^2 \chi_{R^{k_1,k_2}_{\alpha_1,\alpha_2}}(x_1,x_2) \bigg)^{q\over2} d\mu_1(x_1)d\mu_2(x_2).
\end{align*}
 We now add this estimate over all $\overline{R}\in m(\Omega_j)$, note that the power $q/2\geq 1$ can be pulled out of the sum (namely $\sum_k |a_k|^{q/2}\leq (\sum_k |a_k|)^{q/2}$),   and remember that each $R\in \mathcal{B}_j$ is assigned to exactly one $\overline{R}\in m(\Omega_j )$ that contains it, and get 
 \begin{align}\label{eqn:Lq-estimate-aR}
 \sum_{\overline{R}\in m({\Omega}_{j})}\big\|a^{\gamma_1,\gamma_2}_{j,\ell_1,\ell_2, \overline{R}}\big\|^q_{L^q(X_1\times X_2 )}
& \lesssim_q   2^{(\ell_1\omega_1+\ell_2\omega_2)q} \, \lambda_{j,\ell_1,\ell_2}^{-q}  \|S(f_{\mathcal{B}_j})\|_{L^q(X_1\times X_2)}^q
 \nonumber\\
&  \lesssim_q   \big ( (1+\ell_1\omega_1+\ell_2\omega_2) \, 2^{\ell_1\omega_1+\ell_2\omega_2} \mu(\widetilde{\Omega}_{j} )\big )^{{1}-{q\over p}},
\end{align}
where in the last inequality we used the definition~\eqref{atom lambda k q big} of $\lambda_{j,\ell_1,\ell_2}$. This proves  condition (3)(iii-a)  of Definition~\ref{def-of-p-q-atom}.

For $1<q<2$, applying H\"older's inequality and the Journ\'e-type covering lemma,
 we will show that condition~(3)(iii-b) of Definition~\ref{def-of-p-q-atom} holds.
First we observe that
 in this case the decomposition $a^{\gamma_1,\gamma_2}_{j,\ell_1,\ell_2}=\sum_{\overline{R}\in m_1({\Omega}_{j})} a^{\gamma_1,\gamma_2}_{j,\ell_1,\ell_2, \overline{R}} +
 \sum_{\overline{R}\in m_2'({\Omega}_{j})} a^{\gamma_1,\gamma_2}_{j,\ell_1,\ell_2, \overline{R}} \,$ holds.
 Where  the second sum is over $m_2'({\Omega}_{j}):=m_2({\Omega}_{j})\setminus m_1({\Omega}_{j})$ to avoid duplicates. The decomposition is true because this time we assign each $R\in \mathcal{B}_j$ to exactly one $\overline{R}\in m_1(\Omega_j)\cup m_2(\Omega_j)$, namely  $\overline{R}=\tau(R)$ where the function $\tau:\mathcal{B}_j\to m_1(\Omega_j)\cup m_2(\Omega_j)$.
Second,  let us show that given $\delta>0$ there is a constant $C_{q,\delta}>0$ such that
$$\sum_{\overline{R}\in m_1({\Omega}_{j})} \Big({\ell(Q_2)\over\ell(\widehat{Q}_2)}\Big)^\delta \|a^{\gamma_1,\gamma_2}_{j,\ell_1,\ell_2,\overline{R}}\|_{L^q(X_1\times X_2)}^q
\leq C_{q,\delta}\, \big ( (1+\ell_1 \omega_1+\ell_2 \omega_2 )2^{\ell_1\omega_1+\ell_2\omega_2} \mu(\widetilde{\Omega}_{j} )\big )^{{1}-{q\over p}}. $$  
A similar argument will take care of the  sum over $\overline{R}\in m_2(\Omega_{j})$, and hence over $\overline{R}\in m'_2(\Omega_{j})$. First, using  H\"older's inequality with exponent $s=2/q>1$, the support property of the rectangular atoms, and the doubling condition of the measures (as in~\eqref{doubling-dilate-cubes}),
we get that
$$\|a^{\gamma_1,\gamma_2}_{j,\ell_1,\ell_2,\overline{R}}\|_{L^q(X_1\times X_2)}^q\lesssim
\|a_{j,\ell_1,\ell_2,\overline{R}}\|_{L^2(X_1\times X_2)}^{q}\big (
2^{\ell_1\omega_1+\ell_2\omega_2} \mu (\overline{R})\big )^{\frac{2-q}{2}}.$$
Second, substituting this estimate and using H\"older's inequality in the sum with exponents $s=2/q$ and $s'=2/(2-q)$, 
we get
\begin{align*}
\sum_{\overline{R}\in m_1(\Omega_{j})} \Big({\ell(Q_2)\over\ell(\widehat{Q}_2)}\Big)^\delta \|a^{\gamma_1,\gamma_2}_{j,\ell_1,\ell_2,\overline{R}}\|_{L^q(\widetilde{X})}^q
&\lesssim \sum_{\overline{R}\in m_1(\Omega_{j})} \Big({\ell(Q_2)\over\ell(\widehat{Q}_2)}\Big)^\delta
\|a^{\gamma_1,\gamma_2}_{j,\ell_1,\ell_2,\overline{R}}\|_{L^2(X_1\times X_2)}^{q}\big (2^{\ell_1\omega_1+\ell_2\omega_2} \mu (\overline{R})\big )^{\frac{2-q}{2}}\\
&\hskip -4cm
 \lesssim \Big ( 2^{\ell_1\omega_1+\ell_2\omega_2} \sum_{\overline{R}\in m_1(\Omega_{j})} \Big({\ell(Q_2)\over\ell(\widehat{Q}_2)}\Big)^{\frac{2\delta}{2-q}} \mu (\overline{R}) \Big )^{\frac{2-q}{2}}
\Big (\sum_{\overline{R}\in m_1(\Omega_{j})} \|a^{\gamma_1,\gamma_2}_{j,\ell_1,\ell_2,\overline{R}}\|^2_{L^2(X_1\times X_2)} \Big )^{\frac{q}{2}}\\
& \hskip -4cm \lesssim_{q,\delta}\, \big (2^{\ell_1\omega_1+\ell_2\omega_2} \mu(\Omega_j)\big )^{1-\frac{q}{2}}
\big ( (1+\ell_1 \omega_1+\ell_2 \omega_2 )2^{\ell_1\omega_1+\ell_2\omega_2} \mu(\widetilde{\Omega}_{j} )\big )^{{q\over 2}-{q\over p}} \\  
& \hskip -4cm \lesssim_{q,\delta}\,  \big ( (1+\ell_1 \omega_1+\ell_2 \omega_2 )2^{\ell_1\omega_1+\ell_2\omega_2} \mu(\widetilde{\Omega}_{j} )\big )^{{1}-{q\over p}}.
\end{align*}
We used   the Journ\'e-type covering lemma with $\delta'=\frac{2\delta}{2-q}>0$, and estimate \eqref{eqn:Lq-estimate-aR} (for $q=2$),  in the third inequality. In the last inequality we used the fact that $\mu(\widetilde{\Omega}_j )\sim \mu(\Omega_j )$.
Altogether we obtain  the desired atomic decomposition for $f$.

Finally by  computations similar to those in the proof of Theorem~\ref{theorem-of-fLp-lessthan-fHp-on-product-case} we conclude that when $f\in H^p(X_1\times X_2)\cap L^q(X_1\times X_2)$ then
$\inf \sum_{j\in\mathbb{Z}} | \lambda_{j}|^p \leq C\,\|f\|_{H^p(X_1,X_2)}^p$, where the infimum is taken over all decompositions of the form $f=\sum_{j\in\mathbb{Z}}\lambda_{j} a_{j} $, the functions $a_j$ are $(p,q)$-atoms,
and $\sum_{j\in\mathbb{Z}} |\lambda_j|^p<\infty$.  More precisely, it suffices to show that for the decomposition we just proved, namely $f(x_1,x_2)= \sum_{j\in\mathbb{Z}; \, \ell_1,\ell_2\geq 0} 2^{-\ell_1\gamma_1-\ell_2\gamma_2}  \lambda_{j,\ell_1,\ell_2} a^{\gamma_1,\gamma_2}_{j,\ell_1,\ell_2}(x_1,x_2)$,
  the following inequality holds:
\begin{equation}\label{norm-estimate-lambda-less-than-Hp}
 \sum_{j\in\mathbb{Z}; \, \ell_1,\ell_2\geq 0} |2^{-\ell_1\gamma_1-\ell_2\gamma_2}
 \lambda_{j,\ell_1,\ell_2}|^p\lesssim_q \, \|S(f)\|_{L^p(X_1\times X_2)}^p.
 \end{equation}

{When $1<q<2$, according to definition \eqref{atom lambda k q small} we get, using   that the square function is bounded on $L^2(X_1\times X_2)$, that 
\begin{align*}
\sum_{j\in\mathbb{Z}; \, \ell_1,\ell_2\geq 0} |2^{-\ell_1\gamma_1-\ell_2\gamma_2}
\lambda_{j,\ell_1,\ell_2}|^p & \\
& \hskip -2in= \sum_{j\in\mathbb{Z}\, \ell_1,\ell_2\geq 0} \|S(f_{\mathcal{B}_j})\|_{L^2(X_1\times X_2)}^p 2^{\ell_1p(\omega_1-\gamma_1)}2^{\ell_2p(\omega_2-\gamma_2)}
 \big ((1+\ell_1 \omega_1+\ell_2 \omega_2 )2^{\ell_1\omega_1+\ell_2\omega_2} \mu(\widetilde{\Omega}_{j} )\big )^{{1}-{p\over 2}}\\ 
 &\hskip -2in \lesssim  \;\sum_{j\in\mathbb{Z}} \|f_{\mathcal{B}_j}\|_{L^2(X_1\times X_2)}^p \mu (\widetilde{\Omega}_{j})^{1-\frac{p}{2}}\sum_{\ell_1,\ell_2\geq 0} 2^{\ell_1p (\omega_1(\frac{1}{p}+\frac{1}{2})-\gamma_1 )}2^{\ell_2p (\omega_2 (\frac{1}{p}+\frac{1}{2} )-\gamma_2 )}  (1+\ell_1 \omega_1+\ell_2 \omega_2 )^{1-{p\over2}}.
\end{align*}
 The  series over $\ell_1, \ell_2$ converges if we choose $\gamma_i>\omega_i\big (\frac{1}{p}+\frac{1}{2}\big )$ for $i=1,2$.
 Therefore,
 \[\sum_{j\in\mathbb{Z}; \,\ell_1,\ell_2\geq 0} |2^{-\ell_1\gamma_1-\ell_2\gamma_2}
 \lambda_{j,\ell_1,\ell_2}|^p \lesssim  \sum_{j\in\mathbb{Z}} 2^{pj}\mu(\widetilde{\Omega}_j\setminus \Omega_{j+1})^{\frac{p}{2}} \mu (\widetilde{\Omega}_{j})^{1-\frac{p}{2}}\lesssim  \sum_{j\in\mathbb{Z}} 2^{pj}\mu(\widetilde{\Omega}_j). \] 
In the first inequality we have used the following estimate for the $L^2$-norm of  $f_{\mathcal{B}_j}$:
\begin{align*}
\| f_{\mathcal{B}_j}\|_{L^2(X_1\times X_2)}^2 &=\sum_{R^{k_1,k_2}_{\alpha_1,\alpha_2}\in \mathcal{B}_j}
        \big|\langle f,\psi_{\alpha_1}^{k_1}\psi_{\alpha_2}^{k_2} \rangle\big|^2\\
        &\leq 2\sum_{R^{k_1,k_2}_{\alpha_1,\alpha_2}\in \mathcal{B}_j}
        \mu_1(Q_{\alpha_1}^{k_1})^{-1}\mu_2(Q_{\alpha_2}^{k_2})^{-1}
        \big|\langle f,\psi_{\alpha_1}^{k_1}\psi_{\alpha_2}^{k_2}
        \rangle\big|^2\mu\big (R_{\alpha_1,\alpha_2}^{k_1,k_2}\cap (\widetilde{\Omega}_j\backslash \Omega_{j+1}) \big )\\
      &= 2\| S(f_{\mathcal{B}_j})\|_{L^2(\widetilde{\Omega}_j\backslash \Omega_{j+1})}^2
      \,\leq \, 2\| S(f)\|_{L^2(\widetilde{\Omega}_j\backslash \Omega_{j+1})}^2
      \, \lesssim 2^{2j}\mu(\widetilde{\Omega}_j\setminus \Omega_{j+1}).
\end{align*}
{In the above calculation  we used Plancherel in the first line, and we used the fact that when $R\in \mathcal{B}_j$ then  $2\mu\big (R\cap
(\widetilde{\Omega}_j\backslash \Omega_{j+1}) \big )>\mu (R)$ in the second line (as shown in page~\pageref{R-minus-Omega}). In the third line, the last inequality holds because if $(x_1,x_2)\notin \Omega_{j+1}$ then $|S(f)(x_1,x_2)| \leq 2^{j+1}$.}

Finally, recalling that $\mu(\widetilde{\Omega}_j) \lesssim \mu (\Omega_j )$,  and using \eqref{Sf-norm-Lp-coronas} we conclude that
$$\sum_{j\in\mathbb{Z}; \,\ell_1,\ell_2\geq 0} |2^{-\ell_1\gamma_1-\ell_2\gamma_2}
 \lambda_{j,\ell_1,\ell_2}|^p\lesssim\sum_{j\in\mathbb{Z}} 2^{pj}\mu(\Omega_{j}) \lesssim \|S(f)\|_{L^p(X_1\times X_2)}^p.$$
 Therefore inequality  \eqref{norm-estimate-lambda-less-than-Hp} holds when $1<q<2$ whenever the parameters $\gamma_i$ satisfy the constraint $\gamma_i> \omega_i \big(\frac{1}{p}+\frac{1}{2}\big )$
for $i=1,2$.}  {Notice that in this range $q'>2$ and  $ \big(\frac{1}{p}+\frac{1}{2}\big ) > \big(\frac{1}{p}+\frac{1}{q'}\big )$, therefore the constraint needed in the proof of Theorem~\ref{theorem-of-fLp-lessthan-fHp-on-product-case} in page~\pageref{constraintThm4.2} is satisfied.}

When $q \geq 2$, according to definition \eqref{atom lambda k q big}, by a similar argument  to that  in the proof of Theorem~\ref{theorem-of-fLp-lessthan-fHp-on-product-case}, specifically using~\eqref{eqn:SBj-less-muOmegaj}  
 and  provided that $\gamma_i>\omega_i\big (\frac{1}{p}+\frac{1}{q'}\big )$ for $i=1,2$, we get that
\begin{align*}
\sum_{j\in\mathbb{Z}; \ell_1,\ell_2\geq 0} |2^{-\ell_1\gamma_1-\ell_2\gamma_2}
\lambda_{j,\ell_1,\ell_2}|^p & \\
& \hskip -1.7in = \sum_{j\in\mathbb{Z};\ell_1,\ell_2\geq 0} \|S(f_{\mathcal{B}_j})\|_{L^q(X_1\times X_2)}^p  2^{\ell_1p(\omega_1-\gamma_1)}2^{\ell_2p(\omega_2-\gamma_2)}  \big ( (\ell_1 \omega_1+\ell_2 \omega_2 )2^{\ell_1\omega_1+\ell_2\omega_2} \mu(\widetilde{\Omega}_{j} )\big )^{{1}-{p\over q}}\\
& \hskip -1.7in\lesssim_q \sum_{j\in\mathbb{Z}} 2^{pj} \mu(\widetilde{\Omega}_j)^{\frac{p}{q}}\, \mu (\widetilde{\Omega}_{j})^{1-\frac{p}{q}}\sum_{\ell_1,\ell_2\geq 0} 2^{\ell_1p(\omega_1(\frac{1}{p}+\frac{1}{q'} )-\gamma_1)}2^{\ell_2p(\omega_2(\frac{1}{p}+\frac{1}{q'})-\gamma_2)}  (1+\ell_1 \omega_1+\ell_2 \omega_2 )^{1-{p\over q}}\\
&\hskip -1.7in \lesssim_q  \sum_{j\in\mathbb{Z}} 2^{pj}\mu( \Omega_{j})
\; \lesssim_q \; \|S(f)\|_{L^p(X_1\times X_2)}^p.
 \end{align*}

We conclude that \eqref{norm-estimate-lambda-less-than-Hp} holds when $q\geq 2$ whenever the parameters $\gamma_i$ satisfy the constraint $\gamma_i> \omega_i\big (\frac{1}{p}+\frac{1}{q'}\big )$
for $i=1,2$. {Notice that this is the same constraint needed in the proof of Theorem~\ref{theorem-of-fLp-lessthan-fHp-on-product-case} in page~\pageref{constraintThm4.2}. All the constants appearing in  the inequalities/similarities depend on the geometric constants of the spaces $X_i$ for $i=1,2$, and possibly on the parameters $q>1$ or $\delta>0$ as indicated. }\\

\noindent ($\Leftarrow$)  Given an atomic decomposition $f=\sum_{j\in\mathbb{Z}} \lambda_ja_j$ for $f\in L^q(X_1\times X_2)\cap H^{p,q}_{{\rm at}}(X_1\times X_2)$, with $\sum_{j\in\mathbb{Z}} |\lambda_j|^p<\infty$.
By definition each  product $(p,q)$-atom $a_j$ has underlying dyadic grids~$\mathscr{D}^{a_j}_i$ on $X_i$ for $i=1,2$ belonging to regular families of dyadic grids on $X_i$ that contain all the reference dyadic grids for wavelets on $X_i$. 
The series is assumed to converge in $L^q(X_1\times X_2)$, hence it suffices to verify that
 there is a constant $C>0$  such that for all such $(p,q)$-atoms $a$
\begin{eqnarray}\label{S a uniformly bounded}
\|S(a)\|_{L^p( X_1\times X_2 )}\leq C.
\end{eqnarray}
The constant $C>0$ depends only on the geometric constants of the spaces $X_i$ for $i=1,2$
  and on $p$ and $q$, but not on  the enlargement parameters~$\ell_1,\ell_2\geq 0$   of Definition~\ref{def-of-p-q-atom} of  the $(p,q)$-atom. The constant will depend on the structural constants of the atom's underlying dyadic grids, $\mathcal{D}_i^a$ for $i=1,2$, via the outer balls dilation constants $C_1^i$ and the ratio of the outer and inner balls dilation constants $C_1^i/c_1^i$. These quantities will appear when using the doubling property for dilates of cubes as in~\eqref{doubling-dilate-cubes1}.
 Both  quantities are uniformly bounded by a constant depending only on the quasi-triangle constants of the quasi-metric $d_i$, since the grids $\mathcal{D}_i^a$ are assumed to belong to a regular family of dyadic grids on $X_i$ for $i=1,2$,
 see Definition~\ref{def:regular-dyadic-grids} and \eqref{doubling-dilate-cubes}.

Once we prove estimate~\eqref{S a uniformly bounded} for the atoms,  if $f\in L^q(X_1\times X_2)$ has an atomic decomposition $f=\sum_i \lambda_i a_i,$ where the series converges in both $L^q$-norm and $H^p$-(semi)norm, then by subadditivity of the square function, and since $p\leq 1$,
 together with  \eqref{S a uniformly bounded}, we conclude that
\begin{eqnarray*}
\|f\|^p_{H^p(X_1\times X_2)}=\|{S}(f)\|_{L^p(X_1\times X_2)}^p\leq \sum_{i\in\mathbb{Z}} |\lambda_i|^p \|{ S}(a_i)\|_{L^p(X_1\times X_2)}^p\leq C^p\sum_{i\in\mathbb{Z}} |\lambda_i|^p<\infty,
\end{eqnarray*}
which immediately  proves the norm estimate   $\|f\|_{H^p(X_1\times X_2)}\lesssim \inf\{ \big (\sum_{i\in \mathbb{Z}} |\lambda_i|^p\big )^{1/p}\}$.

To this end, fix a $(p,q)$-atom $a$ with ${\rm supp}\,a\subset \Omega_*$, where $\Omega_*$ is an appropriate enlargement of the open set $\Omega$ in Definition~\ref{def-of-p-q-atom}, more precisely $\Omega_*=\widetilde{\Omega}^{\epsilon_0}_{\ell_1,\ell_2}$ for some enlargement parameters $\ell_1, \ell_2>0$.
Recall that  $\mu(\Omega)\sim \mu(\widetilde{\Omega}^{\epsilon_0}) \leq \mu (\widetilde{\Omega}^{\epsilon_0}_{\ell_1,\ell_2})\lesssim (1+\ell_1\omega_1+\ell_2\omega_2)2^{\ell_1\omega_1+\ell_2\omega_2}\mu(\widetilde{\Omega}^{\epsilon_0})$, where the last inequality holds by~\eqref{eqn:measure-enlargements}.  Assume the $(p,q)$-atom has
a {decomposition $a=\sum_{R\in m(\Omega)}a_R$ when $q\geq2$, and a decomposition
$a=\sum_{R\in m_1(\Omega)}a_R + \sum_{R\in m_2'(\Omega)}a_R$ when $1<q<2$}.
We will work in detail the case when $q\geq 2$. A similar argument will take care of the  case $1<q<2$,  we only need to  start with dyadic rectangles $R$ in $m_1(\Omega)$ or in $m_2(\Omega)$.

Let  $\widetilde{\Omega}$ 
be the {$\epsilon$-}enlargement of $\Omega$ and let $\widetilde{\widetilde{\Omega}}$ be the {$\epsilon$-}enlargement of $\widetilde{\Omega}$, as defined in \eqref{epsilon-enlargement} {for $\epsilon= 1/2$}, that is,
\begin{eqnarray*}
{\widetilde{\Omega}} & = & \{(x_1,x_2)\in  X_1\times X_2 :\ {M}_s(\chi_{{\Omega}})(x_1,x_2)>1/2 \}, \\\widetilde{\widetilde{\Omega}} &= & \{(x_1,x_2)\in  X_1\times X_2 :\ {M}_s(\chi_{\widetilde{\Omega}})(x_1,x_2)>1/2 \}.
\end{eqnarray*}
It will be useful to keep in mind that $\Omega\subset \widetilde{\Omega}\subset \widetilde{\widetilde{\Omega}}$ and that $\mu (\Omega) \sim \mu(\widetilde{\Omega}) \sim \mu (\widetilde{\widetilde{\Omega}})$ by \eqref{weak-L^2-Ms}.

Moreover,  recall that $m_i(\Omega)$ denotes the family of dyadic rectangles
$R\subset\Omega$, $R=Q_1\times Q_2$, with $Q_i\in\mathscr{D}_i^a$,  which are maximal in the $i$th ``direction'', $i=1,2$, we  define $m_i(\widetilde{\Omega})$ similarly.
Also recall that  $m(\Omega)$ is the set of all maximal dyadic rectangles contained in $\Omega$. 
Then for any $R=Q_1\times Q_2 \in m(\Omega)$,  
set $\widehat{R}:=\widehat{Q}_1\times Q_2$.
 By definition of $\widehat{Q}_1$  in page~\pageref{def:widehatQ2Q1}, one has that  $Q_1\subset \widehat{Q}_1$,  $\mu(\widehat{R}\cap\Omega)>{\mu(\widehat{R})/2}$, and that $\widehat{Q}_1\in\mathscr{D}^a_1$ is maximal  with respect to these properties, hence $\widehat{R}\in m_1(\widetilde{\Omega})$. Similarly, set
 $\widehat{\widehat{R}}:=\widehat{Q}_1\times \widehat{Q}_2 \in
m_2(\widetilde{\widetilde{\Omega}}),$ since  $Q_2\subset \widehat{Q}_2$, $ \mu(\widehat{\widehat{R}}\cap\widetilde{\Omega})>{\mu(\widehat{\widehat{R}})/2}$, and  $\widehat{Q}_2\in\mathscr{D}_2^a$ is maximal  with respect to these properties.

The set $\Omega$ is a placeholder, rectangles $R$ refer back to $\Omega$, rectangles $\widehat{R}$ to $\widetilde{\Omega}$, and rectangles $\widehat{\widehat{R}}$ to $\widetilde{\widetilde{\Omega}}$.
However we want to relate to the support of the $(p,q)$-atom  $a$ for the estimates, hence we will consider the
$(\ell_1,\ell_2)$-enlargement of these sets. Specifically   echoing the $*$ notation we are using for $\Omega_*$  the support of $a$, we denote
$\widetilde{\Omega}_*:=(\widetilde{\Omega})_{\ell_1,\ell_2}$  and
$\widetilde{\widetilde{\Omega}}_*:=(\widetilde{\widetilde{\Omega}})_{\ell_1,\ell_2}$. We will also consider appropriate $(\ell_1,\ell_2)$-enlargements of the rectangles, specifically~$\widehat{\widehat{R}}_*:= 2^{\ell_1}\widehat{Q}_1\times  2^{\ell_2}\widehat{Q}_2$ and~$R_*=2^{\ell_1}Q_1\times 2^{\ell_2}Q_2$.

Decompose $ \|S(a)\|_{L^p( X_1\times X_2 )}^p$ into pieces that are near or far from {$\Omega_*$ (the support of $a$).} 
 $$ \|S(a)\|_{L^p( X_1\times X_2 )}^p = A +B, \quad \mbox{where }$$
\begin{eqnarray*}
A & := &  \int_{\cup_{R\in m(\Omega)}100\overline{C}\widehat{\widehat{R}}_*}
|S(a)(x_1,x_2)|^p \, d\mu_1(x_1) \, d\mu_2(x_2) \quad\quad\quad \mbox{(near $\Omega_*$)}, \\
B & := & \int_{(\cup_{R\in m(\Omega)}100\overline{C}\widehat{\widehat{R}}_*)^c}
|S(a)(x_1,x_2)|^p \, d\mu_1(x_1) \, d\mu_2(x_2) \quad \mbox{(far from $\Omega_*$)}.
\end{eqnarray*}
Here $\overline{C}\widehat{\widehat{R}}_*:= C_12^{\ell_1}\widehat{Q}_1\times C_2 2^{\ell_2}\widehat{Q}_2$. The constants
$C_i=2(A_0^{(i)})^2$,  for $i=1,2$, are the dilation constants appearing in the support of the rectangular atoms  property~(3)(i) of Definition~\ref{def-of-p-q-atom}, and the parameters $\ell_i$,  for $i=1,2$, are the enlargement parameters in the support of  the $(p,q)$-atom in property~(1)  of Definition~\ref{def-of-p-q-atom}.  To ease notation, we will denote $\overline{C}_i=C_i2^{\ell_i}$ for $i=1,2$.  This ensures that  $ \overline{C}_1Q_1\times \overline{C}_2Q_2 \subset \overline{C} \widehat{\widehat{R}}_*$, and ${\rm supp}(a) \subset \cup_{R\in m(\Omega)}\overline{C}\widehat{\widehat{R}}_*$.

Applying H\"older's inequality with exponent $s=q/p>1$, the desired estimate $A\lesssim 1$ for the  integral $A$  follows from the  $L^q$-boundedness of $S$ and the $L^q$-norm estimate of the atom $a$ as in (2) of Definition~\ref{def-of-p-q-atom}. 
More precisely,
 \begin{align*}
 A & \lesssim \|a\|_{L^q(X_1\times X_2)}^p \big ( \mu (\cup_{R\in m(\Omega)}100\overline{C}\widehat{\widehat{R}}_*)\big )^{1-\frac{p}{q}}\\
 & \lesssim_q \big ( (1+\ell_1 \omega_1+\ell_2 \omega_2 )2^{\ell_1\omega_1} 2^{\ell_2\omega_2}\mu ({\Omega})\big )^{\frac{p}{q}-1}\big ((1+\ell_1 \omega_1+\ell_2 \omega_2 )(100\overline{C}_1)^{\omega_1}(100\overline{C}_2 )^{\omega_2} \mu (\widetilde{\widetilde{\Omega}})\big )^{1-\frac{p}{q}}\\
 & \lesssim_q  \big (\mu ({\Omega})\big )^{\frac{p}{q}-1}\big (\mu (\widetilde{\widetilde{\Omega}})\big )^{1-\frac{p}{q}} \\ &\lesssim_q \; 1.
 \end{align*}
 In the second  inequality, similar to~\eqref{eqn:measure-enlargements}, we again  used the $L\log_+ L$ to weak $L^1$ estimate of the strong maximal function to estimate the upper bound of $\mu (\cup_{R\in m(\Omega)}100\overline{C}\widehat{\widehat{R}}_*)$.
  In the last inequality we used the fact that $\mu({\Omega}) \sim \mu(\widetilde{\widetilde{\Omega}})$.

Using the decomposition of  the atom $a$ as in (3)(ii-a)  of Definition \ref{def-of-p-q-atom}, 
the sublinearity of the product square function $S$, and that $p\leq 1,$ the integral $B$ can be estimated as follows:
$$B\leq \sum_{R\in m(\Omega)}\int_{(100\overline{C}\widehat{\widehat{R}}_*)^c}
|S(a_R)(x_1,x_2)|^p  \, d\mu_1(x_1) \, d\mu_2(x_2).$$

We split the integral over $(100\overline{C}\widehat{\widehat{R}}_*)^c$ into two parts, one over  $(100\overline{C}_1 \widehat{Q}_1)^c\times X_2$ and  the other over $X_1\times (100\overline{C}_2\widehat{Q}_2)^c$. Denote
the sum over  $R\in m(\Omega)$ of the first integrals by $B_1$ and  of the second integrals by $B_2,$ respectively, so that $B\leq B_1+B_2$. It suffices to estimate $B_1$ since the estimate for $B_2$ is similar by symmetry.

To estimate $B_1,$ we  further  split each integral into two pieces, one over $(100\overline{C}_1\widehat{Q}_1)^c\times 100\overline{C}_2Q_2$ and the other over $(100\overline{C}_1\widehat{Q}_1)^c\times (100\overline{C}_2 Q_2)^c$. 
Denote the sum over  $R\in m(\Omega)$ of the first integrals by $B_{11}$ and of the second integrals $B_{12}$  respectively, so that $B_1= B_{11}+B_{12}$.
\subsubsection*{Estimate for $B_{11}$}
Applying Fubini for the integrals, then H\"older's inequality on the second variable with exponent $s=q/p>1$, and using the doubling property of $\mu_2$, we  estimate
\begin{eqnarray*}
B_{11} & = &
\sum_{R\in m(\Omega)}\int_{(100\overline{C}_1\widehat{Q}_1)^c\times (100\overline{C}_2{Q}_2)}
|S(a_R)(x_1,x_2)|^p  \, d\mu_1(x_1) \, d\mu_2(x_2)\\
 &\lesssim & \sum_{R\in m(\Omega)}\big ( (\overline{C}_2)^{\omega_2}\mu_2(Q_2)\big )^{1-\frac{p}{q}}
\int_{x_1\not\in100\overline{C}_1\widehat{Q}_1}\bigg[ \int_{X_2} |S(a_R)(x_1,x_2)|^q \, d\mu(x_2)\bigg ]^{\frac{p}{q}} d\mu (x_1).
\end{eqnarray*}
We estimate the inner integral  on the right-hand side using an $L^q$-vector-valued one-parameter square function estimate with respect to the variable $x_2$ for $\mu_1$-a.e.$\,x_1$, where we consider $x_1$ a fixed parameter.  More precisely,  let $F: X_2\to {L^q_{\ell^2(\mathbb{S})}(X_2,\mu_2)=: L^q_{\ell^2}(X_2)}$ where $\mathbb{S}$ is a countable set,
meaning that for each $x_2\in X_2$, $F(x_2) = \{F_k(x_2)\}_{k\in\mathbb{S}}\in \ell^2(\mathbb{S})$ where
$\|F(\cdot)\|_{\ell_2(\mathbb{S})}\in L^q(X_2)$, furthermore we let
$\|F\|_{L^q_{\ell_2}(X_2)}:= \big \| \|F(\cdot)\|_{\ell_2(\mathbb{S})} \big \|_{L^q(X_2)}$.
Then, using the notation $\widetilde{\chi}_{Q_{\alpha_i}^{k_i}}= \chi_{Q_{\alpha_i}^{k_i}} / { \mu_i ({Q_{\alpha_i}^{k_i}})}$ {(denoting an $L^1$-normalization instead of denoting an $L^2$-normalization)} and where $Q^{k_i}_{\alpha_i}\in\mathscr{D}_i^W$, we define
$$ S_2 (F)(x_2): = \Big ( \sum_{k_2\in\mathbb{Z}}\sum_{\alpha_2 \in \mathscr{Y}^{k_2}}
\Big \| \langle \psi^{k_2}_{\alpha_2}, F\rangle_{L^2(X_2)} \Big\|^2_{\ell^2(\mathbb{S})}\widetilde{\chi}_{Q_{\alpha_2}^{k_2}}(x_2)  \Big )^{\frac12}.$$
Here $\langle \psi^{k_2}_{\alpha_2}, F\rangle_{L^2(X_2)}$ denotes the sequence $\{ \langle \psi^{k_2}_{\alpha_2}, F_{k}\rangle_{L^2(X_2)}\}_{k\in\mathbb{S}}$.  For all $q>1$ the following vector-valued inequality holds:
$  \ \big \| S_2(F)\big\|_{L^q(X_2)}\; \leq  \; C_q \big \| F\big\|_{L^q_{\ell^2}(X_2)}$.
We point out that $\{\psi^{k_2}_{\alpha_2}\}_{k_2\in \mathbb{Z},\mathscr{Y}^{k_2}}$ is an orthogonal wavelet basis in $X_2$ satisfying suitable size, smoothness, and cancellation conditions. Hence by following the proof of the $L^q$-boundedness of the Littlewood-Paley square function as in \cite{HLW}   for $q>1$, we obtain the $L^q$-boundedness of the vector-valued Littlewood-Paley operator $S_2$. For the Euclidean version, we refer to \cite[Section~5.1.2]{Gra}.

With these preliminaries in mind, we can now estimate for $\mu_1$-a.e. $x_1\in X_1$ the $L^q(X_2)$-norm of $S(a_R)(x_1,\cdot)$.
\begin{eqnarray*}
&& \hskip -.2in \int_{X_2} |S(a_R)(x_1,x_2)|^q \,d\mu_2(x_2 )\\
&&= \int_{X_2} \Big [ \sum_{k_2\in\mathbb{Z}}\sum_{\alpha_2 \in \mathscr{Y}^{k_2}}\sum_{k_1\in\mathbb{Z}}\sum_{\alpha_1\in \mathscr{Y}^{k_1}}
\big| \langle \psi^{k_1}_{\alpha_1}\psi^{k_2}_{\alpha_2}, a_R\rangle_{L^2(X_1\times X_2)}
\big|^2 {\widetilde{\chi}_{Q_{\alpha_1}^{k_1}}(x_1)}\widetilde{\chi}_{Q_{\alpha_2}^{k_2}}(x_2) \Big]^{\frac{q}{2}}  d\mu_2(x_2)\\
&& = \int_{X_2} \Big [\sum_{k_2\in\mathbb{Z}}\sum_{\alpha_2 \in \mathscr{Y}^{k_2}} \Big ( \sum_{k_1\in\mathbb{Z}}\sum_{\alpha_1\in \mathscr{Y}^{k_1}}
\big| \big \langle \psi^{k_2}_{\alpha_2},\langle \psi^{k_1}_{\alpha_1}, a_R\rangle_{L^2(X_1)} \big \rangle_{L^2(X_2)}
\big|^2 {\widetilde{\chi}_{Q_{\alpha_1}^{k_1}}(x_1)}\Big ) \widetilde{\chi}_{Q_{\alpha_2}^{k_2}}(x_2) \Big]^{\frac{q}{2}}  d\mu_2(x_2)\\
&& = \int_{X_2} \Big [\sum_{k_2\in\mathbb{Z}}\sum_{\alpha_2 \in \mathscr{Y}^{k_2}}
\Big \| \langle \psi^{k_2}_{\alpha_2}, F^{(x_1)}\rangle_{L^2(X_2)} \Big\|^2_{\ell^2(\mathbb{S})}\widetilde{\chi}_{Q_{\alpha_2}^{k_2}}(x_2) \Big]^{\frac{q}{2}}   d\mu_2(x_2)\\
&& = \int_{X_2} \big | S_2(F^{(x_1)})(x_2)\big |^q d\mu_2(x_2) \; \leq  \; C\int_{X_2} \big \| F^{(x_1)}(x_2)\big\|^q_{\ell^2(\mathbb{S})}\, d\mu_2(x_2).
\end{eqnarray*}
Here $F^{(x_1)}(x_2)=\{ F^{x_2,x_1}_{k_1,\alpha_1}\}_{k_1\in\mathbb{Z},\alpha_1\in\mathscr{Y}^{k_1}}$, where  $F_{k_1,\alpha_1}^{x_2,x_1}:= \langle a_R(\cdot,x_2), \psi^{k_1}_{\alpha_1}\rangle_{L^2(X_1)}\, \big (\widetilde{\chi}_{Q^{k_1}_{\alpha_1}}(x_1)\big )^{1/2}$ and  $\mathbb{S}=\{(k_1,\alpha_1): \; k_1\in\mathbb{Z}, \; \alpha_1\in\mathcal{Y}^{k_1}\}$ is a countable set.

Altogether we now estimate the term $B_{11}$ as follows:
\begin{eqnarray*}
B_{11}
&\lesssim & \sum_{R\in m(\Omega)}\big ( (\overline{C}_2)^{\omega_2}\mu_2(Q_2)\big )^{1-\frac{p}{q}} \int_{x_1\not\in100\overline{C}_1\widehat{Q}_1}
\bigg [\int_{X_2} \big \| F^{(x_1)}(x_2)\big\|^q_{\ell^2(\mathbb{S})}\, d\mu_2(x_2)\bigg]^{\frac{p}{q}}   d\mu_1(x_1) \\
&= & \sum_{R\in m(\Omega)}\big ( (\overline{C}_2)^{\omega_2} \mu_2(Q_2)\big )^{1-\frac{p}{q}} \\
& & \hskip -.5in\times \hskip -.2in\mathop{\int}_{x_1\not\in100\overline{C}_1\widehat{Q}_1}\bigg[ \mathop{\int}_{X_2}\Big[\sum_{k_1\in\mathbb{Z}}\sum_{{\alpha_1}\in\mathscr{Y}^{k_1}}
\big|\mathop{\int}_{X_1}\psi_{\alpha_1}^{k_1}(y_1)a_R(y_1,x_2) \, d\mu_1(y_1)\big|^2 {\widetilde{\chi}_{Q_{\alpha_1}^{k_1}}(x_1)} \Big]^{\frac{q}{2}}  d\mu_2(x_2)\bigg]^{\frac{p}{q}}   d\mu_1(x_1).
\end{eqnarray*}

Applying the decomposition \eqref{decomposition of wavelet into atom} in Lemma~\ref{lemma-decomposition} to $\psi_{\alpha_1}^{k_1}$, we get that for $\gamma>\omega_1$ (where $\gamma$ is to be determined later)
and for $\overline{C}_1=C_1\,2^{\ell_1}$ playing the role of  $\overline{C}$,
    $$    {\psi_{\alpha_1}^{k_1}(y_1) }
        = \sqrt{\mu\big (B(y_{\alpha_1}^{k_1},\delta^{k_1})\big )}\sum_{\ell=0}^\infty (2^{\ell}\overline{C}_1)^{-\gamma}\varphi^{\gamma,\overline{C}_1}_{\ell,k_1,\alpha_1}(y_1), $$
Substituting   and noting that ${\mu\big (B(y_{\alpha_1}^{k_1},\delta^{k_1})\big )} \,\widetilde{\chi}_{Q_{\alpha_1}^{k_1}}(x_1)={\chi}_{Q_{\alpha_1}^{k_1}}(x_1)$, we  continue estimating:
\begin{eqnarray*}
B_{11} &\lesssim &\sum_{R\in m(\Omega)} \big ( (\overline{C}_2)^{\omega_2}\mu_2(Q_2)\big )^{1-\frac{p}{q}} \\
&&\hskip -.5in\times  \hskip -.2in \mathop{\int}_{x_1\not\in100\overline{C}_1\widehat{Q}_1}\bigg[\mathop{\int}_{X_2}\Big[ \mathop{\sum_{k_1\in\mathbb{Z}}}_{\alpha_1 \in\mathscr{Y}^{k_1}} 
\Big| \sum_{\ell=0}^\infty (2^{\ell}\overline{C}_1)^{-\gamma}\big \langle  \varphi^{\gamma,\overline{C}_1}_{\ell,k_1,\alpha_1}, a_R(\cdot,x_2)\big \rangle_{L^2(X_1)}
\Big|^2 {{\chi}_{Q_{\alpha_1}^{k_1}}(x_1)} \Big]^{\frac{q}{2}}  d\mu_2(x_2)\bigg]^{\frac{p}{q}}   d\mu_1(x_1).
\end{eqnarray*}
First applying the Cauchy-Schwarz  inequality to the sum over $\ell\geq 0$ after factoring out the constant $(\overline{C}_1)^{-\gamma}$,  
and considering the decaying exponential factor as a weight so that $\sum_{\ell\geq 0} 2^{-\ell\gamma}<\infty$ is a harmless finite constant.
Second,  interchanging sums over $\ell$ and over $(k_1, \alpha_1)\in \mathbb{S}$, applying H\"older's inequality  with exponent $s=q/2>1$  (we are in the case $q\geq 2$ and when $q=2$ this step is unnecessary. When $1<q<2$ the power $q/2<1$ and it will travel into the sum over $\ell$, the only difference being that the exponential factor will be  $2^{-\frac{\ell\gamma q}{2}}$ instead of $2^{-\ell\gamma} $)
 to the sum over $\ell$ and considering the decaying exponential factor as a weight as before. Third,
interchanging the sum over $\ell$ and the integral over $X_2$, and using the fact that $p/q<1$ and the exponent can travel inside the sum over $\ell$. Finally, interchanging the sum over $\ell$ with the outer integral over  $(100\overline{C}_1\widehat{Q}_1)^c$ and then with the sum over $R\in m(\Omega)$,  we  find that
\begin{eqnarray*}
B_{11}  &\lesssim  & (\overline{C}_1)^{-\gamma p}
\sum_{R\in m(\Omega)}\big ( (\overline{C}_2)^{\omega_2}\mu_2(Q_2)\big )^{1-\frac{p}{q}} \\
&& \hskip -.5in \times \mathop{\int}_{x_1\not\in100\overline{C}_1\widehat{Q}_1}\sum_{\ell=0}^\infty 2^{{-\frac{\ell\gamma p}{{q}}}}\bigg[ \mathop{\int}_{X_2}\Big[\mathop{\sum_{k_1\in\mathbb{Z}}}_{\alpha_1\in\mathscr{Y}^{k_1}} 
\Big| \big \langle\varphi^{\gamma,\overline{C}_1}_{\ell,k_1,\alpha_1}, a_R(\cdot,x_2)\big \rangle_{L^2(X_1)}
\Big|^2 {{\chi}_{Q_{\alpha_1}^{k_1}}(x_1)} \Big]^{\frac{q}{2}}  d\mu_2(x_2)\bigg]^{\frac{p}{q}}   d\mu_1(x_1)\\
& \lesssim & (\overline{C}_1)^{-\gamma p}
\sum_{\ell=0}^\infty 2^{-\frac{\ell\gamma p}{{q}}} \sum_{R\in m(\Omega)} \big ( (\overline{C}_2)^{\omega_2}\mu_2(Q_2)\big )^{1-\frac{p}{q}}
 \\
& & \hskip -.5in \times \mathop{\int}_{x_1\not\in100\overline{C}_1\widehat{Q}_1}\bigg[ \mathop{\int}_{X_2}\Big[
\mathop{\sum_{k_1\in\mathbb{Z}}}_{\alpha_1\in\mathscr{Y}^{k_1}} 
\Big| \big \langle \varphi^{\gamma,\overline{C}_1}_{\ell,k_1,\alpha_1}, a_R(\cdot,x_2)\big \rangle_{L^2(X_1)}
\Big|^2 {{\chi}_{Q_{\alpha_1}^{k_1}}(x_1)} \Big]^{\frac{q}{2}}  d\mu_2(x_2)\bigg]^{\frac{p}{q}}   d\mu_1(x_1).
\end{eqnarray*}
(In the case $1<q<2$ the only difference in the estimate is that instead of $2^{-\frac{\ell\gamma p}{q}}$ one gets the exponential $2^{-\frac{\ell\gamma p}{2}}$, where~$q$ has been replaced by~2 in the exponent's denominator.)

The support of $a_R$ is $\overline{C}_1Q_1\times \overline{C}_2Q_2$, note that if $y_1\in \overline{C}_1Q_1$ then  $d_1(y_1, z_1)\leq {C_1^1\overline{C}_1} \, \ell(Q_1)$, where $z_1$ is the center of $Q_1$ and $C^1_1$ is the dilation constant for the outer balls in the fixed  dyadic grid $\mathscr{D}_1^a$ on $X_1$. Recall that $R=Q_1\times Q_2$.  If ${C_1^1\overline{C}_1}\, \ell(Q_1)\leq \delta_1^{k_1}$, then $d_1(y_1,z_1)\leq \delta_1^{k_1}$ and using the  smoothness property~(iii)  in Lemma~\ref{lemma-decomposition} of $\varphi^{\gamma,\overline{C}_1}_{\ell,k_1,\alpha_1}$,  the cancellation condition~(3)(ii) in the first variable of $a_R$ in Definition~\ref{def-of-p-q-atom}, together with the geometric considerations and H\"older's inequality,  we conclude that when both $x_1$ and $y^{k_1}_{\alpha_1}$ are in $Q^{k_1}_{\alpha_1}$, \label{inner-product}
\begin{eqnarray*}
\left  |\big \langle \varphi^{\gamma,\overline{C}_1}_{\ell,k_1,\alpha_1}(\cdot), a_R(\cdot,x_2)\big \rangle_{L^2(X_1)} \right |
& \leq  &  \int_{{\overline{C}_1Q_1}}
{\big | \varphi^{\gamma,\overline{C}_1}_{\ell,k_1,\alpha_1}(y_1) - \varphi^{\gamma,\overline{C}_1}_{\ell,k_1,\alpha_1}(z_1)\big | }\big |a_R(y_1,x_2) \big | \, d\mu_1(y_1)  \\
&& \hskip -.7in\lesssim \; \frac{\big (\overline{C}_1\ell(Q_1)\big )^{\eta_1} (\overline{C}_12^{\ell}\delta_1^{k_1})^{-\eta_1}(\overline{C}_12^{\ell})^{\omega_1}}{\mu_1\big (B_{X_1}(y^{k_1}_{\alpha_1}, \overline{C}_12^{\ell}\delta_1^{k_1}) \big )} \int_{{\overline{C}_1Q_1}}\big |a_R(y_1,x_2) \big | \, d\mu_1(y_1) \\
&& \hskip -.7in \lesssim \; \frac{\ell(Q_1) ^{\eta_1} (2^{\ell}\delta_1^{k_1})^{-\eta_1}(\overline{C}_12^{\ell})^{\omega_1}}{\mu_1\big (B_{X_1}(x_1, \overline{C}_12^{\ell}\delta_1^{k_1}) \big )}\big ( (\overline{C}_1)^{\omega_1} \mu_1 (Q_1)\big )^{{\frac{q-1}{q}}} \|a_R(\cdot, x_2)\|_{L^q(X_1)}.
\end{eqnarray*}
Here the doubling condition on the measure allows us to compare nearby balls with the same radius; specifically,  $\frac{\mu_1(B_{X_1}({x_1}, 2^{\ell}\delta_1^{k_1}))}{\mu_1(B_{X_1}(y^{k_1}_{\alpha_1}, 2^{\ell}\delta_1^{k_1})) } \sim 1$,  since we are assuming that $x_1$ and $y^{k_1}_{\alpha_1}$ are in $Q^{k_1}_{\alpha_1}$.

Assume now that  $C_1^1\overline{C}_1\ell(Q_1)> \delta_1^{k_1}$. 
Recall that to get the desired estimate for the inner product it suffices to obtain the estimate for the inner product with differences of the functions $(2^{\ell}\overline{C}_1)^{\gamma}\Lambda_{\ell}^{\overline{C}_1}$ instead of differences of the functions $\varphi^{\gamma,\overline{C}_1}_{\ell,k_1,\alpha_1}$, the other piece can be estimated as above.
 Therefore we can assume that  $y_1\in \overline{C_1}Q_1\cap   {\rm supp}(\Lambda_{\ell}^{\overline{C}_1})$, this means
$2^{\ell-3}\overline{C}_1\delta_1^{k_1}\leq d_1(y_1,y^{k_1}_{\alpha_1})\leq (A_0^{(1)})^22^{\ell}\overline{C}_1\delta_1^{k_1}$ and $d_1(y_1,z_1)\leq C_1^1\overline{C}_1\ell(Q_1)$. We also know that $x_1\in (100 (2A_0^{(1)}) \overline{C}_1\widehat{Q}_1)^c$, hence $d_1(z_1, x_1)\geq 100 (2A_0^{(1)}) C^1_1\overline{C}_1\ell (\widehat{Q}_1)\geq 100(2A_0^{(1)}) C^1_1\overline{C}_1\ell (Q_1)$ and $x_1\in Q^{k_1}_{\alpha_1}$ hence $d_1(z_1, y^{k_1}_{\alpha_1}) \sim  d_1(y_1, y^{k_1}_{\alpha_1}) \geq 100 (2A_0^{(1)}) C^1_1\overline{C}_1\ell (Q_1) $.
From the proof of Lemma~\ref{lemma-decomposition},  we can use a test-function-like  smoothness
property   for  the function  $\Lambda^{\overline{C}_1}_{\ell}$ encoded in~\eqref{test-function-estimate-Lambda-x} and valid when $y_1\in {\rm supp}(\Lambda^{\overline{C}_1}_{\ell})$ and
$d(y_1,z_1)\leq (2A_0^{(1)})^{-1}\big (\delta^{k_1}_1 + d(y_1,y^{k_1}_{\alpha_1})\big ),$ which both hold by the assumptions made in this case, namely:
\[
(2^{\ell}\overline{C}_1)^{\gamma}| \Lambda^{\overline{C}_1}_{\ell}(y_1)-\Lambda^{\overline{C}_1}_{\ell}(z_1)|
\lesssim \frac{
(\overline{C}_12^{\ell}\delta^{k_1}_1)^{-\eta_1}}{\mu \big (B(y_{\alpha_1}^{k_1},\delta^{k_1}_1 )\big ) + \mu\big (B(y_1,d(y_1,y^{k_1}_{\alpha_1}))\big )} d(y_1,z_1)^{\eta_1}.
\]
Furthermore since nearby balls with same radius have comparable measure by the doubling property,
$\mu\big (B(y_1,d(y_1,y^{k_1}_{\alpha_1}))\big )\sim \mu\big (B(y_{\alpha_1}^{k_1},\overline{C}_12^{\ell}\delta_1^{k_1})\big )$ we get that in our case
\begin{equation} \label{test-function-estimate-Lambda-x2}
(2^{\ell}\overline{C}_1)^{\gamma}| \Lambda^{\overline{C}_1}_{\ell}(x)-\Lambda^{\overline{C}_1}_{\ell}(y)|
\lesssim \frac{
(\overline{C}_12^{\ell}\delta^{k_1}_1)^{-\eta_1}}{\mu\big (B(y_{\alpha_1}^{k_1},\overline{C}_12^{\ell}\delta_1^{k_1})\big )}\big ( C_1^1\overline{C}_1\ell(Q_1)\big )^{\eta_1}.
\end{equation}

Inequality~\eqref{test-function-estimate-Lambda-x2} together with the geometric considerations and H\"older's inequality,  and given that  both $x_1$ and $y^{k_1}_{\alpha_1}$ are in $Q^{k_1}_{\alpha_1}$, yield
\begin{eqnarray*}
  \int_{{\overline{C}_1Q_1}\cap \,{\rm supp} (\Lambda^{\overline{C}_1}_{\ell})} 
 (\overline{C}_1 2^{\ell})^{\gamma}\big |\Lambda_{\ell}^{\overline{C}_1}(y_1)-\Lambda_{\ell}^{\overline{C}_1}(z_1) \big | \big |a_R(y_1,x_2) \big | \, d\mu_1(y_1)   && \\
& & \hskip -3in\lesssim \;  \frac{ \ell(Q_1)^{\eta_1} (2^{\ell}\delta_1^{k_1})^{-\eta_1}}
{\mu_1\big (B_{X_1}(y^{k_1}_{\alpha_1}, \overline{C}_12^{\ell}\delta_1^{k_1}) \big )}\int_{{\overline{C}_1Q_1}}\big |a_R(y_1,x_2) \big | \, d\mu_1(y_1) \\
& & \hskip -3in \lesssim  \frac{\ell(Q_1)^{\eta_1} (2^{\ell}\delta_1^{k_1})^{-\eta_1}}
{\mu_1\big (B_{X_1}(x_1, \overline{C}_12^{\ell}\delta_1^{k_1}) \big )} \big ( (\overline{C}_1)^{\omega_1}  \mu_1 (Q_1)\big )^{{\frac{q-1}{q}}} \|a_R(\cdot, x_2)\|_{L^q(X_1)}.
\end{eqnarray*}
Note that in the first $\lesssim$ the constant $(C_1^1)^{\eta_1}\leq C_1^1$ has been absorbed since it is bounded above by a constant depending only on the geometric constants of the space $X_1$.

Therefore we conclude that in all cases, when both $x_1$ and $y^{k_1}_{\alpha_1}$ are in $Q^{k_1}_{\alpha_1}$,
\[\left  |\big \langle \varphi^{\gamma,\overline{C}_1}_{\ell,k_1,\alpha_1}, a_R(\cdot,x_2)\big \rangle_{L^2(X_1)} \right | \lesssim
\frac{ \ell(Q_1)^{\eta_1}(2^{\ell}\delta_1^{k_1})^{-\eta_1}(\overline{C}_12^{\ell})^{\omega_1}}{\mu_1\big (B_{X_1}(x_1, \overline{C}_12^{\ell}\delta_1^{k_1}) \big )} \big ( (\overline{C}_1)^{\omega_1} \mu_1 (Q_1)\big )^{{\frac{q-1}{q}}} \|a_R(\cdot, x_2)\|_{L^q(X_1)}.\]

Notice that in above calculation  
$\ell(Q_1)$ refers to the underlying dyadic grid $\mathscr{D}_1^a$ for the atom, possibly different than the reference dyadic grid $\mathscr{D}_1^W$ for the wavelets on $X_1$. Also note that the inequalities $\lesssim$ and the similarities $\sim$ introduce constants depending only on the geometric constants of the space of homogeneous type, in this case $X_1$.

{Now recall that supp$(\varphi^{\gamma,\overline{C}_1}_{\ell,k_1,\alpha_1})\subset B_{X_1}(y^{k_1}_{\alpha_1},{ 2(A_0^{(1)})^2}\, 2^{\ell}\overline{C}_1\delta_1^{k_1})$, so the inner product we just estimated will be nonzero only when $(\overline{C}_1\,Q_1)\cap B_{X_1}(y^{k_1}_{\alpha_1},{ 2(A_0^{(1)})^2}\, 2^{\ell}\overline{C}_1\delta_1^{k_1})\neq \emptyset$,
 where  $y^{k_1}_{\alpha_1}$ is the center of the cube $Q^{k_1}_{\alpha_1}\in \mathscr
 {D}_1^W$ that contains $x_1\not\in100\overline{C}_1\widehat{Q}_1$.
Therefore,  when estimating $B_{11}$, in the sum over $k_1$ and $\alpha_1$ the only scales that intervene are
those integers $\ell \geq 0$ such that $2^{\ell}\overline{C}_1\delta_1^{k_1}\sim d_1(x_1,z_1)$, where $z_1$ is the center of $Q_1$ (it helps to draw a picture to understand the geometry).
With this in mind, applying the above  estimate on the inner product we conclude that
\begin{eqnarray*}
B_{11} &\lesssim & (\overline{C}_1)^{-\gamma p}
 \sum_{\ell=0}^\infty 2^{-\frac{\ell\gamma p}{{q}}} \sum_{R\in m(\Omega)}\big ( (\overline{C}_2)^{\omega_2} \mu_2(Q_2)\big )^{1-\frac{p}{q}}
 \\
& & \hskip -1cm \times \hskip -.3in \mathop{\int}_{x_1\not\in100\overline{C}_1\widehat{Q}_1}\bigg[ \mathop{\int}_{X_2}\Big[\sum_{{k_1,\alpha_1:\,2^{\ell}\overline{C}_1\delta_1^{k_1}\sim d_1(x_1,z_1)}}
\bigg| \frac{{\ell(Q_1)^{\eta_1} (2^{\ell}\delta_1^{k_1})^{-\eta_1}(\overline{C}_12^{\ell})^{\omega_1}}}{\mu_1\big (B_{X_1}({x_1}, 2^{\ell}\delta_1^{k_1}) \big )} \big ( (\overline{C}_1)^{\omega_1}\mu_1 (Q_1)\big )^{1-\frac1q}  \\
&& \hskip 1cm \times \;  \|a_R(\cdot, x_2)\|_{L^q(X_1)} \bigg|^2\,{{\chi}_{Q_{\alpha_1}^{k_1}}(x_1)} \Big]^{\frac{q}{2}}  d\mu_2(x_2)\bigg]^{\frac{p}{q}}   d\mu_1(x_1).
\end{eqnarray*}
There is exactly one dyadic cube $Q^{k_1}_{\alpha_1}\in\mathscr{D}_1^W$ in generation $k_1$ containing $x_1$, so the double sum over $k_1, \alpha_1$ reduces to a single sum over $k_1$. Furthermore, note that
when $2^{\ell}\overline{C}_1\delta_1^{k_1}\sim d_1(x_1,z_1)$ then $\mu_1\big (B_{X_1}({x_1}, 2^{\ell}\overline{C}_1\delta_1^{k_1}) \big ) \sim\mu_1\big (B_{X_1}({x_1}, d_1(x_1,z_1)) \big ) \lesssim (\overline{C}_1)^{\omega_1}\mu_1\big (B_{X_1}({x_1}, 2^{\ell}\delta_1^{k_1})\big )$. Therefore
\begin{eqnarray*}
B_{11} &\lesssim & (\overline{C}_1)^{-\gamma p}
 \sum_{\ell=0}^\infty 2^{-\frac{\ell\gamma p}{{q}}} \sum_{R\in m(\Omega)}
\big ((\overline{C}_2)^{\omega_2}\mu_2(Q_2)\big )^{1-\frac{p}{q}}
\|a_R\|_{L^q(X_1\times X_2)}^p
 \\
& & \hskip -1cm \times \int_{x_1\not\in100\overline{C}_1\widehat{Q}_1}
\bigg[ \sum_{k_1: {\,2^{\ell}\overline{C}_1\delta_1^{k_1} \sim  d_1(x_1,z_1)}}
\Big| \frac{\ell(Q_1)^{\eta_1}(2^{\ell}\delta_1^{k_1})^{-\eta_1}(2^{\ell}\overline{C}_1)^{\omega_1}}{\mu_1\big (B_{X_1}({x_1}, d_1(x_1,z_1)) \big )} \big ((\overline{C}_1)^{\omega_1}\mu_1 (Q_1)\big )^{1-\frac1q}
\Big|^2 
\bigg]^{\frac{p}{2}}   d\mu_1(x_1) \\
&\lesssim & (\overline{C}_1)^{-\gamma p}
 \sum_{\ell=0}^\infty 2^{-\frac{\ell\gamma p}{{q}}} \sum_{R\in m(\Omega)}
\big ((\overline{C}_2)^{\omega_2}\mu_2(Q_2)\big )^{1-\frac{p}{q}}
\|a_R\|_{L^q(X_1\times X_2)}^p \, 2^{\ell\omega_1p} \big ( (\overline{C}_1)^{\omega_1}\mu_1(Q_1)\big )^{p-\frac{p}{q}}\\
& & \hskip .3cm \times \int_{x_1\not\in100\overline{C}_1\widehat{Q}_1}
\bigg[ \sum_{k_1:  {\,2^{\ell} \overline{C}_1 \delta^{k_1}\sim d_1(x_1,z_1)}}
{ (2^{\ell}\delta^{k_1})^{-2\eta_1}}
\bigg]^{\frac{p}{2}}  \frac{\ell(Q_1)^{\eta_1 p}(\overline{C}_1)^{\omega_1p}}{\mu_1\big (B_{X_1}({x_1}, d_1(x_1,z_1)) \big )^p} \, d\mu_1(x_1).
\end{eqnarray*}
Notice that the sum over $k_1$ is a geometric sum comparable to $d_1(x_1,z_1)^{-2\eta_1}(\overline{C}_1)^{2\eta_1}$. Therefore
\begin{eqnarray*}
B_{11}&\lesssim &  (\overline{C}_1)^{(\omega_1-\gamma) p} \sum_{\ell=0}^\infty 2^{-\frac{\ell\gamma p}{{q}}} \sum_{R\in m(\Omega)}
\big ((\overline{C}_2)^{\omega_2}\mu_2(Q_2)\big )^{1-\frac{p}{q}}
\|a_R\|_{L^q(X_1\times X_2)}^p 
2^{\ell\omega_1p} \big ((\overline{C}_1)^{\omega_1} \mu_1(Q_1)\big )^{p-\frac{p}{q}}\\
& & \hskip 1cm \times \int_{x_1\not\in100\overline{C}_1\widehat{Q}_1}
\frac{d_1(x_1,z_1)^{-\eta_1 p} \ell(Q_1)^{\eta_1 p}(\overline{C}_1)^{\eta_1p}}{\mu_1\big (B_{X_1}({x_1}, d_1(x_1,z_1))\big )^p}
\,   d\mu_1(x_1).
  \end{eqnarray*}
 The integral  over $(100\overline{C}_1\widehat{Q}_1)^c$ can be further decomposed into integrals over disjoint annuli~$D_{j+1}\setminus D_j$. Here
$D_j:=2^j100\overline{C}_1\widehat{Q}_1$, so that for all $j\geq 0$,
$(100\overline{C}_1\widehat{Q}_1)^c = \cup_{j\geq 0}  (D_{j+1}\setminus D_j)$.
 For $x_1\in D_{j}\setminus D_{j-1}$ we have that 
  $(\overline{C}_1)^{\eta_1 p}d_1(x_1,z_1)^{-\eta_1 p} \sim 2^{-j\eta_1 p} \,  \ell(\widehat{Q}_1)^{-\eta_1 p} $.
 Note that nearby balls with the same radius have comparable mass  by the doubling property of the measure. In particular $\mu_1\big (B_{X_1}(z_1,d_1(x_1,z_1))\big ) \sim \mu_1\big (B_{X_1}(x_1,d_1(x_1,z_1))$, and certainly  $\widehat{Q}_1\subset D_j \subset B_{X_1}(z_1,d_1(x_1,z_1)) \subset D_{j+1}$.  All together,
  we obtain the following estimate  
   \[\mathop{\int}_{x_1\in(100\overline{C}_1\widehat{Q}_1)^c}
\frac{d_1(x_1,z_1)^{-\eta_1 p}\ell(Q_1)^{\eta_1 p}(\overline{C}_1)^{\eta_1p}}{\mu_1\big (B_{X_1}({x_1}, d_1(x_1,z_1))\big )^p} \, d\mu_1(x_1)
  \lesssim \frac{\ell(Q_1)^{\eta_1p}\ell(\widehat{Q}_1)^{-\eta_1 p}}{ \mu_1 (\widehat{Q}_1)^{p-1}
}  \sum_{j\geq 0} \bigg (\frac{\mu_1 (D_{j+1})}{\mu_1 ( D_j)} 2^{-j\eta_1 p}\bigg ),\]
where  the sum over $j$ is comparable to $1$ by the doubling condition of $\mu_1$. Substituting and noting that
$\gamma>\omega_1$ hence $(\overline{C}_1)^{(\omega_1-\gamma)p}<1$,   that $p-1<0$ hence $(\overline{C}_1)^{\omega_1(p-1)}<1$, and recalling that $(\overline{C}_i)^{\omega_i}\sim 2^{\ell_i\omega_i}$ for $i=1,2$,  we obtain
  \begin{eqnarray*}
 B_{11} &\lesssim & 
  \sum_{\ell=0}^\infty 2^{-\frac{\ell\gamma p}{{q}}} \sum_{R\in m(\Omega)}\big (2^{\ell_1\omega_1+\ell_2\omega_2}\mu(R)\big )^{1-\frac{p}{q}}
\|a_R\|_{L^q(X_1\times X_2)}^p \ell(Q_1)^{\eta_1p}\,2^{\ell\omega_1p}  \\
& & \hskip 2in\times \big ((\overline{C}_1)^{\omega_1}\mu_1(Q_1)\big )^{p-1}
\frac{\ell(\widehat{Q}_1)^{-\eta_1 p}}{ \mu_1 (\widehat{Q}_1)^{p-1}}  \\
  &\lesssim &
   \sum_{R\in m(\Omega)}\big (2^{\ell_1\omega_1+\ell_2\omega_2}\mu(R)\big )^{1-\frac{p}{q}} \|a_R\|_{L^q(X_1\times X_2)}^p
  \left [\frac{\ell(Q_1)}{\ell(\widehat{Q}_1)}\right ]^{\eta_1 p} \left [
  \frac{\mu_1(Q_1)}{\mu_1(\widehat{Q}_1)}\right ]^{p-1}   \sum_{\ell=0}^\infty 2^{-\frac{\ell\gamma p}{{q}} +\ell\omega_1p}.
  \end{eqnarray*}
For the geometric sum to converge we need to choose $\gamma > {q}\,\omega_1$ when $q\geq 2$ and when $1<q<2$ we choose $\gamma > 2\omega_1$.} With this choice
and using the doubling property~\eqref{eqn:upper dimension} once more since $p-1<0$, we estimate for $2\leq q<\infty$,
\begin{eqnarray}\label{eqn:B11}
B_{11}&\lesssim&  
2^{(\ell_1\omega_1+\ell_2\omega_2)(1-\frac{p}{q})}\sum_{R\in m(\Omega)}\|a_R\|_{L^q( X_1\times X_2 )}^p\, \mu(R)^{1-\frac{p}{q}} \, w\Big({\ell(Q_1)\over \ell(\widehat{Q}_1)}\Big),
\end{eqnarray}
where $w(x)=x^\alpha$ with $\alpha=p\eta_1+(p-1)\omega_1>0$. This is where we explicitly used the choice of $p>p_0$ where $p_0=\max\big ( \omega_i/(\omega_i+\eta_i): \, i=1,2\}$. It was also used in the definition of $H^p(X_1\times X_2)$ in \cite{HLW}.

To be more precise on how the  doubling condition was used in~\eqref{eqn:B11}.
 Let $\widehat{z}_1$ be the center of $\widehat{Q}_1$ and $z_1$ the center of $Q_1$. Recall that $Q_1\subset \widehat{Q}_1$,  then
\begin{eqnarray*}
\frac{\mu_1(\widehat{Q}_1)}{\mu_1(Q_1)} &\leq & \frac{\mu_1\big (B_{X_1}(\widehat{z}_1,C^1_1\ell(\widehat{Q}_1))\big )}{\mu_1\big (B_{X_1}(z_1,c^1_1\ell(Q_1))\big )}
\, \leq \,  \frac{\mu_1\big (B_{X_1}({z}_1,A_0^{(1)}(d_1(\widehat{z}_1,z_1)+C^1_1\ell(\widehat{Q}_1)))\big )}{\mu_1\big (B_{X_1}(z_1,c^1_1\ell(Q_1))\big )} \\
&\leq & \frac{\mu_1\big (B_{X_1}(z_1,2A_0^{(1)}C^1_1\ell(\widehat{Q}_1))\big )}{\mu_1\big (B_{X_1}(z_1,c^1_1\ell(Q_1))\big )}  \, \lesssim \, \Big (\frac{2A_0^{(1)}C^1_1\ell(\widehat{Q}_1)}{c^1_1\ell(Q_1)}\Big )^{\omega_1}
\, \lesssim \, \Big (\frac{\ell(\widehat{Q}_1)}{\ell(Q_1)}\Big )^{\omega_1}.
\end{eqnarray*}

We continue estimating $B_{11}$.
Applying H\"older's inequality  to the right-hand-side of \eqref{eqn:B11}, with exponent $s=q/p>1$, setting $\widetilde{w}= w^{\frac{q}{q-p}}$, using {property (iii-a)} in the definition  of $(p,q)$-atoms, and applying the Journ\'e-type covering lemma gives   
\begin{eqnarray*}
B_{11}&\lesssim &  
2^{(\ell_1\omega_1+\ell_2\omega_2)(1-\frac{p}{q})}\bigg( \sum_{R\in m(\Omega)} \|a_R\|_{L^q( X_1\times X_2 )}^q\ \bigg)^{\frac{p}{q}}\bigg( \sum_{R\in m(\Omega)} \mu(R) \, {\widetilde{w}}\Big({\ell(Q_1)\over \ell(\widehat{Q}_1)}\Big) \bigg)^{1-\frac{p}{q}}\\
&\lesssim&  
2^{(\ell_1\omega_1+\ell_2\omega_2)(1-\frac{p}{q})} \big ( (1+\ell_1\omega_1+\ell_2\omega_2)2^{\ell_1\omega_1+\ell_2\omega_2} \mu(\widetilde{\Omega})\big )^{{p\over q}-{1}}
\mu(\Omega)^{1-\frac{p}{q}} \; \lesssim \;1. 
\end{eqnarray*}
The last inequality because $\mu(\widetilde{\Omega})\sim \mu(\Omega)$.

For $1<q<2$, setting $\overline{w}=w^{1\over 2}$, $\widetilde{w}=\overline{w}^{q\over q-p}$ and $\widetilde{\widetilde{w}}=\overline{w}^{q\over p}$ and applying the same estimate as above, we obtain
\begin{eqnarray*}
B_{11}&\lesssim&  
2^{(\ell_1\omega_1+\ell_2\omega_2)(1-\frac{p}{q})}\sum_{R\in m(\Omega)}\|a_R\|_{L^q( X_1\times X_2 )}^p\, \mu(R)^{1-\frac{p}{q}} \, w\Big({\ell(Q_1)\over \ell(\widehat{Q}_1)}\Big)\\
&\lesssim&  
2^{(\ell_1\omega_1+\ell_2\omega_2)(1-\frac{p}{q})} \sum_{R\in m(\Omega)}\|a_R\|_{L^q( X_1\times X_2 )}^p \, \overline{w}\Big({\ell(Q_1)\over \ell(\widehat{Q}_1)}\Big)\, \mu(R)^{1-\frac{p}{q}} \, \overline{w}\Big({\ell(Q_1)\over \ell(\widehat{Q}_1)}\Big).
\end{eqnarray*}
Applying H\"older's inequality with exponent $s=q/p>1$, and then using property (iii-b), with $\delta=q/(2p)>0$,  from the atoms and the Journ\'e-type covering lemma implies
\begin{eqnarray*}
B_{11} &\lesssim& 
2^{(\ell_1\omega_1+\ell_2\omega_2)(1-\frac{p}{q})}\bigg( \sum_{R\in m(\Omega)} \|a_R\|_{L^q( X_1\times X_2 )}^q\ \widetilde{\widetilde{w}}\Big({\ell(Q_1)\over \ell(\widehat{Q}_1)}\Big)\bigg)^{\frac{p}{q}}\bigg( \sum_{R\in m(\Omega)} \mu(R) \, \widetilde{w}\Big({\ell(Q_1)\over \ell(\widehat{Q}_1)}\Big) \bigg)^{1-\frac{p}{q}}\\
&\lesssim&  
2^{(\ell_1\omega_1+\ell_2\omega_2)(1-\frac{p}{q})}  \big ( (1+\ell_1\omega_1+\ell_2\omega_2)2^{\ell_1\omega_1+\ell_2\omega_2} \mu(\widetilde{\Omega})\big )^{\frac{p}{q}-1}     
\mu(\Omega)^{1-\frac{p}{q}} \; \lesssim \;  1.
\end{eqnarray*}
The last inequality because $\frac{p}{q}-1 <0$ so $(1+\ell_1\omega_1+\ell_2\omega_2)^{\frac{p}{q} -1}<1$.
The constants involved in the similarities depend only on $p$ and $q$, the geometric constants of the spaces directly (quasi-triangle constants $A_0^{(i)}$, doubling constants, and upper dimensions $\omega_i$ for $i=1,2$) or indirectly via   the absolute constants $C_i$ appearing in the definition of the $(p,q)$-atoms, or the constants appearing in the Journ\'e Lemma, or the dilation constants of the underlying dyadic grids or their ratios, themselves depending only on the geometric constants.

\subsubsection*{Estimate for $B_{12}$}
Using the cancellation condition of the atoms $a_{R}$, we write $B_{12}$ as
\begin{eqnarray*}
B_{12}&=&\sum_{R\in m(\Omega)}\int_{x_1\not\in100\overline{C}_1\widehat{Q}_1}\int_{x_2\not\in 100\overline{C}_2Q_2}\bigg| \sum_{k_1=-\infty}^{\widehat{k}_1}\sum_{k_2=-\infty}^{\widehat{k}_2}\Big|\int_{ X_1\times X_2 }[\psi_{\alpha_1}^{k_1}(y_1)-\psi_{\alpha_1}^{k_1}(z_1)]\\
&&\hskip -.3in\times [\psi_{\alpha_1}^{k_1}(y_2)-\psi_{\alpha_1}^{k_1}(z_2)]
\, a_{R}(y_1,y_2) \, d\mu_1(y_1) \,  d\mu_2(y_2) {\chi_{Q_{\alpha_1}^{k_1}}(x_1)\over \mu_1(Q_{\alpha_1}^{k_1})}  {\chi_{Q_{\alpha_2}^{k_2}}(x_2)\over \mu_2(Q_{\alpha_2}^{k_2})} \Big|^q \bigg|^{\frac{p}{q}}  d\mu_1(x_1) \, d\mu_2(x_2).
\end{eqnarray*}
Here the constants $\widehat{k}_1$ and $\widehat{k}_2$ satisfy $\delta_1^{\widehat{k}_1}\approx \ell(\widehat{Q}_1)$ and $\delta_2^{\widehat{k}_2}\approx \ell(Q_2)$, respectively. Applying the smoothness properties of $\psi_{\alpha_1}^{k_1}(x_1,y_1)$ and $\psi_{\alpha_1}^{k_1}(x_2,y_2)$ yields that $B_{12}$ satisfies the same estimate as $B_{11}$ does as in \eqref{eqn:B11}. 
This concludes the proof of Theorem \ref{theorem-Hp atom decomp}.
\end{proof}



\bigskip

\noindent Department of Mathematics, Auburn University, AL
36849-5310, USA.

\noindent {\it E-mail address}: \texttt{hanyong@auburn.edu}

\medskip

\noindent Department of Mathematics, Macquarie University, NSW
2019, Australia.

\noindent {\it E-mail address}: \texttt{ji.li@mq.edu.au}

\medskip
\noindent Department of Mathematics and Statistics,
         University of New Mexico,
         Albuquerque, NM 87131, USA

\noindent {\it E-mail address}: \texttt{crisp@math.unm.edu}

\medskip

\noindent School of Information Technology and Mathematical
Sciences, University of South Australia, Mawson Lakes SA 5095,
Australia.

\noindent {\it E-mail address}:
\texttt{lesley.ward@unisa.edu.au}

\end{document}